\documentclass[reqno]{amsart}
\usepackage{amsfonts}
\usepackage{amsmath}
\usepackage{amsthm, amscd}
\usepackage{mathtools}
\usepackage{amssymb}
\usepackage{mathrsfs}

\usepackage[alphabetic]{amsrefs}
\usepackage{etex}
\usepackage{hyperref}
\hypersetup{
  colorlinks = true,
  linkcolor  = black
}

\usepackage{tikz}
\usepackage{xypic}

\allowdisplaybreaks

\theoremstyle{plain}\newtheorem{Theorem}{Theorem}[section]
\theoremstyle{plain}\newtheorem{Corollary}[Theorem]{Corollary}
\theoremstyle{plain}\newtheorem{Lemma}[Theorem]{Lemma}
\theoremstyle{plain}\newtheorem{Definition}[Theorem]{Definition}
\theoremstyle{plain}\newtheorem{Proposition}[Theorem]{Proposition}
\theoremstyle{plain}
\theoremstyle{plain}\newtheorem*{Theorem*}{Theorem}
\theoremstyle{remark}\newtheorem{remark}[Theorem]{Remark}

\DeclareMathOperator{\parr}{par}
\DeclareMathOperator{\ch}{ch}
\DeclareMathOperator{\TD}{TD}
\DeclareMathOperator{\rank}{rank}
\DeclareMathOperator{\II}{I}
\DeclareMathOperator{\ad}{ad}

\DeclareMathOperator{\SHI}{SHI}
\DeclareMathOperator{\AHI}{AHI}
\DeclareMathOperator{\AKh}{AKh}
\DeclareMathOperator{\muu}{\mu^{orb}}

\DeclareMathOperator{\Sp}{Sp}
\DeclareMathOperator{\SU}{SU}
\DeclareMathOperator{\SO}{SO}
\DeclareMathOperator{\Ad}{Ad}
\DeclareMathOperator{\spann}{span}
\DeclareMathOperator{\Hom}{Hom}
\DeclareMathOperator{\pt}{pt}
\DeclareMathOperator{\id}{id}

\newcommand{\bA}{\mathbb{A}}
\newcommand{\bC}{\mathbb{C}}
\newcommand{\bE}{\mathbb{E}}
\newcommand{\bK}{\mathbb{K}}
\newcommand{\bL}{\mathbb{L}}
\newcommand{\bQ}{\mathbb{Q}}
\newcommand{\bR}{\mathbb{R}}
\newcommand{\bU}{\mathbb{U}}
\newcommand{\bV}{\mathbb{V}}
\newcommand{\bW}{\mathbb{W}}
\newcommand{\bZ}{\mathbb{Z}}

\author{Yi Xie, Boyu Zhang}
\title{Instanton Floer homology for sutured manifolds with tangles}

\begin{document}

\begin{abstract}
We prove an excision theorem for the singular instanton Floer homology introduced in \cite{KM:YAFT,KM:Kh-unknot} that allows the excision surfaces to intersect the singular locus. This is an extension of the non-singular excision theorem by Kronheimer and Mrowka \cite[Theorem 7.7]{KM:suture} and the genus-zero singular excision theorem by Street \cite{Street}. We use the singular excision theorem to define an instanton Floer homology theory for sutured manifolds with tangles. As applications, we prove that the annular Khovanov homology introduced in \cite{APS} (1) detects the unlink, (2) detects the closure of the trivial braid, and (3) 
distinguishes braid closures 
from other links; we also prove that
 the annular instanton Floer homology introduced in \cite{AHI} detects the Thurston norm of meridional surfaces. \end{abstract}

\maketitle

\setcounter{tocdepth}{1}
\tableofcontents

\section{Introduction}
\label{sec_intro}

Heegaard Floer homology for 
sutured 3-manifolds is introduced by Juh\'asz in \cite{Juhasz-holo-disk,Juh:sut}. 
Motivated by Juh\'asz's work, Kronheimer and
Mrowka defined the instanton Floer homology and the monopole Floer homology for sutured 3-manifolds in \cite{KM:suture}. Another invariant introduced by Kronheimer and Mrowka
is the singular instanton Floer homology for 3-manifolds with links \cite{KM:YAFT,KM:Kh-unknot}. 
These invariants have become very important tools in the study of knots and links in 3-manifolds. 
In this article, we study the interaction between the sutured instanton Floer homology and the singular instanton Floer homology.  
More precisely, we use the singular instanton Floer homology to define invariants for sutured 3-manifolds with tangles. 

A crucial ingredient in the definition of sutured instanton Floer homology is the excision theorem \cite[Theorem 7.7]{KM:suture},
which is a generalization of Floer's original torus excision theorem \cite{floer1990instanton,Floer-memorial,braam1995floer}. 
Roughly speaking, Floer's excision theorem states that if one cuts a 3-manifold along two tori and re-glue it in a different way, then
the instanton Floer homology of the resulting 3-manifold  is isomorphic to the original one. Kronheimer and Mrowka \cite{KM:suture} generalized this theorem to arbitrary positive genera.

The proof of Kronheimer and Mrowka's excision theorem relies on a computation by Mu\~noz \cite{Munoz} for the instanton Floer homology ring of $S^1\times\Sigma$ where $\Sigma$ is a closed surface. In fact, the excision theorem follows from Mu\~noz's computation and the formal properties of the instanton Floer homology.  Mu\~noz's computation, on the other hand, is based on an explicit expression of the cohomology ring structure of the moduli space of stable bundles with rank $2$ over $\Sigma$, which was studied in \cite{Bara ,Zagier,STian, KingNewstead}.

We extend Kronheimer and Mrowka's excision theorem to singular instanton Floer homology. The singular instanton Floer homology is defined for a link $L$ in a 3-manifold $Y$. Kronheimer and Mrowka's excision theorem is still valid for the singular instanton Floer homology as long as the excision surfaces are disjoint from $L$, see Theorem \ref{thm_nonsingularExcision} below for the precise statement. Our result allows the excision surfaces to intersect the link $L$, under the condition that the number of intersection points of an excision surface with $L$ is odd and at least three. Theorem \ref{thm_excision} and Theorem \ref{thm_excisionAlongOneSurface} below show that in this case, the singular instanton Floer homology of the resulting $3$-manifold is isomorphic to the original one in a suitable sense. 
 When the genus of the excision surfaces is $0$, the
singular excision theorem
was proved by Street in \cite{Street}. 

Similar to the non-singular case, the proof of the singular excision theorem can be reduced to the computation of a singular instanton Floer homology ring on $S^1\times\Sigma$. 
Street \cite{Street} adapted the argument of Mu\~noz \cite{Munoz} and established a relation between this Floer homology ring and the cohomology ring of the moduli space of stable parabolic bundles over $\Sigma$ with rank $2$. However, the cohomology ring of this moduli space seems to be more complicated than its non-singular counterpart, and a direct generalization of Mu\~noz's computation seems difficult. 
In the genus-zero case, Street \cite{Street} computed the cohomology ring  of the moduli space using a volume formula given by \cite{jeffrey1994toric}. The algebra in \cite{Street} depends heavily on the genus assumption and seems difficult to generalize to the higher genus case. 

Our proof applies a different approach. Instead of computing the complete ring structure of the cohomology of the moduli space, we use an idea of Mumford (cf. \cite{Zagier}) to find one particular relation of the canonical generators using the vanishing of Chern classes beyond the rank of an index bundle over the moduli space. At the same time, we study the property of the singular instanton Floer homology ring by computing several cobordism maps. We then show that the above considerations already yield enough algebraic information to prove the singular excision theorem.

The singular excision theorem allows us to define the instanton Floer homology for a balanced sutured manifold with a balanced tangle following the strategy of Kronheimer and Mrowka \cite{KM:suture}.
The definition of a balanced sutured manifold was introduced by Juh\`asz in \cite{Juhasz-holo-disk}, which is a variation of Gabai's original 
definition of sutured manifold in \cite{G:Sut-1}. By definition, a balanced sutured manifold consists of a 3-manifold $M$ and a collection of circles $\gamma\subset \partial M$,
such that there is a decomposition
\begin{equation}\label{partial-M-decom'}
\partial M=A(\gamma)\cup R^+(\gamma)\cup R^-(\gamma),
\end{equation}
which satisfies extra conditions.
The precise definition will be reviewed in Definition \ref{Def-sutured-mfd}. 
Scharlemann \cite{Schar} studied a sutured manifold $(M,\gamma)$
with a properly embedded 1-complex $T$, and generalized   
Gabai's definition of tautness for sutured manifolds \cite{G:Sut-1} to the triple $(M,\gamma,T)$. 
The definition of tautness for $(M,\gamma,T)$ will be reviewed in Definition \ref{Def-taut}. We will only consider the case when $T$ is a 1-manifold. In this case, we make the following definition.
\begin{Definition}\label{Def-tangle'}
Let $(M,\gamma)$ be a balanced sutured manifold with the decomposition \eqref{partial-M-decom'} as above. 
A \emph{tangle} $T$ in $M$ is a properly embedded 1-manifold with $\partial T\subset R^+(\gamma)\cup R^-(\gamma)$. 
The tangle $T$ in $M$ is called \emph{balanced} if $|T\cap R^+(\gamma)|=|T\cap R^-(\gamma)|$.
The tangle $T$ in $M$ is called \emph{vertical} if  every connected component of $T$
intersects both $R^+(\gamma)$ and $R^-(\gamma)$. 
\end{Definition}

If $T$ is a balanced tangle on a balanced sutured manifold $(M,\gamma)$,
we use the singular excision theorem to define an instanton Floer homology group $\SHI(M,\gamma,T)$.
When $T$ is empty, our definition coincides with Kronheimer and Mrowka's definition of sutured instanton Floer homology in \cite{KM:suture}.
Moreover, we prove that the group $\SHI$ satisfies the following properties. 
\begin{Theorem}\label{SHI-properties}
Let $(M,\gamma,T)$ be a balanced sutured manifold with a balanced tangle.
\begin{itemize}
\item[(a)] If $(M,\gamma, T)$ is taut, then $\SHI(M,\gamma, T)\neq 0$.

\item[(b)]  If  $T$ is vertical,  $(M,\gamma)$ is a homology product, and
\begin{equation*}
\SHI(M,\gamma,T)\cong \mathbb{C},
\end{equation*}
then the triple $(M,\gamma,T)$ is diffeomorphic to a product sutured manifold with a product tangle, i.e. 
$$(M,T)\cong ([-1,1]\times F,[-1,1]\times \{p_1,\cdots,p_n\}),$$ 
where $F$ is a compact surface, $p_1,\cdots,p_n$ are distinct points on $F$,
and $R^\pm(\gamma)\subset \partial M$ are given by $\{\pm 1\}\times F$.
\end{itemize}
\end{Theorem}
The results above are analogous to the non-vanishing theorem \cite[
Theorem 7.12]{KM:suture} and the product detection theorem \cite[
Theorem 7.18]{KM:suture} for
sutured instanton Floer homology.

Under the additional assumption that $R^\pm(\gamma)$ can be embedded
in $S^2$, an equivalent version of $\SHI(M,\gamma,T)$ was already defined in \cite{Street}.  Part (b) of Theorem 
\ref{SHI-properties} was also proved under this assumption in \cite{Street}. 
\\

The instanton Floer homology for sutured manifolds with tangles has interesting applications in the theory of annular links. By definition, 
an \emph{annular link} is a link contained in $A\times [0,1]$, where $A$ is an annulus. An annular link $L$ is called an \emph{unlink} if it 
bounds a disjoint union of disks. 
Let $\mathcal{U}_n$ be the unlink with $n$ components. 
Let $\mathcal{K}_n$ be the closure of the trivial braid with $n$ strands. 

Khovanov \cite{Kh-Jones} defined a bi-graded homology group for an oriented link in the $3$-ball which is a categorification of the Jones polynomial. Asaeda, Przytycki, and Sikora
\cite{APS} introduced a generalization of Khovanov homology  to oriented links contained in $F\times [0,1]$ for compact surfaces $F$.
In particular, if $F$ is an annulus, then the invariant defined by \cite{APS} is an invariant for oriented annular links, which is called the \emph{annular Khovanov homology}. 
For more details, the reader may also refer to \cite[Section 2]{Rob}.
The annular Khovanov homology is equipped with three gradings called the h-grading, q-grading, and f-grading. (The f-grading is called the ``Alexander grading'' in \cite{Rob}.)
For an oriented annular link $L$,
we use $\AKh(L)$ to denote the annular Khovanov homology of $L$, and use $\AKh(L,i)$ to denote the component of $\AKh(L)$ with f-grading $i$. Changing the orientations of components of $L$ does not change the isomorphism classes of $\AKh(L,i)$.

As an application of the instanton Floer homology for sutured manifolds with tangles, we will prove the following 
theorem.
\begin{Theorem}
\label{thm_annular_intro}
Let $L$ be an oriented annular link, we have the following results:
\begin{itemize}
\item[(a)] $L$ is included in a 3-ball in $A\times [0,1]$ if and only if $\AKh(L;\bQ)$ is supported at the f-grading 0.
\item[(b)] $L$ is isotopic to the closure of a braid  with $n$ strands
if and only if the top f-grading of $\AKh(L;\bQ)$ is $n$, and $\AKh(L,n;\bQ)\cong \bQ.$ 
\end{itemize}
\end{Theorem}

Under the additional assumption that all the components of $L$ are null-homologous, Part (a) of Theorem \ref{thm_annular_intro} was 
proved in \cite{AHI}. Part (b) of the theorem generalizes \cite[Corollary 1.2]{GN-braid}, where the original statement assumes in addition 
 that there exists a disk with meridian boundary that intersects $L$
transversely at $n$ points. Using the unlink detection theorem of Khovanov homology in $S^3$ by Baston-Seed \cite{Kh-unlink} and Hedden-Ni \cite{HN-unlink}, and 
a theorem that detects the trivial braid among braids by Baldwin-Grigsby \cite{BG-braid}, we immediately have the following corollary.

\begin{Corollary}
\label{cor_annular_intro}
Let $L$ be an oriented annular link, we have
\begin{itemize}

\item[(a)] Suppose $L$ has $n$ components, then $L$ is isotopic to  $\mathcal{U}_n$ if and only if 
     $$\AKh(L;\bZ/2)\cong\AKh(\mathcal{U}_n;\bZ/2)$$ 
     as triply-graded abelian groups.
\item[(b)] $L$ is isotopic to $\mathcal{K}_n$ if and only if $\AKh(L;\bZ/2)\cong \AKh(\mathcal{K}_n;\bZ/2)$ as triply-graded abelian groups.
\end{itemize}
\end{Corollary}

The proof of Theorem \ref{thm_annular_intro} relies on the study of the annular instanton Floer homology introduced in \cite{AHI}. The annular instanton Floer homology $\AHI(L)$ is a gauge-theoretic invariant for the annular link $L$. 
In $\bC$-coefficients, the group $\AHI(L;\bC)$ is equipped with a $\mathbb{Z}$-grading which is also called  the f-grading. Given an annular link $L$, we use 
$\AHI(L,i)$ to denote the component of $\AHI(L;\bC)$ with f-degree $i$. 
 Before discussing the properties of $\AHI(L,i)$, we need to introduce the following definition. 
\begin{Definition}
A properly embedded, connected, oriented 
surface $S\subset A\times[0,1]$ is called a \emph{meridional surface} if $\partial S$ is a meridian of $A\times[0,1]\cong S^1\times D^2$.
\end{Definition}
We prove the following result using the instanton Floer homology for sutured manifolds with tangles.
\begin{Theorem}\label{2g+n_intro}
Given an annular link $L$, suppose $S$ is a meridional surface that intersects $L$ transversely. Let $g$ be the genus of $S$, and let $n:=|S\cap L|$. Suppose $S$ minimizes the value of $2g+n$ among meridional surfaces,
then we have
\begin{equation*}
\AHI(L,i)= 0
\end{equation*}
when $|i|> 2g+n$, and
\begin{equation*}
\AHI(L,\pm(2g+n))\neq 0.
\end{equation*}
\end{Theorem}
When all the components of $L$ are null-homologous, the above theorem is equivalent to  \cite[Theorem 1.5]{AHI}.  

In \cite{AHI}, a spectral sequence relating the annular Khovanov homology and the annular instanton Floer homology is constructed. Moreover,
this spectral sequence respects the f-gradings of the two homology groups. In Section \ref{sec_applications}, we will prove 
Theorem \ref{thm_annular_intro} as a corollary of Theorem \ref{2g+n_intro} using this spectral sequence. 

\subsection*{Acknowledgements}
This project originated from a question posed to us by Peter Kronheimer. We would like to express our sincere gratitude for his encouragement.

\section{Preliminaries}
\label{sec_review}
\subsection{Singular instanton Floer homology}\label{review-instanton}
This subsection gives a brief review of the singular instanton Floer homology theory developed by Kronheimer and Mrowka 
\cite{KM:Kh-unknot,KM:YAFT}. Let $(Y,L,\omega)$ be a triple where
\begin{itemize}
\item $Y$ is a closed oriented 3-manifold,

\item $L\subset Y$ is a link,
 
\item  $\omega\subset Y$ is an embedded 1-manifold such that $ \partial \omega=\omega\cap L$.
\end{itemize}
There is a unique orbifold structure on $Y$ such that $L$ is its singular locus and the local stabilizer group at every point in $L$ is $\mathbb{Z}/2$. The 1-manifold $\omega$ determines a \emph{singular bundle data} $\check{P}$ on $(Y,L)$. (The reader may refer to \cite[Definition 2.1]{KM:Kh-unknot} for the definition of singular bundle data, and \cite[Section 4.2]{KM:Kh-unknot} for the construction of $\check{P}$ from $\omega$.) The singular bundle data $\check{P}$ is, roughly speaking, an $\SO(3)$-bundle over the orbifold determined by $(Y,L)$ whose second Stiefel-Whitney class is dual to $[\omega]$.

A closed embedded surface $\Sigma\subset Y$ is called a \emph{non-integral surface} of $(Y,L,\omega)$ if at least one of the following holds:
\begin{itemize}
  \item $\Sigma$ is disjoint from $L$ and the intersection number of $\omega$ and $\Sigma$ is odd,
  \item  The intersection number of $L$ and $\Sigma$ is odd. 
\end{itemize}
The triple $(Y, L,\omega)$ is called \emph{admissible} if every connected component of $Y$ contains a non-integral surface. For such a triple, Kronheimer and Mrowka \cite[Definition 3.7]{KM:Kh-unknot} defined the singular instanton Floer homology group $\II(Y,L,\omega)$.

We give a brief review of the definition of $\II(Y,L,\omega)$, the reader may refer to \cite{KM:YAFT, KM:Kh-unknot} for more details.
Let $\mathcal{A}$ be the space of orbifold connections on the singular bundle data $\check{P}$. Since the adjoint action $\Ad$ on $\SO(3)$ lifts to an action $\widetilde{\Ad}$ of $\SO(3)$ on $\SU(2)$,  the bundle $\check{P}\times_{\Ad}\SO(3)$ has a double cover defined by $\check{P}\times_{\widetilde{Ad}}\SU(2)$. Let $\mathcal{G}$ be the group consisting of
gauge transformations that lift to $\check{P}\times_{\widetilde{Ad}}\SU(2)$. 
Let $\mathcal{B}(Y,L,\omega):=\mathcal{A}/\mathcal{G}$ 
be the quotient of $\mathcal{A}$ by the action of $\mathcal{G}$. We will use $\mathcal{B}$ for $\mathcal{B}(Y,L,\omega)$ when the triple $(Y,L,\omega)$ is clear from the context.
The instanton Floer homology 
$\II(Y,L,\omega)$
is defined
to be a Morse homology of the Chern-Simons functional on $\mathcal{B}$. 
Since the Chern-Simons functional is not always Morse, a perturbation is needed. The set of critical points of the \emph{unperturbed} Chern-Simons functional is given by 
flat $\SU(2)$ connections on $Y\backslash(L\cup \omega)$ such that the holonomies around $\omega$ are $-1$ and 
the holonomies around $L$ have trace zero. Notice that these flat connections have structure group $\SU(2)$ instead of $\SO(3)$ because of the special definition of $\mathcal{G}$. The group $\II(Y,L,\omega)$ is equipped with a relative $\mathbb{Z}/4$ grading. In this paper we will always use the complex number field $\bC$ as the coefficients for instanton Floer homology groups. 
In this case, if $Y$ is disconnected, then $\II(Y,L,\omega)$ is the tensor product of the singular instanton Floer homology groups of the connected components of $Y$.
 
Let $M$ be a closed oriented submanifold of $Y-L$, then there is an operator $\mu(M)$ of degree $(4-\dim M)$ on
$\II(Y,L,\omega)$. The operator $\mu(M)$ is defined by evaluating the class $-\frac{1}{4} p_1(\mathbb{P})/[M]$
on the moduli spaces of trajectories of the Chern-Simons functional, where $\mathbb{P}$ is the universal $SO(3)$-bundle over
$\mathcal{B}\times (Y\setminus L)$. By definition, $\mu(M)$ only depends on the homology class of $M$. 
If $M,N\subset Y-L$ are two closed oriented submanifolds, then $
\mu(M)\mu(N)=(-1)^{\dim M\dim N}\mu(N)\mu(M)$.
The operator $\mu$ was introduced in \cite{donaldson1990polynomial, D-Floer}, and it has become a ubiquitous construction in gauge theory. For more details in the context of singular instanton Floer homology, the reader may refer to \cite[Section 2.3.2]{Street}.
Our convention of the constant $-1/4$ follows \cite{KM:suture, DK}.

Although $\check{P}|_{Y-L}$ does not always extend to an ordinary $\SO(3)$-bundle on $Y$, it 
always extends locally near each point $p\in L$.  
In fact, 
 at each point $p\in L$ there are two ways to extend $\check{P}|_{Y-L}$ to a neighborhood of $p$, and the choices of the extensions give us a double cover
 $L_\Delta\to L$ whose first Stiefel-Whitney class equals $P.D.[\partial \omega]\in H^1(L,\mathbb{Z}/2)$.
 Therefore $\check{P}|_{Y-L}$
 can be extended to an $\SO(3)$-bundle on $Y-\partial \omega$. 
 Given an embedded manifold $M\subset Y$ with $M\cap \partial \omega=\emptyset$,
let $P$ be a choice of the extension of $\check{P}|_{M-L}$ to ${M}$. We can extend the universal bundle $\mathbb{P}$
to $\mathcal{B}\times M$ according to $P$ and define an operator $\mu(M)$ with respect to $P$.

Let $p\in L$, there is an operator $\sigma_p$ on $\II(Y,L,\omega)$ defined in a similar way to $\mu(M)$. We sketch the definition of $\sigma_p$; for more details, the reader may refer to \cite[Section 2.2]{Kr-ob} and \cite[Section 2.3]{Street}. let $U$ be an open neighborhood of $p$ in $Y$, and let $P$ be an extension of $\check{P}|_{U-L}$ to $U$.
Recall that $\mathcal{A}$ is the space of orbifold connections, hence by definition, every connection in $\mathcal{A}$ has an asymptotic holonomy around $L$ conjugate to 
$$\begin{pmatrix}1 & 0 & 0\\ 0 & -1 & 0 \\ 0 & 0 & -1\end{pmatrix}\in \SO(3).$$
The bundle $P$ defines an extension of $\mathbb{P}|_{\mathcal{B}\times(U-L)}$ to $\mathbb{P}|_{\mathcal{B}\times U}$, we will abuse notation and denote the extended bundle by $\mathbb{P}$.
Let $\mathbb{S}$ be the associated $\bR^3$-bundle of $\mathbb{P}$, for every $[A]\in\mathcal{B}$,
the vector space $\mathbb{S}|_{[A]\times p}$ is decomposed into the eigenspaces of
the asymptotic holonomy of $[A]$. It can be proved that $\mathbb{S}|_{\mathcal{B}\times \{p\}}=\underline{\mathbb{R}}\oplus\mathbb{K}_p$, where $\underline{\mathbb{R}}$ is a trivial real line bundle corresponding to the eigenvalue $1$, and $\mathbb{K}_p$ is an orientable plane bundle corresponding to the eigenvalue $-1$.
A choice of orientation of $L$ at $p$ determines an orientation of $\mathbb{K}_p$ \cite[Section 2(iv)]{KM-surface1}.
The operator $\sigma_p$ is then defined by evaluating the class  
$\delta_p:=-\frac{1}{2} e(\mathbb{K}_p)$ on the moduli spaces of trajectories of the Chern-Simons functional. 
If we change the extension $P$, then the orientation of $\mathbb{K}_p$ will also change \cite[Section 2(iv)]{KM-surface1}.
Hence the sign of $\sigma_p$ depends on the choice
of $P$ and the orientation of $L$ at $p$.

If $Y$ is connected, the following formula is a straightforward consequence of \cite[Proposition 4.1]{Kr-ob} once all the notations are translated:
\begin{equation}\label{delta^2+beta}
\sigma_p^2+\mu(\pt)=2\id.
\end{equation}
We use the above formula to derive the following useful result.
\begin{Proposition}\label{mu-pt=2}
Suppose $(Y,L,\omega)$ is a connected admissible triple and there is a connected component $L_0$ of $L$ 
such that $|\partial \omega \cap L_0|$ is odd. Then we have
$$
\mu(\pt)=2\id
$$
on $\II(Y,L,\omega)$.
\end{Proposition} 
\begin{proof}
Since $|\partial \omega \cap L_0|$ is odd, the double cover $L_{0,\Delta}\to L_0$ 
parameterizing the extensions of $\check{P}$ on $L_0$ is non-trivial. 
Pick $p\in L_0$ and an orientation of $L_0$ at $p$. 
The loop on $L_0$ with base point $p$ lifts to an arc on $L_{0,\Delta}$ with two different end points.
 This means the bundle $\mathbb{K}_p$
can be deformed into its inverse and hence $\delta_p=-\delta_p$. Therefore $\sigma_p=0$, and \eqref{delta^2+beta} implies
$\mu(\pt)=2\id$.
\end{proof}

Let
\begin{equation*}
(W,S,\omega):(Y_0,L_0,\omega_0)\to (Y_1,L_1,\omega_1)
\end{equation*}
be a cobordism between admissible triples, \cite[Section 3.7]{KM:YAFT} defined a map
\begin{equation}\label{I-functor}
\II(W,S,\omega): \II(Y_0,L_0,\omega_0)\to \II(Y_1,L_1,\omega_1).
\end{equation} 
This map is only well-defined up to an overall sign. 
The singular instanton Floer homology is a functor from the cobordism category
of admissible triples to the category of vector spaces modulo $\pm 1$. If $W$ is an almost complex manifold and $S$ is an almost complex submanifold, a canonical sign of $\II(W,S,\omega)$ can be chosen which is compatible with compositions of cobordisms.

\subsection{Nonsingular excisions}
\label{subsection_nonsingularEx}
The following two theorems were originally stated for $L=\emptyset$, but the proofs apply verbatim to the general case. 

\begin{Theorem}[{\cite[Corollary 7.2]{KM:suture}}]
\label{thm_nonsingularRangeOfEigenvalues}
Let $(Y,L,\omega)$ be a triple, let $\Sigma\subset Y-L$ be a connected non-integral surface of $(Y,L,\omega)$ with genus $g\ge 1$.  Then the simultaneous eigenvalues of
the actions $(\mu(\Sigma),\mu(\pt))$ on $\II(Y,L,\omega)$ is a subset of
$$
\{(i^r(2k),(-1)^r2)|k=0,\cdots,g-1,\,\, r=0,1,2,3\}.
$$
\end{Theorem}

For $Y,L,\omega,\Sigma$ as in Theorem \ref{thm_nonsingularRangeOfEigenvalues}, define $\II(Y,L,\omega|\Sigma)$ to be the simultaneous generalized eigenspace of the operators $(\mu(\Sigma),\mu(\pt))$ for the eigenvalues $(2g-2,2)$. If $\Sigma$ has genus one,  then by Theorem \ref{thm_nonsingularRangeOfEigenvalues} the only eigenvalue of $\mu(\Sigma)$ is zero, therefore $\II(Y,L,\omega|\Sigma)$ is equal to the generalized eigenspace of $\mu(\pt)$ for the eigenvalue $2$.

More generally, if $\Sigma\subset Y-L$ is disconnected and every connected component of $\Sigma$ is a non-integral surface, define $\II(Y,L,\omega|\Sigma)$ to be the intersection of $\II(Y,L,\omega|\Sigma_i)$ where $\{\Sigma_i\}$ are the connected components of $\Sigma$.

\begin{Theorem}[{\cite[Theorem 7.7]{KM:suture}}]
\label{thm_nonsingularExcision}
Let $\Sigma_1,\Sigma_2\subset Y-L$ be two disjoint connected closed surfaces such that they have the same positive genus,  and suppose $\Sigma_1$ and $\Sigma_2$ intersect $\omega$ transversely at the same odd number of points. Let $\varphi:\Sigma_1\to\Sigma_2$ be a diffeomorphism that maps $\Sigma_1\cap \omega$ to $\Sigma_2\cap \omega$. Let $(\widetilde{Y}, \widetilde{L}, \tilde{\omega})$ be the resulting triple after cutting $Y$ open along $\Sigma=\Sigma_1\cup \Sigma_2$ and glue the boundary given by $\Sigma_1$ to the boundary given by  $\Sigma_2$ by the map $\varphi$, let $\widetilde{\Sigma}\subset\widetilde{Y}$ be the image of $\Sigma$ after the excision. Then
$$
\II(Y,L,\omega|\Sigma) \cong \II(\widetilde{Y},\widetilde{L},\tilde{\omega}| \widetilde{\Sigma}).
$$
\end{Theorem}

The special case of Theorem \ref{thm_nonsingularExcision} when $g=1$ was due to Floer \cite{floer1990instanton, braam1995floer}.

\subsection{Flip symmetry}\label{subsection-flip}
Let $(Y,L,\omega)$ be an admissible triple, and suppose $L_1$ is a union of connected components of $L$ such that
$\omega\cap L_1=\emptyset$. Let $L_1'$ be a parallel copy of $L_1$, then the triple $(Y,L,\omega+L_1')$ also satisfies the non-integral condition. Let $\check{P}$ and $\check{P}'$ be the singular bundle data for $(Y,L,\omega)$ and $(Y,L,\omega+L_1)$. There is an isomorphism from $\check{P}|_{Y-L}$ to $\check{P}'|_{Y-L}$ which induces a diffeomorphism
\begin{equation}\label{flip-sym}
\tau : \mathcal{B}(Y,L,\omega)\to  \mathcal{B}(Y,L,\omega+L_1).
\end{equation}
The map $\tau$ is called the \emph{flip symmetry}, and it was first introduced in \cite[Section 2(iv)]{KM-surface1}; see also \cite[Section 4.2]{Kr-ob}. The map $\tau$ induces diffeomorphisms on the critical sets and the moduli spaces of trajectories of the (perturbed) Chern-Simons functional, therefore the map $\tau$ induces an isomorphism on Floer homology, 
\begin{equation}\label{flip-Floer}
\tau:\II(Y,L,\omega)\stackrel{\cong}{\longrightarrow}  \II(Y,L,\omega+L_1),
\end{equation}
which we also denote by $\tau$ by abusing notation.

If we further assume that $0=[L_1]\in H_1(Y,\mathbb{Z}/2)$, then the singular bundle data $\check{P}$ and $\check{P}'$
are the same, and $\tau$ is an involution on $\mathcal{B}(Y,L,\omega)$.  
Fix a choice of orientation on $L$ and take $p\in L_1$, recall that $\delta_p=-\frac12 e(\mathbb{K}_p)$. We have
\begin{equation}\label{delta-flip-sym}
\tau^\ast (\delta_p)=-\delta_p.
\end{equation} 
Let $\Sigma\subset Y$ be a closed oriented embedded surface that intersects $L_1$ transversely, the following formula was proved in \cite[Section 4.2]{Kr-ob}; for a reference closer to our notation, see \cite[Lemma 1.5.14]{Street}:
\begin{equation}\label{mu-flip-sym-homology}
\tau^\ast \Big( -\frac14 p_1(\mathbb{P})/[\Sigma]\Big)=-\frac14 p_1(\mathbb{P})/[\Sigma]+\sum_{p\in \Sigma\cap L_1} \text{sign}_p(\Sigma,L_1)   \delta_p.
\end{equation} 
As a consequence,
\begin{equation}\label{delta-flip-sym}
\tau^\ast (\sigma_p)=-\sigma_p,
\end{equation} 
\begin{equation}\label{mu-flip-sym}
\tau^\ast (\mu(\Sigma))=\mu(\Sigma)+\sum_{p\in \Sigma\cap L_1} \text{sign}_p(\Sigma,L_1)  \sigma_p .
\end{equation} 
Therefore, the cohomology class
$$
-\frac14 p_1(\mathbb{P})/[\Sigma] + \frac12 \sum_{p\in \Sigma\cap L_1} \text{sign}_p(\Sigma,L_1)   \delta_p
$$
and the operator
\begin{equation}\label{mu-orb}
\muu(\Sigma):=\mu(\Sigma)+\frac{1}{2}\sum_{p\in \Sigma\cap L_1} \text{sign}_p(\Sigma,L_1)  \sigma_p 
\end{equation} 
are both invariant under $\tau$. 

\subsection{The space $\mathbb{V}_{g,n}$}\label{subsection-Vgn}
\label{subsec_Vgn}
This subsection reviews some basic constructions from \cite{Street} that will be useful later.
Let $\Sigma$ be a closed oriented surface with genus $g$,
and let $p_1,\cdots,p_n\in \Sigma$ be $n$ points with $n$ odd, then the triple 
$
(S^1\times \Sigma, S^1\times\{p_1,\cdots,p_n\},\emptyset)
$
is admissible. Fix an orientation on $S^1$, it induces orientations on $S^1\times \Sigma$ and 
$S^1\times\{p_1,\cdots,p_n\}$.  Define
\begin{equation*}
\mathbb{V}_{g,n}:=\II(S^1\times \Sigma, S^1\times\{p_1,\cdots,p_n\},\emptyset).
\end{equation*} 
Let $\omega:=S^1\times \{q\}\subset S^1\times \Sigma$ with $q\in \Sigma-\{p_1,\cdots,p_n\}$.
Apply \eqref{flip-Floer} to $L_1:=S^1\times\{p_1\}$, we have
\begin{equation*}
\mathbb{V}_{g,n}\cong \mathbb{V}_{g,n}':= \II(S^1\times \Sigma, S^1\times\{p_1,\cdots,p_n\},\omega).
\end{equation*} 

Notice that the critical point set of the Chern-Simons functional that defines $\bV_{0,1}$  is empty, therefore $\bV_{0,1}=\bV_{0,1}'=0$. For the rest of this article, we will always assume $(g,n)\neq(0,1)$ when discussing properties of $\bV_{g,n}$.

If $S\subset\{1,\cdots,n\}$ is a set with an even number of elements, then $S^1\times \{p_i|i\in S\}$ is null-homologous in $H_1(S^1\times\Sigma;\bZ/2)$, hence
the flip symmetries along this set give involutions
on $\mathbb{V}_{g,n}$ and $\mathbb{V}_{g,n}'$. We denote both involutions by $\tau_S$.

Let $F$ be a 2-dimensional cobordism, i.e. an oriented compact surface with boundary. Equip $F\times \Sigma$ with the product almost complex structure, then $F\times \{p_1,\cdots,p_n\}$ is an almost complex submanifold, hence the map 
\begin{equation*}
\II(F\times \Sigma, F\times \{p_1,\cdots,p_n\},\emptyset)
\end{equation*}
is defined without sign ambiguity.
In particular, if we choose $F$ to be the pair-of-pants cobordism from two circles to one circle, we obtain a map
\begin{equation*}
\mathbb{V}_{g,n}\otimes \mathbb{V}_{g,n} \to \mathbb{V}_{g,n}.
\end{equation*}
This map defines a commutative multiplication on $\bV_{g,n}$. Let $e\in \mathbb{V}_{g,n}$ be the image of the element $1\in \bC$ under the map
$$\II(D^2\times \Sigma, D^2\times\{p_1,\cdots,p_n\},\emptyset):\bC\cong \II(\emptyset,\emptyset,\emptyset) \to \mathbb{V}_{g,n}.$$
By the functoriality of $\II$, multiplication by $e$ on $\bV_{g,n}$ gives the identity map. Therefore the multiplication on $\bV_{g,n}$ defines a ring structure.

 Let $F'$ be a 2-sphere with two disks removed, and regard it
as a cobordism from two circles to the empty set, we obtain a map
\begin{equation}\label{Vgn-pairing}
\langle\cdot~,\cdot\rangle:\mathbb{V}_{g,n}\otimes \mathbb{V}_{g,n} \to \mathbb{C}\cong \II(\emptyset, \emptyset,\emptyset).
\end{equation}    
This is a bilinear pairing on $\mathbb{V}_{g,n}$. Since
$F'$ becomes the product cobordism after changing the orientation
of one of its boundary, the dual of \eqref{Vgn-pairing} induces an isomorphism  $\mathbb{V}_{g,n}\to \mathbb{V}_{g,n}^\ast$, therefore the pairing defined by \eqref{Vgn-pairing} is non-degenerate.

Let 
\begin{equation*}
E : \mathbb{V}_{g,n}\to  \mathbb{V}_{g,n}
\end{equation*}
be the map induced by
\begin{equation}\label{epsilon-cob}
([0,1]\times S^1\times \Sigma, [0,1]\times S^1\times \{p_1,\cdots,p_n\}, \{\pt\}\times \Sigma),
\end{equation}
where $\pt$ is a point in $(0,1)\times S^1$. It follows from the index formula of moduli spaces of trajectories of the Chern-Simons functional that $E$ is a degree 2 map.
Since $[\{\pt\}\times \Sigma]+[\{\pt\}\times \Sigma]=0\in H_2([0,1]\times S^1\times \Sigma;\bZ/2)$, we have $E^2=\id$.
The space $\bV_{g,n}$ decomposes as $\bV_{g,n}=\bV_{g,n}^+\oplus \bV_{g,n}^-$, where $\bV_{g,n}^\pm$ is the eigenspace of $E$ with the eigenvalue $\pm 1$ respectively. Since $E$ is a degree $2$ isomorphism acting on a $\mathbb{Z}/4$-graded space, we have
\begin{equation}\label{eqn_VgnpmHaveSameDim}
\dim\bV_{g,n}^+ = \dim\bV_{g,n}^-.
\end{equation}
It follows from the functoriality of $\II$ that $(Ex)(y)=(Exy)$ for all $x,y\in\bV_{g,n}$, therefore $\bV_{g,n}^\pm$ are rings themselves and $\bV_{g,n}=\bV_{g,n}^+\oplus\bV_{g,n}^-$ as a direct sum of rings.

We use $R_{g,n}$  to denote the space of flat $\SU(2)$ connections on $\Sigma-\{p_1,\cdots,p_n\}$ whose holonomies around all $p_i$'s
have trace zero modulo gauge transformations. Let $d_i$ be a loop around $p_i$ on $\Sigma$, and let
$\{a_i\}_{1\le i\le 2g}$ be a set of based loops on $\Sigma$ that gives the standard generators of $\pi_1(\Sigma)$, then 
\begin{equation*}
\pi_1(\Sigma-\{p_1,\cdots,p_n\})=\langle a_1,\cdots,a_{2g},d_1,\cdots,d_n \rangle/ 
\langle [a_1,a_{g+1}]\cdots [a_{g},a_{2g}]  d_1\cdots d_n \rangle.
\end{equation*}
The space $R_{g,n}$ can also be described as the space of
$\SU(2)$ representations of $\pi_1(\Sigma-\{p_1,\cdots,p_n\})$ such that the images of $d_i$ have trace zero modulo conjugations. Therefore
\begin{multline}\label{Rgn-BC}
R_{g,n}=\{(B_1,\cdots,B_{2g}, C_1,\cdots, C_n)\in \SU(2)^{2g+n}| 
\\ 
[B_1,B_{g+1}]\cdots [B_{g},B_{2g}]  C_1\cdots C_n=1, C_i^2=-1\}/\sim
\end{multline}  
where $\sim$ denotes the conjugations in $\SU(2)$. Since $n$ is odd, all the representations in $R_{g,n}$
are irreducible, hence $R_{g,n}$ is a smooth compact manifold, and it is straighforward to compute that $\dim R_{g,n} = 6g+2n-6$. Let
\begin{equation*}
G:=\pi_1(S^1\times (\Sigma-\{p_1,\cdots,p_n\}) )\cong \pi_1(S^1)\times \pi_1(\Sigma-\{p_1,\cdots,p_n\}).
\end{equation*}
The critical set of the unperturbed Chern-Simons functional of $\II(S^1\times \Sigma, S^1\times\{p_1,\cdots,p_n\},\emptyset)$ are conjugation classes of $\SU(2)$-representations of $G$, which consists of two copies of $R_{g,n}$:
the first copy consists of representations of $G$ which is trivial on $\pi_1(S^1)$; the second copy consists of 
representations of $G$ which maps the generator of $\pi_1(S^1)$ to $-1\in \SU(2)$. 

Since $\dim R_{g,n}=6g+2n-6$, the Chern-Simons functional is not Morse unless $(g,n)=(0,3)$. However,
the Chern-Simons functional is always Morse-Bott. Therefore the instanton Floer homology $\mathbb{V}_{g,n}$ 
is a sub-quotient of 
$H_\ast(R_{g,n}\sqcup R_{g,n};\bC) $, hence 
\begin{equation}\label{MB-inequality}
\dim \mathbb{V}_{g,n} \le 2\dim H_\ast(R_{g,n};\bC).
\end{equation}

Similarly, we can consider the critical set of the unperturbed Chern-Simons functional which defines $\II(S^1\times \Sigma, S^1\times\{p_1,\cdots,p_n\},\{\pt\}\times \Sigma)=\bV_{g,n}'$. In this case, the critical set is given by 
\begin{multline}\label{Rgn'-BC}
R_{g,n}':=\{(B_1,\cdots,B_{2g}, C_1,\cdots, C_n)\in \SU(2)^{2g+n}| \\
[B_1,B_{g+1}]\cdots [B_{g},B_{2g}]  C_1\cdots C_n=-1, C_i^2=-1\}/\sim.
\end{multline}  
The flip symmetry \eqref{flip-sym} along $S^1\times \{p_i\}$ restricts to a map from $R_{g,n}$ to $R_{g,n}'$, which maps $C_i$ to $-C_i$ and fixes the other coordinates in
\eqref{Rgn-BC} and \eqref{Rgn'-BC}. The flip symmetry along multiple connected components of $S^1\times\{p_1,\cdots,p_n\}$ are described in a similar way.

There are some special submanifolds of $R_{g,n}$ which we will  refer to later. In the notations of \eqref{Rgn-BC}, let $D_{i,j}^+$ be the submanifold of $R_{g,n}$ given by $C_i=C_j$, and let $D_{i,j}^-$ be the submanifold of $R_{g,n}$ given by $C_i=-C_j$. Recall that if $(Y,L,\omega)$ is an admissible triple and $p\in L$, Section \ref{review-instanton} defined a cohomology class $\delta_p$ on $\mathcal{B}(Y,L,\omega)$ with the sign depending on auxiliary choices. In the following proposition, we abuse notation and let $\delta_{p_i}$ be the cohomology class defined by a point on $S^1\times\{p_i\}$.
\begin{Proposition}[{\cite[Proposition 1.4.2, Proposition 1.4.12]{Street}}] \label{delta=Dij}
Given $1\le i\neq j\le n$,
there are $r,s\in \{1,-1\}$ depending on the orientations of $D_{i,j}^\pm$ and the choice of sign of $\delta_{p_i}$, such that
\begin{equation*}
P.D.(\delta_{p_i}|_{R_{g,n}})=r[D_{i,j}^+]+s[D_{i,j}^-].
\end{equation*} 
\end{Proposition}

Let 
\begin{equation}\label{A_{g,n}}
\mathbb{A}_{g,n}:=\mathbb{C}[\alpha,\beta,\psi_1,\cdots,\psi_{2g},\delta_1,\cdots,\delta_n]
\end{equation}
be the free graded commutative $\mathbb{C}$-algebra with
\begin{equation*}
\deg \alpha=2,~\deg \beta=4, ~\deg \psi_i=3,~\deg \delta_j=2.
\end{equation*}
Notice that the $\psi_i$'s are anti-commutative to each other.

Recall that $e\in\bV_{g,n}$ is the ``$1$'' element in the ring structure of $\bV_{g,n}$.
Define a ring homomorphism
\begin{equation}\label{Phi-map}
\Phi:\mathbb{A}_{g,n}[\epsilon]/(\epsilon^2-1)\to \mathbb{V}_{g,n}
\end{equation}
by 
$\Phi(\alpha) :=  \mu(\Sigma)(e) $,
$\Phi(\beta) :=  \mu(\pt)(e) $,
$\Phi(\psi_i) := \mu(a_i)(e)$,
$\Phi(\delta_i) := \sigma_{p_i}(e)$,
$\Phi(\epsilon) := E(e) $.
For $S\subset \{1,\cdots,n\}$, there is an endomorphism on $\mathbb{A}_{g,n}[\epsilon]/(\epsilon^2-1)$ that maps $\delta_i$ to $-\delta_i$ if $i\in S$, maps $\alpha$ to $\alpha+\sum_{i\in S} \delta_i$, and fixes the other generators. By \eqref{delta-flip-sym} and \eqref{mu-flip-sym}, if $S$ has an even number of elements, this endomorphism is a lift of the flip symmetry $\tau_S$ via $\Phi$.

Define two other ring homomorphisms
\begin{equation}\label{Phi-pm}
\Phi^\pm :\mathbb{A}_{g,n}\to \mathbb{V}_{g,n}^\pm
\end{equation}
by
\begin{equation*}
\Phi^\pm (f):=\frac{\id\pm E}{2}(\Phi(f)).
\end{equation*}

\subsection{Parabolic bundles}
\label{subsection_ParabolicBundles}
In this subsection, let $\Sigma$ be a closed Riemann surface of genus $g$ 
with $n$ marked points $\{p_1,\cdots, p_n\}$.
\begin{Definition}
A \emph{parabolic bundle with rank 2} over $(\Sigma,\{p_1,\cdots, p_n\})$ 
consists of the following data:
\begin{itemize}
\item A holomorphic vector bundle $E$ over $\Sigma$;

\item Filtrations $E|_{p_i}=E_{p_i}^{(1)}\supset E_{p_i}^{(2)}$.

\item Weights $w_{i}^{(1)}<w_{i}^{(2)}$ for each filtration.
\end{itemize} 
In this paper, we require that $\rank E_{p_i}^{(2)}=1$, $w_{i}^{(1)}=-1/4$, $w_{i}^{(2)}=1/4$ for all $i$.
\end{Definition}

\begin{Definition}
A \emph{parabolic line bundle} over $(\Sigma,\{p_1,\cdots, p_n\})$ 
consists of the following data:
\begin{itemize}
\item A holomorphic line bundle $L$ over $\Sigma$;
\item A weight $w_{i}$ at each $p_i$.
\end{itemize} 
In this paper, we require that $w_i\in\{-1/4,1/4\}$ for all $i$.
\end{Definition}

The reader may refer to \cite{MSesh} for the general definition of parabolic bundles with arbitrary ranks.
We will use the same notation to denote a parabolic vector bundle and its corresponding holomorphic vector bundle when there is no source of confusion.

The parabolic degree of a parabolic bundle $E$ with rank $2$ is defined to be
$$
\parr \deg E:= \deg E + \sum_i w_i^{(1)} \rank (E_{p_i}^{(1)}/E_{p_i}^{(2)}) + \sum_i w_i^{(2)}\rank E_{p_i}^{(2)}.
$$
Because of our specific choice of ranks and weights, we have
$$
\parr \deg E= \deg E.
$$
The 
parabolic degree of a parabolic line bundle $L$ is defined to be
\begin{equation}
\parr \deg L:= \deg L + \sum_i w_i.
\end{equation}

\begin{Definition}
\label{def_parabolicMap}
Let $L$ be a parabolic line bundle and $E$ be a parabolic bundle with rank $2$ on $(\Sigma,\{p_1,\cdots, p_n\})$. 
A holomorphic map $f:L\to E$ is called \emph{parabolic} if for every $p_i$ whose weight is $1/4$ with respect to $L$, we have $f(L|_{p_i})\subset E_{p_i}^{(2)}$.
\end{Definition}

\begin{remark}
Definition \ref{def_parabolicMap} agrees with the standard definition of parabolic maps under our specific choices of ranks and weights. For the general definition of a parabolic map, the reader may refer to \cite{MSesh}.
\end{remark}

A parabolic bundle $E$ with rank $2$ is called \emph{semi-stable} if every non-zero map from a parabolic line bundle $L$ to $E$ satisfies
$$
\parr \deg L  \le \frac{\parr \deg E}{2}.
$$
The parabolic bundle $E$ is called \emph{stable} if the above inequality is always strict.
When $n$ is odd, the parabolic degree of every parabolic line bundle is an odd multiple of $1/4$, and
the parabolic degree of $E$ is an integer, therefore every semi-stable parabolic bundle with rank $2$ is stable. We will assume $n$ is odd for the rest of this subsection.

For $d\in\mathbb{Z}$, let $\widetilde{N}^d_{g,n}$ be the moduli space of all rank-$2$ stable parabolic bundles on $(\Sigma,\{p_1,\cdots,p_n\})$ with weights given as above and degree equal to $d$. The space $\widetilde{N}^d_{g,n}$ is a smooth projective  variety \cite{MSesh}. Moreover 
there is a universal family   
 $(\widetilde{\bE}, \widetilde{\mathbb{E}}^{(2)})$  over 
 $(\widetilde{N}_{g,n}^d\times \Sigma,\widetilde{N}_{g,n}^d\times \{p_1,\cdots,p_n\})$ \cite[Proposition 3.2]{BY}. 
 Let $J_g$ be the Jacobian of $\Sigma$. Then $\widetilde{N}_{g,n}^d$ is a fiber bundle over $J_g$ via the determinant map,
 let $N_{g,n}^d$ be a fiber.  We denote the restriction of the universal family $(\widetilde{\bE}, \widetilde{\mathbb{E}}^{(2)})$
  to $(N_{g,n}^d\times \Sigma,N_{g,n}^d\times \{p_1,\cdots,p_n\})$ by  $({\bE},{\mathbb{E}}^{(2)})$.
  The fiber bundle $\widetilde{N}_{g,n}^d\to J_g$ is a homological product in $\bC$-coefficients 
 (cf. \cite[Proposition 9.7]{AB}). 
 Most importantly, we have the following correspondence which is a special case of a more general result
 of Mehta and Seshadri. 
\begin{Theorem}[{\cite[Theorem 4.1]{MSesh}}]\label{DUY-par}
There is a canonical diffeomorphism
$N_{g,n}^d\cong R_{g,n}$
when $d$ is even, and a canonical diffemorphism
$N_{g,n}^d\cong R_{g,n}'$
when $d$ is odd.
\end{Theorem}
We use $\ad \mathbb{\widetilde{E}}$ to denote the adjoint $\mathfrak{so}(3)$-bundle of $\mathbb{\widetilde{E}}$. As topological vector bundles, we have the decompositions 
$\mathbb{\widetilde{E}}_{p_i}\cong \mathbb{\widetilde{E}}_{p_i}^{(2)}\oplus \mathbb{\widetilde{E}}_{p_i}/\mathbb{\widetilde{E}}_{p_i}^{(2)}$ and 
$\ad \mathbb{\widetilde{E}}_{p_i}\cong  \underline{\mathbb{R}} \oplus \widetilde{\mathbb{K}}_i$ where
$$
 \widetilde{\mathbb{K}}_i :=\mathbb{\widetilde{E}}_{p_i}^{(2)}\otimes (\mathbb{\widetilde{E}}_{p_i}/\mathbb{\widetilde{E}}_{p_i}^{(2)})^{-1}.
$$
Let $\bK_i$ be the restriction of $\widetilde{\mathbb{K}}_i$ to $N_{g,n}^d$.
\begin{Theorem}[\cite{BR}]
\label{thm_generatorsForN_gd}
The cohomology classes 
\begin{equation*} 
p_1(\ad \mathbb{E})/ h ~(h\in H_\ast(\Sigma)), \text{ and }~c_1(\mathbb{K}_i )
~(1\le i\le n)
\end{equation*}
generate the cohomology ring $H^\ast(N_{g,n}^d;\mathbb{C})$.
\end{Theorem}
 We define a graded ring homomorphism 
\begin{equation}
\label{eqn_Psi-tilde}
\widetilde{\Psi}:\mathbb{A}_{g,n}\to H^\ast(\widetilde{N}_{g,n}^d,\mathbb{C})
\end{equation}
by 
$\widetilde{\Psi}(\alpha) := -\frac14 p_1(\ad  \widetilde{\bE})/[\Sigma]$, 
$\widetilde{\Psi}(\beta) := -\frac14 p_1(\ad  \widetilde{\bE})/[\pt]$,
$\widetilde{\Psi}(\psi_j) := -\frac14 p_1(\ad  \widetilde{\bE})/[a_j]$, 
$\widetilde{\Psi}(\delta_i) := -\frac{1}{2}c_1(  \widetilde{\bK}_i)$,
where $\{a_i\}_{1\le i\le 2g}$ is a set of oriented closed curves on $\Sigma$ whose holomogy classes give the standard basis of $H_1(\Sigma;\bZ)$. The curves $a_i$ are oriented such that the intersection number of $a_i$ and $a_j$ is 1 if $j=i+g$, and is $0$ if $|i-j|\neq g$.
By Theorem \ref{thm_generatorsForN_gd}, the map $\widetilde{\Psi}$ restricts to a surjective homomorphism
\begin{equation}
\label{Psi-map}
\Psi:\mathbb{A}_{g,n}\twoheadrightarrow H^\ast(N_{g,n}^d,\mathbb{C}).
\end{equation}

\subsection{Connections between $\mathbb{V}_{g,n}$ and $N_{g,n}^0$}
Let $\Sigma$ be a closed Riemann surface of genus $g$ 
with $n$ marked points $\{p_1,\cdots, p_n\}$, and suppose $n$ is odd.

We identify $R_{g,n}$ with $N_{g,n}^0$ using the canonical diffeomorphism given by Theorem \ref{DUY-par}. Let $\mathbb{P}\to\mathcal{B}\times (S^1\times \Sigma)$ be an extension of the universal $\SO(3)$-bundle over $\mathcal{B}\times (S^1\times \Sigma-S^1\times\{p_1,\cdots,p_n\})$ as in Section \ref{review-instanton}. It follows from the universal properties that we can find $\mathbb{P}$ and $\mathbb{E}$ such that the pull-back of $\mathbb{P}$ to 
$R_{g,n}\times \Sigma\cong R_{g,n}\times (\{\pt\}\times \Sigma) $ coincides with the frame bundle of $\ad\mathbb{E}$,
and the pull-back of the bundle $\mathbb{K}_{p_i}$ defined in Section \ref{review-instanton} coincides with the bundle $\mathbb{K}_i$ defined in 
Section \ref{subsection_ParabolicBundles}. Therefore $\Psi(\alpha),\Psi(\beta),\Psi(\psi_i),\Psi(\delta_j)$ coincide with 
the restrictions of the
cohomology classes $-\frac14 p_1(\mathbb{P})/[\Sigma], -\frac14 p_1(\mathbb{P})/[\pt],-\frac14 p_1(\mathbb{P})/[a_i],\delta_{p_j}$ to $R_{g,n}$.

Let $\Psi^{-1}$ be a homogeneous 
right inverse of $\Psi$ in \eqref{Psi-map}. Notice that we only require $\Psi^{-1}$ to be a homogeneous linear map instead of a 
ring homomorphism. The following result is implicitly included in {\cite[Proposition 2.6.4]{Street}}.
\begin{Proposition}[{\cite{Street}}]\label{ring-deformation}
Let $\Phi^\pm$ be the maps in \eqref{Phi-pm}.
The compositions 
\begin{equation*}
F^\pm:=\Phi^\pm\circ \Psi^{-1}:H^\ast(R_{g,n})\to \mathbb{V}^\pm _{g,n}
\end{equation*}
are isomorphisms of vector spaces. As a consequence, $\Phi$ and $\Phi^\pm$ are epimorphisms.  
\end{Proposition}

\begin{proof} For $f\in \mathbb{A}_{g,n}$, and a surface $(S,\partial S)\subset (S^2\times\Sigma,S^2\times\{p_1,\cdots,p_n\})$, let 
$$D(S^2\times\Sigma,S^2\times\{p_1,\cdots,p_n\},S, f)$$ be the Donaldson invariant of the triple $(S^2\times\Sigma,S^2\times\{p_1,\cdots,p_n\},S)$ with respect to the $\mu$ class given by $f$. In our notation, if $S\subset D^2\times\Sigma\subset S^2\times \Sigma$, then the Donaldson invariant is given by 
\begin{equation*}
D(S^2\times\Sigma,S^2\times\{p_1,\cdots,p_n\},S,f)=\II(D^2\times\Sigma, S^2\times\{p_1,\cdots,p_n\},S)(\Phi(f))\in \mathbb{C},
\end{equation*}
where 
$$\II(D^2\times\Sigma, S^2\times\{p_1,\cdots,p_n\},S):\II(S^1\times\Sigma, S^2\times\{p_1,\cdots,p_n\},\emptyset)\to\II(\emptyset,\emptyset,\emptyset)\cong \bC$$ is the cobordism map.

Pick $u,v\in H^\ast(R_{g,n})$ of degrees $d_1,d_2$ respectively. Then 
the pairing $$\langle \Phi(\Psi^{-1} (u)), \Phi(\Psi^{-1} (v))\rangle$$ 
given by \eqref{Vgn-pairing} is equal to
$D(S^2\times\Sigma, S^2\times\{p_1,\cdots,p_n\},\emptyset,\Psi^{-1} (u)\Psi^{-1} (v))$.
Since 
$$
\pi_1(S^2\times \Sigma- S^2\times\{p_1,\cdots, p_n\})= \pi_1( \Sigma- \{p_1,\cdots, p_n\}),
$$
the moduli space of flat $\SU(2)$ connections on the triple
\begin{equation}\label{S2xSigma}
(S^2\times \Sigma,S^2\times\{p_1,\cdots, p_n\},\emptyset)
\end{equation}
is just $R_{g,n}$. It is also the moduli space of ASD connections on \eqref{S2xSigma} with minimal (zero) energy. Every component of the moduli space of ASD connections on \eqref{S2xSigma} has formal dimension at least $6g+2n-6=\dim R_{g,n}$.
Therefore if $d_1+d_2< 6g+2n-6$, we have 
$D(S^2\times\Sigma,S^2\times\{p_1,\cdots,p_n\}, \Psi^{-1} (u)\Psi^{-1} (v)  )=0$. If $d_1+d_2=6g+2n-6$, then the Donaldson invariant equals the integration of the corresponding cohomology class on $R_{g,n}$, therefore
\begin{multline*}
D(S^2\times\Sigma, S^2\times\{p_1,\cdots,p_n\},\emptyset,\Psi^{-1} (u)\Psi^{-1} (v)  )
\\
=\Psi(\Psi^{-1} (u)\Psi^{-1} (v)  )[R_{g,n}] = (u\smile v) [R_{g,n}].
\end{multline*}
Every moduli space of ASD connections on the triple \eqref{S2xSigma}
with non-zero energy has formal dimension at least $\dim R_{g,n}+4$, thus if $\dim R_{g,n}<d_1+d_2<\dim R_{g,n}+4$, then    
$D(S^2\times\Sigma,S^2\times\{p_1,\cdots,p_n\},\emptyset,\Psi^{-1} (u)\Psi^{-1} (v)  )=0$. 

Similarly, the pairing  
 $\langle \Phi(\epsilon\Psi^{-1} (u)), \Phi(\Psi^{-1} (v))\rangle$ is equal to the Donaldson invariant 
\begin{equation*}
D(S^2\times\Sigma, S^2\times\{p_1,\cdots,p_n\},S^2\times \{\pt\},\Psi^{-1} (u)\Psi^{-1} (v))
\end{equation*} 
The moduli space of ASD connections with minimal energy on the triple 
$$(S^2\times \Sigma,S^2\times\{p_1,\cdots, p_n\},S^2\times\{\pt\})$$ has dimension $\dim R_{g,n}+2$. Therefore
$\langle \Phi(\epsilon\Psi^{-1} (u)), \Phi(\Psi^{-1} (v))\rangle =0$ when $d_1+d_2<\dim R_{g,n}+2$.

Notice that
\begin{align*}
\langle F^+(u), F^+(v)\rangle &=\frac{1}{4} \langle \Phi(\Psi^{-1} (u)) +  \Phi(\epsilon\Psi^{-1} (u)), 
\Phi(\Psi^{-1} (v)) +  \Phi(\epsilon\Psi^{-1} (v)) \rangle \\
&= \frac{1}{2}\langle \Phi(\Psi^{-1} (u)), \Phi(\Psi^{-1} (v) ) \rangle +\frac{1}{2} 
\langle \Phi(\epsilon\Psi^{-1} (u)), \Phi(\Psi^{-1} (v) ) \rangle.
\end{align*}
When $\deg u+\deg v=d_1+d_2 < 6g+2n-6$ or $d_1+d_2$ is odd, 
the above pairing is $0$. When $d_1+d_2=6g+2n-6$, we have
\begin{equation}\label{F+-pairing}
\langle F^+(u), F^+(v)\rangle =\frac12 u\smile v [R_{g,n}].
\end{equation}
Take a (homogeneous) basis $\{u_i\}$ for $H^\ast(R_{g,n})$, then the discussion above implies 
that the pairing matrix for $\{F^+(u_i)\}$ is a  skew upper triangular block matrix with non-degenerate 
skew diagonal blocks. Such a matrix is non-degenerate by simple linear algebra. This implies 
$F^+$ is an injection. On the other hand we have
$$
\dim \mathbb{V}^+_{g,n}=\frac{1}{2}\mathbb{V}_{g,n}\le \dim H^\ast(R_{g,n})
$$
by \eqref{eqn_VgnpmHaveSameDim} and \eqref{MB-inequality}. Therefore $F^+$ is a linear isomorphism. A similar argument works for $F^-$.
\end{proof}

The proof above established the following result, which we state separately for later reference.
\begin{Proposition}\label{Agn-pairing-Rgn}
Suppose $f,f'\in \mathbb{A}_{g,n}$ are homogeneous polynomials of degrees $d_1$ and $d_2$ respectively. Then if $d_1+d_2<\dim R_{g,n}$ or $d_1+d_2$ is odd, then
\begin{equation*}
\langle \Phi^\pm(f),\Phi^\pm(f') \rangle=0
\end{equation*}
and if $d_1+d_2=\dim R_{g,n}$, then
\begin{equation*}
\langle \Phi^\pm(f),\Phi^\pm (f') \rangle=\frac12 \Psi(f)\smile \Psi(f')[R_{g,n}].
\end{equation*}
\end{Proposition}

\begin{Corollary}\label{deform-vanishing-polynomial}
Suppose $f\in\mathbb{A}_{g,n}$ is a homogeneous polynomial of degree $d$ such that $\Psi(f)=0$. Then there is a polynomial
\begin{equation*}
\widetilde{f}=f+ \sum_{\substack{m\le d-2\\ m \equiv d \text{ mod } 2}} g_m
\end{equation*} 
where $g_m\in \mathbb{A}_{g,n}$ has degree $m$,
such that $\Phi^+(\widetilde{f})=0$. A similar result holds for $\Phi^-$.
\end{Corollary}

Roughly speaking, the above result says that the ring structures of $\mathbb{V}_{g,n}^\pm$ are graded deformations of the cohomology ring $H^\ast(R_{g,n})$. 

\begin{proof}
By Proposition \ref{Agn-pairing-Rgn} and the assumption that $\Psi(f)=0$, we have
\begin{equation*}
\langle\Phi^+(f), \Phi^+(f') \rangle=0
\end{equation*}
whenever $\deg f' \le \dim R_{g,n}-d+1$.
Let $F^+:H^\ast(R_{g,n})\to \mathbb{V}_{g,n}^+$ be the isomorphism in Proposition \ref{ring-deformation}. 
It follows from Proposition \ref{Agn-pairing-Rgn} that the pairing
\begin{equation*}
\langle\cdot,\cdot\rangle : F^+(H^{\ast\le d-2}(R_{g,n})) \times  F^+(H^{\ast\ge \dim R_{g,n}-d+2}(R_{g,n})) \to \mathbb{C}
\end{equation*}
is non-degenerate. Therefore there exists $h\in H^{\ast\le d-2}(R_{g,n})$ such that
the map
\begin{equation*}
\langle F^+(h),\cdot\rangle : F^+(H^{\ast\ge \dim R_{g,n}-d+2}(R_{g,n})) \to \mathbb{C}
\end{equation*}
is equal to the map
\begin{equation*}
\langle \Phi^+(f),\cdot\rangle : F^+(H^{\ast\ge \dim R_{g,n}-d+2}(R_{g,n})) \to \mathbb{C}.
\end{equation*}
On the other hand, the degree of $h$ and Proposition \ref{Agn-pairing-Rgn} imply 
\begin{equation*}
\langle F^+(h),F^+(f')\rangle=0=\langle \Phi^+(f),F^+(f')\rangle
\end{equation*}
for all $f'\in H^{\ast < \dim R_{g,n}-d+2}(R_{g,n})$. The non-degeneracy of the pairing on $\mathbb{V}_{g,n}$ then implies
$\Phi^+(f)=F^+(h)$. Hence 
\begin{equation*}
\Psi^{-1}(-h)=:\sum_{\substack{m\le d-2\\ m \equiv d \text{ mod } 2}} g_m
\end{equation*}
gives the desired polynomial. A similar argument works for $F^-$.
\end{proof}

\section{Cohomology of the moduli space of stable parabolic bundles}
\label{sec_parabolic}
  
Let $\Sigma$ be a closed Riemann surface of genus $g$ 
with marked points $\{p_1,\cdots, p_n\}$, and suppose $n$ is odd.
Recall that 
$$
\mathbb{A}_{g,n}=\mathbb{C}[\alpha,\beta,\psi_1,\cdots,\psi_{2g},\delta_1,\cdots,\delta_n],
$$
and there are ring homomorphisms $\widetilde{\Psi}:\bA_{g,n} \to H^*(\widetilde{N}_{g,n}^d; \bC)$ and $\Psi:\bA_{g,n} \to H^*(N_{g,n}^d; \bC)$ defined by \eqref{eqn_Psi-tilde} and \eqref{Psi-map}. In this section we will use the elements of $\bA_{g,n}$ to refer to its images under $\widetilde{\Psi}$ and $\Psi$ when it is clear that we are refering to elements of $H^*(\widetilde{N}_{g,n}^d; \bC)$ or $H^*(N_{g,n}^d; \bC)$.

\subsection{Decomposition of $H^*(N_{g,n}^d; \bC)$}
\label{Rgn-decom}
Let $W=\mathbb{C}^{2g}$ be the standard representation of $\Sp(2g,\bR)$, let $\omega\in \Lambda^2 W^*$ be an $\Sp(2g, \bR)$-invariant symplectic 2-form on $W$, and let $\gamma_\omega \in \Lambda^2 W$ be the dual of $\omega$. For $0\le k\le 2g$, let
$$
\varphi_k: \Lambda^k W \to \Lambda^{k-2} W
$$
be the contraction with $\omega$. For $k=0,1$, the map $\varphi_k$ is the zero map on $\Lambda^k W$.

Let 
$\Lambda_0^k W$ be the kernel of $
\varphi_k$. It is well-known that for $k\le g$,
$$
 \gamma_\omega^{g-k}\wedge: \Lambda_0^k W \to \Lambda^{2g-k}W
$$
is injective,
and
$$
\gamma_\omega^{g-k+1}\wedge: \Lambda_0^k W \to \Lambda^{2g-k+2}W
$$
is the zero map.  Therefore,
\begin{equation}\label{eqn_decompositionOfWedgeW}
\Lambda^*W =\bigoplus_{k=0}^g \, \Lambda_0^k W\otimes \bC[\gamma_\omega]/(\gamma_\omega^{g-k+1}).
\end{equation}
For a reference, see for example \cite[Proposition 1.2.30]{Huybrechts}.

Recall that $\{a_i\}_{1\le i\le 2g}$ is a set of oriented closed curves on $\Sigma$ whose holomogy classes give the standard basis of $H_1(\Sigma;\bZ)$. 
Define
$$
\gamma := \sum_{j=1}^g \psi_j \psi_{j+g}\in \bA_{g,n}.
$$
Identify $\spann\{\psi_1,\cdots,\psi_{2g}\}\subset \bA_{g,n}$ with $H_1(\Sigma,\bC)=\spann\{a_1,\cdots,a_{2g}\}$. 
By equation \eqref{eqn_decompositionOfWedgeW},
$$
\bA_{g,n} \cong \bigoplus_{k=0}^{g} \Lambda_0^k H_1(\Sigma;\bC) \otimes \mathbb{C}[\alpha,\beta,\gamma,\delta_1,\cdots,\delta_n]/(\gamma^{g-k+1}).
$$
We will denote the pull-backs of $\Psi$ and $\widetilde{\Psi}$ to 
$$\bigoplus_{k=0}^{g} \Lambda_0^k H_1(\Sigma;\bC) \otimes \mathbb{C}[\alpha,\beta,\gamma,\delta_1,\cdots,\delta_n]$$
still by $\Psi$ and $\widetilde{\Psi}$ by abuse of notation.
 Since the maps $\widetilde{\Psi}$ and $\Psi$ are equivariant with respect to the mapping class group of $(\Sigma, \{p_1,\cdots,p_n\})$, they are equivariant with respect to the action of $\Sp(2g,\bZ)$ on $H_1(\Sigma,\bC)$.
  By the Borel density theorem \cite{Borel}, $\Sp(2g,\bZ)$ is Zariski dense in $\Sp(2g,\bR)$. It is well-known that for $0\le k\le g$, the space $\Lambda_0^k W$ is irreducible as an $\Sp(2g,\bR)$-representation \cite[Theorem 17.5]{FultonHarris}, therefore it is also irreducible as an $\Sp(2g,\bZ)$ representation. By Schur's lemma, we have the following proposition.

\begin{Proposition}
\label{prop_DefOfIgnkd}
 The kernel of the surjection 
$$
\Psi:\bigoplus_{k=0}^{g} \Lambda_0^k H_1(\Sigma;\bC) \otimes \mathbb{C}[\alpha,\beta,\gamma,\delta_1,\cdots,\delta_n] \twoheadrightarrow H^\ast(R_{g,n}^d) $$
has the form  
$$
\ker \Psi = \bigoplus_{k=0}^{g} \Lambda_0^k H_1(\Sigma;\bC) \otimes I_{g,n,k}^d,
$$
where 
$I_{g,n,k}^d$ are ideals of  $\bC[\alpha,\beta,\gamma,\delta_1,\cdots,\delta_n].$ \qed
\end{Proposition}
From now on, we will use $I_{g,n,k}^d$ to denote the ideals given by Proposition \ref{prop_DefOfIgnkd}.

\begin{Corollary}
\begin{equation}\label{decom}
H^\ast(N_{g,n}^d)\cong \bigoplus_{k=0}^{g} \Lambda_0^k H_1(\Sigma;\bC) \otimes \bC[\alpha,\beta,\gamma,\delta_1,\cdots,\delta_n]/I_{g,n,k}^d.
\end{equation}
\end{Corollary}

The following result was proved for the non-punctured case in \cite{KingNewstead}, but the proof works verbatim for the punctured case.
\begin{Theorem}[\cite{KingNewstead}, Theorem 3.2]
\label{thm_IOnlyDependsOng-k}
The ideal $I_{g,n,k}^d$ only depends on $g-k$, $n$, and $d$. 
\end{Theorem}

The same decomposition argument applies to the maps $\Phi^\pm$, and we have the following result.

\begin{Proposition}
\label{prop_decompositionOfPhi}
 The kernel of the surjection 
$$
\Phi^+:\bigoplus_{k=0}^{g} \Lambda_0^k H_1(\Sigma;\bC) \otimes \mathbb{C}[\alpha,\beta,\gamma,\delta_1,\cdots,\delta_n] \twoheadrightarrow \bV_{g,n}^+ $$
has the form  
$$
\ker \Phi^+ = \bigoplus_{k=0}^{g} \Lambda_0^k H^1(\Sigma;\bC) \otimes J_{g,n,k}^+
$$
where 
$J_{g,n,k}^+$ are ideals of  $\bC[\alpha,\beta,\gamma,\delta_1,\cdots,\delta_n].$ 
Therefore
$$
\bV_{g,n}^+\cong \bigoplus_{k=0}^{g} \Lambda_0^k H_1(\Sigma;\bC) \otimes \bC[\alpha,\beta,\gamma,\delta_1,\cdots,\delta_n]/J_{g,n,k}^+.
$$
The same result also holds for $\Phi^-$. \qed
\end{Proposition}
From now on, we will use $J_{g,n,k}^\pm$ to denote the ideals given by Proposition \ref{prop_decompositionOfPhi}.
We will discuss more about these ideals in Section \ref{sec_cobordism}.

\subsection{Mumford relations}
Recall that $n$ is odd. Let $m:=(n-1)/2$. 
The main result of this subsection is the following proposition.
\begin{Proposition} \label{prop_MumfordRelations}
Given $g,n,d$, there exists a homogeneous polynomial $f\in I_{g,n,0}^d$ with degree $2(g+m)$, such that the coefficient of $\alpha^{g+m}$ in $f$ is $1$.
\end{Proposition}

Notice that taking the tensor product with a fixed line bundle of degree $1$ maps $N_{g,n}^d$ to $N_{g,n}^{d+2}$ while preserving the cohomology classes 
$$\Psi(\alpha),\Psi(\beta),\Psi(\gamma),\Psi(\delta_1),\cdots,\Psi(\delta_n).$$ On the other hand, 
using the identification of $N_{g,n}^d$ and $R_{g,n}$ (or $R_{g,n}'$) given by Theorem \ref{DUY-par}, 
the discussions in Section \ref{subsection-flip} and Section \ref{subsec_Vgn} show that
there is a flip symmetry that maps $N_{g,n}^0$ to $N_{g,n}^1$ and maps $\Psi(\alpha)$ to $\Psi(\alpha + \delta_1)$, maps 
$\Psi(\delta_1)$ to $-\Psi(\delta_1)$, and preserves $\Psi(\beta)$, $\Psi(\gamma)$ and $\Psi(\delta_i)$ ($2\le i\le n$).
Therefore, for every given pair $(g,n)$, one only needs to verify Proposition \ref{prop_MumfordRelations} for a single value of $d$. We choose the value of $d$ as follows. If $m$ is even, take $d=1$; if $m$ is odd, take $d=0$. We will use $N_{g,n}$ and $\widetilde N_{g,n}$ to denote the corresponding moduli spaces with this particular choice of $d$.

Let $\hat L$ be a parabolic line bundle over $(\Sigma,\{p_1\,\cdots,p_n\})$ such that all the weights $w_i=-1/4$, and $\deg \hat L = (d+m+1)/2$. Then  $\parr \deg \hat L = d /2+1/4$. 

Let $E$ be a stable parabolic bundle with rank $2$ and degree $d$, such that its weights are given as in Section \ref{subsection_ParabolicBundles}. Let $\parr \mathcal{H}om(\hat L,E)$ be the sheaf of parabolic maps from $\hat L$ to $E$, and let $\mathcal{H}om(\hat L,E)$ be the sheaf of holomorphic maps from $\hat L$ to $E$ as vector bundles. Since all the weights on $\hat L$ are taken to be $-1/4$, we have
$$\parr \mathcal{H}om(\hat L,E)=\mathcal{H}om(\hat L,E).$$

Recall that $\widetilde \bE\to \widetilde N_{g,n}\times \Sigma$ is a universal bundle of the stable parabolic bundles with rank $2$ and degree $d$ over $(\Sigma,\{p_1,\cdots,p_n\})$.
For all $x\in \widetilde N_{g,n}$, let $E_x$ be the restriction of $\bE$ to $\{x\}\times \Sigma$. Since $E_x$ is stable, 
$\Hom(\hat L,E_x)= \parr\Hom(L,E_x) = 0$. By the Riemann-Roch theorem, $\rank H^1(\mathcal{H}om(\hat L,E_x))= 2g+m-1$. 
By the standard construction of index bundles, the vector spaces $H^1(\mathcal{H}om(\hat L,E_x))$ can be arranged to define a vector bundle of rank $(2g+m-1)$ over $\widetilde N_{g,n}$. Alternatively, this vector bundle can be constructed by derived functors as follows. 
Let $\pi:\widetilde N_{g,n}\times \Sigma\to \widetilde N_{g,n}$ be the projection map, let $L$ be the pull-back of $\hat L$ to $\widetilde N_{g,n}\times\Sigma$.
By cohomology and base change, we have
$R^0\pi_\ast \mathcal{H}om(L, \widetilde\bE)=0$, and 
$R^1\pi_\ast \mathcal{H}om(L, \widetilde\bE)$ is a holomorphic vector bundle with rank $(2g+m-1)$ over $\widetilde N_{g,n}$. 

Let
$$
\mathbf{R}\pi_\ast \mathcal{H}om(L, \widetilde\bE) := [R^0\pi_\ast \mathcal{H}om(L, \widetilde\bE)]-[R^1\pi_\ast \mathcal{H}om(L, \widetilde\bE)]
\in K(\widetilde N_{g,n}).
$$
For a topological space $X$ and $V\in K(X)$, let $c_t(V) := \sum_{n=0}^\infty c_n(V) t^n$ be the formal power series generated by the Chern classes of $V$ in $\bC$-coefficients. 

Let $\sigma\in H^2(\Sigma;\bC)$ be the Poincar\'e dual of $[\pt]$. Recall that $\{a_j\}_{j=1}^{2g}$ is a set of oriented simple closed curves on $\Sigma$ such that the intersection number of $a_i$ and $a_j$ is 1 if $j=i+g$, and is zero if $|i-j|\neq g$. Let $e_1,\cdots,e_{2g}\in H^1(\Sigma,\bC)$ be the Poincar\'e duals of $a_1,\cdots,a_{2g}$. We have the following lemma.

\begin{Lemma} \label{lem_GRRForPushForward}
Let $X$ be a smooth projective variety over $\bC$, and let $V$ be a holomorphic vector bundle over $X\times \Sigma$. Let $\mathcal{V}$ be the sheaf of holomorphic sections of $V$. If 
\begin{equation}
\label{eqn_lnc_t(V)}
\ln c_t(V) = u(t)\otimes 1 + \sum_{j=1}^{2g} v_j(t)\otimes e_j + w(t)\otimes \sigma
\end{equation}
for $u(t),v_j(t),w(t)\in H^*(X;\bC)[[t]]$, then
$$
\ln c_t(\mathbf{R}\pi_\ast  \mathcal{V}) = -(g-1) u(t) - \int_0^t \frac{w(s)-w'(0)s}{s^2}\, ds.
$$
The integral is defined formally on power series using term-by-term integrations.
\end{Lemma}

\begin{proof}
By the Grothendieck–Riemann–Roch theorem,
$$
\ch (\mathbf{R}\pi_\ast  \mathcal{V}) = (\ch(V)\TD(\Sigma)) /[\Sigma].
$$
Write
$$
u(t) =: \sum_{i=1}^\infty u^{(2i)} t^i,
$$
$$
v_j(t) =: \sum_{i=1}^\infty v_j^{(2i-1)} t^i,
$$
$$
w(t) =: \sum_{i=1}^\infty w^{(2i-2)} t^i.
$$
Let $k$ be the rank of $V$,
then by \eqref{eqn_lnc_t(V)} and \cite[Lemma 1]{Zagier},
$$
\ch (V) = k + \sum_{i=1}^\infty \frac{(-1)^{i+1}}{(i-1)!} (u^{(2i)}+v_j^{(2i-1)}\otimes e_j + w^{(2i-2)}\otimes \sigma).
$$
Since $\TD(\Sigma) = 1-(g-1)\sigma$, we have
\begin{align*}
\ch (\mathbf{R}\pi_\ast  \mathcal{V}) &= (\ch(V)\TD(\Sigma)) /[\Sigma]
\\
&= - k(g-1)
+\sum_{i=1}^\infty \frac{(-1)^{i+1}}{(i-1)!}  w^{(2i-2)} 
-(g-1)\sum_{i=1}^\infty \frac{(-1)^{i+1}}{(i-1)!} u^{(2i)},
\end{align*}
thus by \cite[Lemma 1]{Zagier} again,
$$
\ln c_t(\mathbf{R}\pi_\ast  \mathcal{V}) = -(g-1) \sum_{i=1}^\infty u^{(2i)}t^i - \sum_{i=2}^\infty\frac{1}{i-1}w^{(2i-2)}t^{i-1},
$$
and this proves the lemma.
\end{proof}

Recall that $L\to \widetilde N_{g,n}\times\Sigma$ is the pull-back of a parabolic line bundle over $\Sigma$ with all the weights equal to $-1/4$ and degree equals to $(d+m+1)/2$. 
We will use Lemma \ref{lem_GRRForPushForward} to compute the Chern classes of $R^1\pi_\ast  \mathcal{H}om(L, \widetilde\bE)$ in terms of $\alpha$, $\beta$, and $\gamma$. Since $\rank R^1\pi_\ast \parr \mathcal{H}om(L, \widetilde\bE) = 2g+m-1$, every element in the ideal generated by 
$$c_k(R^1\pi_\ast \parr \mathcal{H}om(L, \widetilde\bE)),\qquad k\ge 2g+m
$$
is zero in $H^\ast(\widetilde{N}_{g,n};\bC)$. Our argument is adapted from \cite{Zagier}.

Recall that $J_g$ denotes the Jacobian of $\Sigma$. Let $\bL\to J_g\times \Sigma$ be a universal family for the line bundles of degree $d$ over $\Sigma$, such that $\det \widetilde{\bE}$ is the pull-back of $\bL$ from $J_g$. Let $d_j\in H^1(J_g;\bC)$ be given by $d_j := c_1(\bL)/[a_j]$. Then there exists $x\in H^2(J_g;\bC)$ such that
$$
c_1(\widetilde{\bE}) = d\otimes \sigma+\sum_{j=1}^{2g} d_j\otimes e_j + x\otimes 1,
$$
where the right-hand side is understood as the pull-back from $H^*( J_g\times \Sigma;\bC)$ to $H^\ast(\widetilde{N}_{g,n}\times\Sigma;\bC)$. Therefore
$$
c_1(\Hom(L,\widetilde{\bE})) = 
-(m+1)\otimes \sigma + \sum_{j=1}^{2g}d_j \otimes e_j+ x\otimes 1.
$$
On the other hand, by the definitions of $\alpha$, $\psi_j$, and $\beta$,
$$
-\frac14 p_1(\ad \widetilde \bE) = \alpha \otimes \sigma + \sum_{j=1}^{2g} \psi_j\otimes e_j+\beta\otimes 1.
$$
Therefore
\begin{align*}
c_2(\Hom(L,\widetilde{\bE})) & = \frac14 c_1(\Hom(L,\widetilde{\bE}))^2 -\frac14 p_1(\ad \widetilde \bE) \\
& = -\frac12 \sum_{j=1}^{g} (d_j\smile d_{j+g}) \otimes \sigma + \alpha \otimes \sigma + \sum_{j=1}^{2g} \psi_j\otimes e_j+\beta\otimes 1\\
& \qquad + \frac12 (x\otimes 1)\smile \Big(-(m+1)\otimes \sigma + \sum_{j=1}^{2g}d_j \otimes e_j\Big) + \frac14 x^2\otimes 1.
\end{align*}

The formulas for the Chern classes of $\Hom(L,\widetilde{\bE})$ contains an unknown term $x$ because the universal family $\widetilde{\bE}$ is not unique. In fact, taking the tensor product of $\widetilde{\bE}$ with the pull-back of a line bundle over $\widetilde{N}_{g,n}$ produces another universal family. Formally, the contribution of the terms involving $x$ is equivalent to taking the tensor product of $\widetilde{\bE}$ with a line bundle with first Chern class $x/2$. Let $c_x(t)$ be the formal polynomial where the coefficient of $t^k$ is equal to $c_k(R^1\pi_\ast \parr \mathcal{H}om(L, \widetilde\bE))$, which depends on  $x$, for all $k$. Since the rank of  $R^1\pi_\ast \parr \mathcal{H}om(L, \widetilde\bE)$ is $2g+m-1$, we have
$$
c_0(t) = (1-tx)^{2g+m-1}c_x\Big(\frac{t}{1-tx}\Big).
$$
Therefore, the fact that $\deg c_x(t)\le 2g+m-1$ for some $x$ implies that $\deg c_0(t)\le 2g+m-1$. Since we are only interested in the ideal generated by $$c_k(R^1\pi_\ast \parr \mathcal{H}om(L, \widetilde\bE)),\qquad k\ge 2g+m,
$$
we can assume $x=0$ without loss of generality.

To simplify notations, we will omit the tensor product symbols from the following computations.
Let 
\begin{align*}
&D  := \sum_{j=1}^{2g}d_j  e_j, \quad \Psi := \sum_{j=1}^{2g} \psi_j e_j,
\\
&A := \sum_{j=1}^{g} d_j\smile d_{j+g}, \quad 
B := \sum_{j=1}^g (-d_j\smile\psi_{j+g}+d_{j+g}\smile \psi_j).
\end{align*}
Then $D^2=-2A \sigma$, $\Psi^2 = -2\gamma\sigma$, $D\Psi = B\sigma$, and $D^3=D^2\Psi=D\Psi^2=\Psi^3=0$. 

Under the assumption that $x=0$,
$$
c_t(\Hom(L,\widetilde{\bE})) = 1+ \big(-(m+1)\sigma+D\big) t + (\alpha\sigma + \Psi+\beta-A\sigma/2)t^2.
$$
Therefore,
\begin{align*}
\ln c_t(\Hom(L,\widetilde{\bE})) &= 
\ln (1+\beta t^2) + \ln \Big(1+\frac{\big(-(m+1)\sigma+D\big) t +(\alpha\sigma + \Psi-A\sigma/2)t^2}{1+\beta t^2} \Big)
\\
&= \ln (1+\beta t^2) +\frac{\big(-(m+1)\sigma+D\big) t +(\alpha\sigma + \Psi-A\sigma/2)t^2}{1+\beta t^2} 
\\
& \qquad -\frac12 \Big(\frac{\big(-(m+1)\sigma+D\big) t +(\alpha\sigma + \Psi-A\sigma/2)t^2}{1+\beta t^2} \Big)^2
\\
&= \ln (1+\beta t^2) +\frac{\big(-(m+1)\sigma+D\big) t +(\alpha\sigma + \Psi-A\sigma/2)t^2}{1+\beta t^2} 
\\
& \qquad -\frac12 \frac{-2A\sigma t^2 +2B\sigma t^3 - 2\gamma\sigma t^4 }{(1+\beta t^2)^2}.
\end{align*}

Applying lemma \ref{lem_GRRForPushForward} with
$$
u(t) = \ln (1+\beta t^2)
$$
and
$$
w(t) = \frac{-(m+1)t+(\alpha-A/2) t^2}{1+\beta t^2} +\frac{At^2-Bt^3+\gamma t^4}{(1+\beta t^2)^2},
$$ we have
\begin{align*}
&\ln c_t(\mathbf{R}\pi_\ast  \mathcal{H}om(L, \widetilde\bE))) 
\\
=& -(g-1)\ln (1+\beta t^2) - \int_0^t
\Bigg( 
\frac{(\alpha-A/2) +(m+1)\beta s}{1+\beta s^2} +\frac{A-Bs+\gamma s^2}{(1+\beta s^2)^2}
\Bigg)
\, ds
\\
=& (\frac12 -g - \frac{m}{2})\ln (1+\beta t^2) 
+\frac{t}{2}\cdot \frac{-A+\gamma/\beta +Bt}{1+\beta t^2}+\frac{2\alpha\beta+\gamma}{4(-\beta)^{3/2}}\cdot \ln \Big(\frac{1+t\sqrt{-\beta}}{1-t\sqrt{-\beta}}\Big),
\end{align*}
therefore
\begin{multline*}
c_t(\mathbf{R}\pi_\ast  \mathcal{H}om(L, \widetilde\bE)))
\\
=
(1+\beta t^2)^{1/2-g-m/2}  \Big(\frac{1+t\sqrt{-\beta}}{1-t\sqrt{-\beta}}\Big)^{\frac{2\alpha\beta+\gamma}{4(-\beta)^{3/2}}} \exp\Big(\frac{t}{2} \frac{-A+\gamma/\beta +Bt}{1+\beta t^2}\Big).
\end{multline*}
As a consequence,
\begin{multline}
\label{eqn_c_tR^1}
c_t(R^1\pi_\ast  \mathcal{H}om(L, \widetilde\bE)))
\\
=
(1+\beta t^2)^{-1/2+g+m/2}  \Big(\frac{1-t\sqrt{-\beta}}{1+t\sqrt{-\beta}}\Big)^{\frac{2\alpha\beta+\gamma}{4(-\beta)^{3/2}}} \exp\Big(-\frac{t}{2} \frac{-A+\gamma/\beta +Bt}{1+\beta t^2}\Big).
\end{multline}

Recall that $c_t(R^1\pi_\ast \parr \mathcal{H}om(L, \widetilde\bE))\in H^*(\widetilde{N}_{g,n})[[t]]$ and 
$$H^*(\widetilde{N}_{g,n};\bC) \cong H^*(N_{g,n};\bC) \otimes H^*(J_g;\bC).$$
Let $[J_g]\in H^{2g}(J_g;\bC)$ be the Poincar\'e dual of $\Pi_{j=1}^{g} (d_j\smile d_{j+g})$. 
The following lemma follows from the definitions of $A$ and $B$ and straightforward algebra. For details, the reader may refer to \cite[Lemma 3]{Zagier}. The difference in constants between our statement and the one in \cite{Zagier} is due to the different conventions in the definitions of $\alpha$, $\beta$, and $\gamma$.
\begin{Lemma}[\cite{Zagier}]
$$
\frac{A^r}{r!} \frac{B^s}{s!}/[J_g]=\begin{cases}
\frac{(-\gamma)^p}{p!} &\quad \text{ if } 2r+s=2g, \,p=s/2,\\
0 &\quad \text{ if } 2r+s\neq 2g.
\end{cases}
$$
\end{Lemma}
As a corollary,
\begin{Corollary}[\cite{Zagier}]
\label{cor_ZagierCor}
Let $\kappa\in H^*(N_{g,n})[[t]]$. Then 
$$
e^{\kappa(-A+Bt)}/[J_g]=(-1)^g\kappa^g e^{\kappa \gamma t^2}.
$$
\end{Corollary}

\begin{proof}
By the previous lemma,
\begin{align*}
e^{\kappa(-A+Bt)}/[J_g] &= \sum_{r+p=g}
\frac{(-A)^r(Bt)^{2p}\kappa^{r+2p}}{r!(2p)!}/[J_g]\\
&=\sum_{r+p=g}\frac{(-1)^r(-\gamma)^p \, t^{2p}}{p!}\kappa^{r+2p}
\\
&= \sum_{p=0}^g (-1)^g \kappa^g \frac{\gamma^p t^{2p}\kappa^{p}}{p!}
\\
&= (-1)^g\kappa^g e^{\gamma \kappa t^2}.
\end{align*}
The last equality follows from the fact that $\gamma^{g+1}=0$.
\end{proof}

\begin{Proposition}
\begin{multline*}
c_t(R^1\pi_\ast \parr \mathcal{H}om(L, \widetilde\bE)))/[J_g] \\
=  (1+\beta t^2)^{-1/2+m/2} \Big(\frac{t}{2}\Big)^g \Big(\frac{1-t\sqrt{-\beta}}{1+t\sqrt{-\beta}}\Big)^{\frac{2\alpha\beta+\gamma}{4(-\beta)^{3/2}}} \exp\Big(-\frac{t\gamma}{2\beta}\Big).
\end{multline*}
\end{Proposition}
\begin{proof}
Apply Corollary \ref{cor_ZagierCor} to \eqref{eqn_c_tR^1} with
$\kappa = -t/(2(1+\beta t^2)),$ we have
\begin{align*}
&c_t(R^1\pi_\ast \parr \mathcal{H}om(L, \widetilde\bE)))/[J_g] \\
=&  (1+\beta t^2)^{-1/2+m/2} \Big(\frac{t}{2}\Big)^g \Big(\frac{1-t\sqrt{-\beta}}{1+t\sqrt{-\beta}}\Big)^{\frac{2\alpha\beta+\gamma}{4(-\beta)^{3/2}}} 
\exp\Big(-\frac{\gamma t}{2\beta (1+\beta t^2)} - \frac{\gamma t^3}{2(1+\beta t^2)}\Big)
\\
=& (1+\beta t^2)^{-1/2+m/2} \Big(\frac{t}{2}\Big)^g \Big(\frac{1-t\sqrt{-\beta}}{1+t\sqrt{-\beta}}\Big)^{\frac{2\alpha\beta+\gamma}{4(-\beta)^{3/2}}} \exp\Big(-\frac{t\gamma}{2\beta}\Big).
\end{align*}
\end{proof}

\begin{proof}[Proof of Proposition \ref{prop_MumfordRelations}]
Let 
$$F(t) := (1+\beta t^2)^{-1/2+m/2} \Big(\frac{1-t\sqrt{-\beta}}{1+t\sqrt{-\beta}}\Big)^{\frac{2\alpha\beta+\gamma}{4(-\beta)^{3/2}}}\exp\Big(-\frac{t\gamma}{2\beta}\Big),$$
then $F(t)$ satisfies the following equation
$$
(1+\beta t^2)\frac{F'(t)}{F(t)} = -\frac{\gamma}{2} t^2 
+ (m-1)\beta t +\alpha.
$$
Write
$$
F(t) =: \sum_{k=0}^\infty \xi_{k,n}(\alpha,\beta,\gamma) t^{k}
$$
with $\xi_{k,n}\in \bC(\alpha,\beta,\gamma)$, then $\{\xi_{k,n}\}$ satisfies the following recursive relation for $k\ge 2$:
$$(k+1)\xi_{k+1,n}= \alpha \xi_{k,n} + (m-k)\beta \xi_{k-1,n} - \frac{\gamma}{2}\xi_{k-2,n}.$$
Moreover,
$\xi_{0,n} = 1$, $\xi_{1,n} = \alpha$, $\xi_{2,n}=\alpha^2/2+(m-1)\beta/2$. As a result, $\xi_{k,n}$ is a homogeneous polynomial in $\bC[\alpha,\beta,\gamma]$ of degree $2k$, and the coefficient of $\alpha^k$ in $\xi_{k,n}$ is $1/k!$. 
Since 
$$
\Big(\frac{t}{2}\Big)^g\cdot F(t)  = c_t(R^1\pi_\ast \parr \mathcal{H}om(L, \widetilde\bE)))/[J_g],
$$
it follows from the previous discussion that when $k\ge g+m$, the cohomology class $\Psi(\xi_{k,n})\in H^*(N_{g,n};\bC)$ is zero. Since $\xi_{k,n}$ is invariant under the $\Sp(2g,\bZ)$ action, we have $\xi_{g+m,n}\in I_{g,n,0}$. The result is proved by taking $f:=(g+m)!\,\xi_{g+m,n}$.
\end{proof}

\section{Cobordism maps}
\label{sec_cobordism}
In this section, let $\Sigma_g$ be a closed surface with genus $g$, let $\{p_1,\cdots, p_n\}$ be marked points on $\Sigma_g$, and suppose $n$ is odd.
Recall that by Proposition \ref{prop_decompositionOfPhi}, we have
\begin{equation}\label{Vgn-decom}
\mathbb{V}_{g,n}^\pm\cong \bigoplus_{k=0}^{g}\Lambda_0^k H_1(\Sigma_g) \otimes \mathbb{C}[\alpha,\beta,\gamma,\delta_1,\cdots,\delta_n]/
J_{g,k,n}^\pm,
\end{equation}
where $J_{g,k,n}^\pm\subset \mathbb{C}[\alpha,\beta,\gamma,\delta_1,\cdots,\delta_n]$ are ideals given by $\ker\Phi^\pm$. Since the map $F^+$ ($F^-$) in Proposition \ref{ring-deformation} can be taken to be $\Sp(2g,\bZ)$-equivariant, by Proposition \ref{ring-deformation} and Corollary \ref{deform-vanishing-polynomial},
\begin{equation}\label{eqn_equalDim}
\dim \mathbb{C}[\alpha,\beta,\gamma,\delta_1,\cdots,\delta_n]/J_{g,k,n}^\pm
=\dim \mathbb{C}[\alpha,\beta,\gamma,\delta_1,\cdots,\delta_n]/I_{g,k,n}^0,
\end{equation}
and
 $J_{g,k,n}^+$ ($J_{g,k,n}^-$) is a deformation of $I_{g,k,n}$, i.e. the elements of $I_{g,k,n}$ are given by the leading order terms of the elements of $J_{g,k,n}^+$ ($J_{g,k,n}^-$). 

We show that the ideal $J_{g,k,n}$ only depends on $g-k$ and $n$. Our argument is 
adapted from \cite[Section 5]{D-TFT}.
Attaching a 1-handle to the product cobordism $\Sigma_g\times [0,1]$ gives the elementary cobordism $Z:\Sigma_g\to \Sigma_{g+1}$.
The cobordism $(S^1\times Z,S^1\times  \{p_1,\cdots,p_n\}\times [0,1],\emptyset)$ defines a map $f^{g,n}_{g+1,n}: \bV_{g,n} \to \bV_{g+1,n}$.
This map is a $\mathbb{C}[\alpha,\beta,\gamma,\delta_1,\cdots,\delta_n,\epsilon]$-module homomorphism.
To see this, notice that the multiplication by $\Phi(\alpha)$ on $\bV_{g,n}$ ($\bV_{g+1,n}$)
is the same as the operator $\mu(\Sigma_g)$ ($\mu(\Sigma_{g+1})$) and the cobordism map intertwines $\mu(\Sigma_g)$ on the in-coming end
with $\mu(\Sigma_{g+1})$ on the out-going end since $\Sigma_g$ and $\Sigma_{g+1}$ are homologous in $Z$.
Similar arguments work for the other generators of the ring.
 In particular, $f^{g,n}_{g+1,n}$ maps $\bV_{g,n}^\pm$ to $\bV_{g+1,n}^\pm$. 
 Moreover, $f^{g,n}_{g+1,n}$ is equivariant with respect to  the action of $\text{Diff}(Z)$. 
Let $\delta$ be the belt circle of the $1$-handle, i.e. $\delta$ is the non-separating circle in $\Sigma_{g+1}$ which is null-homologous in $Z$. We use $G$ to denote the subgroup of  
$\Sp(2g+2,\mathbb{Z})$ whose action on $H_1(\Sigma_{g+1};\bZ)$ fixes $[\delta]$. Since every element of $G$ can be realized by an 
element of $\text{Diff}(Z)$, the map $f^{g,n}_{g+1,n}$ intertwines with the $G$-actions on $\bV_{g,n}$ and $\bV_{g+1,n}$. 
We have the following result:
\begin{Proposition}
There are linear maps 
$$
g_k:\Lambda_0^k H_1(\Sigma_g) \to  \Lambda_0^{k+1} H_1(\Sigma_{g+1})
$$
and 
$$
h_k: \mathbb{C}[\alpha,\beta,\gamma,\delta_1,\cdots,\delta_n]/J_{g,k,n}^+
\to  \mathbb{C}[\alpha,\beta,\gamma,\delta_1,\cdots,\delta_n]/J_{g+1,k+1,n}^+
$$
such that
\begin{equation}\label{f-decom}
f_{g+1,n}^{g,n}|_{ \mathbb{V}_{g,n}^+}=\bigoplus_k g_k\otimes h_k
\end{equation}
under the identification $\eqref{Vgn-decom}$.
Moreover, $h_k$ is an isomorphism of $\mathbb{C}[\alpha,\beta,\gamma,\delta_1,\cdots,\delta_n]$-modules for all $k$.
The same result also holds for $f_{g+1,n}^{g,n}|_{ \mathbb{V}_{g,n}^-}$. 
\end{Proposition}
\begin{proof}[Sketch of the proof]
The proof is exactly the same as \cite[Lemma 22]{D-TFT}.  We briefly sketch the argument and refer the reader to \cite{D-TFT} for more details.
By Schur's lemma, every $G$-equivariant map from 
$\Lambda_0^k H_1(\Sigma_g)$ to  $\Lambda_0^{j} H_1(\Sigma_{g+1})$ is zero when $j\neq k,k+1,k+2$. 
Therefore $f_{g+1,n}^{g,n}$ maps 
$$\Lambda_0^k H_1(\Sigma_g)\otimes \mathbb{C}[\alpha,\beta,\gamma,\delta_1,\cdots,\delta_n]/J_{g,k,n}^+$$ into  
$$\bigoplus_{j=k}^{k+2}\Lambda_0^{j} H_1(\Sigma_{g+1})\otimes \mathbb{C}[\alpha,\beta,\gamma,\delta_1,\cdots,\delta_n]/J_{g,j,n}^+.$$
The components of $f_{g+1,n}^{g,n}$ that maps 
$\Lambda_0^k H_1(\Sigma_g)\otimes \mathbb{C}[\alpha,\beta,\gamma,\delta_1,\cdots,\delta_n]/J_{g,k,n}^+$
into
$$\Lambda_0^{k} H_1(\Sigma_{g+1})\otimes \mathbb{C}[\alpha,\beta,\gamma,\delta_1,\cdots,\delta_n]/J_{g,k,n}^+$$
or
$$\Lambda_0^{k+2} H_1(\Sigma_{g+1})\otimes \mathbb{C}[\alpha,\beta,\gamma,\delta_1,\cdots,\delta_n]/J_{g,k+2,n}^+$$
 must be zero because of the $\mathbb{Z}/4$-gradings. Moreover, by Schur's lemma again, there is a unique $G$-equivariant map up to scalar from $\Lambda_0^k H_1(\Sigma_g)$ to $\Lambda_0^{k+1} H_1(\Sigma_{g+1})$, so we obtain the decomposition \eqref{f-decom}. 
 The cobordism $Z$ can be embedded into the product cobordism $\Sigma_g\times [0,2]$, where the complement of $Z$ is a 
cobordism $Z':\Sigma_{g+1}\to \Sigma_g$. The functoriality of instanton Floer homology then implies that $f_{g+1,n}^{g,n}$ is an injection. On the other hand, by \eqref{eqn_equalDim} and
 Theorem \ref{thm_IOnlyDependsOng-k}, 
$$\dim \mathbb{C}[\alpha,\beta,\gamma,\delta_1,\cdots,\delta_n]/J_{g,k,n}^+
=\dim \mathbb{C}[\alpha,\beta,\gamma,\delta_1,\cdots,\delta_n]/J_{g+1,k+1,n}^+,
$$ hence $h_k$ must be an isomorphism.
\end{proof}

\begin{Corollary}
\label{cor_IdealOnlyDependsOng-k}
The ideal $J_{g,k,n}^+$ only depends on $g-k$ and $n$. The same result also holds for $J_{g,k,n}^-$. \qed
\end{Corollary}
From now on we will also denote $J_{g,k,n}^\pm$ by $J_{g-k,n}^\pm$. The decomposition \eqref{Vgn-decom} is then rewritten as
\begin{equation}\label{Vgn-decom2}
\mathbb{V}_{g,n}^\pm \cong \bigoplus_{k=0}^{g}\Lambda_0^k H_1(\Sigma_g) \otimes \mathbb{C}[\alpha,\beta,\gamma,\delta_1,\cdots,\delta_n]/
J_{g-k,n}^\pm.
\end{equation}
Therefore, to understand the ring structure of $\bV_{g,n}^\pm$, it suffices to understand the ideals
$J_{g',n}^\pm$ for $g'\le g$. 

The following lemma is an analogue of \cite[Lemma 9]{Munoz}. 
\begin{Lemma}
\label{lem_inclusionOfJIng+1}
We have 
\begin{equation*}
J_{g,n}^\pm \subset J_{g-1,n}^\pm
\end{equation*}
\end{Lemma}
\begin{proof}
For each $i$, we have
$$
 \Phi^\pm([a_i]J_{g,n}^\pm)=\Phi^\pm([a_i])\Phi^\pm(J_{g,n}^\pm)=  \{0\}\subset \mathbb{V}_{g,n}^\pm.
$$
Since $[a_i]\in \Lambda^1 H_1(\Sigma_g)=\Lambda_0^1 H_1(\Sigma_g)$, 
The decomposition \eqref{Vgn-decom} implies 
$$J_{g,n}^\pm \subset J_{g,1,n}^\pm=J_{g-1,n}^\pm.$$
\end{proof}

\begin{Lemma} \label{lem_multiplyGamma}
We have 
\begin{equation*}
\gamma J_{g-1,n}^\pm \subset J_{g,n}^\pm.
\end{equation*}
\end{Lemma}
\begin{proof}
Recall that in  the decomposition \eqref{Vgn-decom}, the space $H_1(\Sigma_g)$ is identified with $\spann\{\psi_1,\cdots,\psi_{2g}\}$.
Since $[a_i]\in \Lambda^1 H_1(\Sigma_g)=\Lambda_0^1 H_1(\Sigma_g)$,
 the decomposition \eqref{Vgn-decom} implies that for every $j$,
 $$
 \Phi^\pm(\psi_j J_{g-1,n}^\pm)=\Phi^\pm(\psi_j J_{g,1,n}^\pm)=\{0\}\subset \bV_{g,n}^\pm.
 $$
Since $\gamma=\sum_{j=1}^g \psi_j \psi_{j+g}$,
we have
$$\Phi^\pm(\gamma J_{g-1,n}^\pm) \subset \sum_{j=1}^g\Phi^\pm(\psi_{j} J_{g-1,n}^\pm)= \{0\} \subset \bV_{g,n}^\pm,$$
therefore $\gamma J_{g-1,n}^\pm\subset J_{g,n}^\pm$.
\end{proof}

\begin{figure}
\centering
\begin{tikzpicture}
\draw[dashed] (0,0) to (8,0);  \node at (-0.4,-0.05) {$\Sigma_g$};
\draw[dashed] (0,4) to (8,4);  \node at (-0.4,3.95) {$\Sigma_g$};

\draw[thick]  (0.5,0) to (0.5,4); \node[below] at (0.5,0) {1};
\draw[thick]  (1.5,0) to (1.5,4);  \node[below] at (1.5,0) {2};
\draw[thick]  (2.5,0) to (2.5,4);  \node[below] at (2.5,0) {3};
\node    at (3.25,2) {\bf{...}};   \node[below] at (3.25,-0.1) {\bf{...}};
\draw[thick]  (4,0) to (4,4);      \node[below] at (4,-0.07) {n};
\draw[thick]  (7,0) arc [radius=1, start angle=0, end angle=180];
\node[below] at (5,0) {n+1};
\node[below] at (7,0) {n+2};
\end{tikzpicture}
\caption{The U-cobordism}\label{U-cob}
\end{figure}

Notice that $J_{g,n}^\pm$ and $J_{g,n+2}^\pm$ are ideals in two different rings. The following lemma is stated with respect to the obvious inclusion 
\begin{equation}
\label{eqn_def_iota}
\iota^{g,n}_{g,n+2}:\mathbb{C}[\alpha,\beta,\gamma,\delta_1,\cdots,\delta_n] \to 
\mathbb{C}[\alpha,\beta,\gamma,\delta_1,\cdots,\delta_n,\delta_{n+1},\delta_{n+2}].
\end{equation}

\begin{Lemma} \label{lem_multiplyDelta}
There exist $q_{g,n}\in\{1,-1\}$ and $c_{g,n}\in \mathbb{C}$ such that
$$(  \delta_{n+1} +   q_{g,n}\,\delta_{n+2} +c_{g,n} )J_{g,n}^+\subset J_{g,n+2}^+.$$
 A similar result holds for $J_{g,n}^-$ and $J_{g,n+2}^-$.
\end{Lemma}
\begin{remark}
\label{rmk_c=0}
We will show that $c_{g,n}=0$ in  Lemma \ref{lem_multiplyDeltac=0}.
\end{remark}
\begin{proof}
Let $\mathbb{Y}_{g,n}:=(S^1\times\Sigma_g, S^1\times\{p_1,\cdots,p_n\}, \emptyset)$ be the admissible triple that defines $\mathbb{V}_{g,n}$.
Consider the ``U-cobordism'' 
\begin{equation*}
\mathbb{W}:=([0,1]\times S^1\times \Sigma_g, [0,1]\times S^1\times \{p_1,\cdots,p_n\}\cup  U, \emptyset):\mathbb{Y}_{g,n}\to \mathbb{Y}_{g,n+2},
\end{equation*}
where $U$ is the product of $S^1$ with an arc joining the points $p_n,p_{n+1} \in \{1\}\times \Sigma_g$ in  $[0,1]\times\Sigma_g $. 
Figure \ref{U-cob} shows a slice of the U-cobordism (from the top to the bottom) in $[0,1]\times \Sigma_g$.
The map $\II(\mathbb{W})$ is a 
$$
\mathbb{A}_{g,n}[\epsilon]/(\epsilon^2-1)
=\mathbb{C}[\alpha,\beta,\psi_1,\cdots, \psi_{2g},\delta_1,\cdots,\delta_n,\epsilon]/(\epsilon^2-1)
$$
module homomorphism since the homology classes used to define 
$$\mu(\Sigma_g),\mu(\pt),\mu(a_1),\cdots,\mu(a_{2g}),\sigma_1,\cdots,\sigma_n,E$$ on the two ends are homologous.
This map is also equivariant with respect to the $Sp(2g,\mathbb{Z})$ action on $\bV_{g,n}$ and $\bV_{g,n+2}$ since the action of the mapping class group $MCG(\Sigma_g)$ 
can be lifted to $\mathbb{W}$. 

Let $R(\mathbb{W})$ be the space of flat connections on $\mathbb{W}$. Recall that using 
\eqref{Rgn-BC}, we define a
codimension 2 submanifold $D_{n+1,n+2}^-$ ($D_{n+1,n+2}^+$) of $R_{g,n+2}$ by the equation  $C_{n+1}=-C_{n+2}$ ($C_{n+1}=C_{n+2}$).
The restriction map $R(\mathbb{W})\to R(\mathbb{Y}_{g,n+2})=R_{g,n+2}\sqcup R_{g,n+2}$ 
is an embedding which maps $R(\mathbb{W})$ onto $D_{n+1,n+2}^-\sqcup D_{n+1,n+2}^-$.
The restriction map 
$$ D_{n+1,n+2}^-\sqcup D_{n+1,n+2}^-\cong R(\mathbb{W}) \to R(\mathbb{Y}_{g,n})=R_{g,n}\sqcup R_{g,n}$$
is given by the projection map which forgets the coordinates $C_{n+1}$ and $C_{n+2}$.

We define 
$$
\mathbb{D}_1:=(D^2\times \Sigma_g, D^2\times \{p_1,\cdots,p_n\},\emptyset),
$$
$$
\mathbb{D}_2:=(D^2\times \Sigma_g, D^2\times \{p_1,\cdots,p_{n+2}\},\emptyset),
$$
and
$$
\mathbb{D}_2':=(D^2\times \Sigma_g, D^2\times \{p_1,\cdots,p_{n+2}\},\{\pt\}\times \Sigma_g).
$$
Now we calculate $\II(\mathbb{W})(\Phi^+(1))$. Recall that $\II(\mathbb{W})$ is only defined up to a sign, we choose an arbitrary sign and fix it from now on. By the definition of $\Phi^+$ in \eqref{Phi-pm},
$$
\II(\mathbb{W})(\Phi^+(1))=\frac12 (\II(\mathbb{D}_1\cup\mathbb{W})(1) + E\circ \II(\mathbb{D}_1\cup\mathbb{W})(1) ).
$$
Consider the pairing 
$$
\langle \II(\mathbb{D}_1\cup\mathbb{W})(1), \Phi(f)    \rangle
$$
on $\mathbb{V}_{g,n+2}$ for a given $f\in \mathbb{A}_{g,n+2}$. This is the same as the Donaldson invariant 
$D(\mathbb{D}_1\cup\mathbb{W}\cup \mathbb{D}_2,f)$ with a suitable homological orientation.
The minimal energy moduli space on $\mathbb{D}_1\cup\mathbb{W}\cup \mathbb{D}_2$ is the space of flat connections.
It is the fibered product 
$$R(\mathbb{D}_1)\times_{R(\mathbb{Y}_{g,n})} R(\mathbb{W})\times_{R(\mathbb{Y}_{g,n+2})}R(\mathbb{D}_2),$$ which can be identified with $D_{n+1,n+2}^-\subset R_{g,n+2}=R(\mathbb{D}_2)$. 

If $d=\deg f \le  \dim R_{g,n+2}-3$ or $d=\dim R_{g,n+2}-1$, then 
\begin{equation}\label{f-vanihsing}
\langle \II(\mathbb{D}_1\cup\mathbb{W}), \Phi(f)    \rangle=0
\end{equation} 
since no component of the moduli space of $ASD$ connections over 
$\mathbb{D}_1\cup\mathbb{W}\cup \mathbb{D}_2$ has dimension $d$.
If $d=\dim R_{g,n+2}-2$, then we have
\begin{equation}\label{f-D-minus}
\langle \II(\mathbb{D}_1\cup\mathbb{W}), \Phi(f)    \rangle=\Psi(f)[D_{n+1,n+2}^-].
\end{equation}
Similarly, we can consider the paring 
$$
\langle \II(\mathbb{D}_1\cup\mathbb{W}), \Phi(\epsilon f)    \rangle.
$$
It is the same as the  Donaldson invariant
$D(\mathbb{D}_1\cup\mathbb{W}\cup \mathbb{D}_2'; f)$.
The minimal energy moduli space on $\mathbb{D}_1\cup\mathbb{W}\cup \mathbb{D}_2'$ has dimension $\dim D_{n+1,n+2}^-+2$. 
Therefore
\begin{equation}\label{epsilon-f-vanishing}
\langle \II(\mathbb{D}_1\cup\mathbb{W}), \Phi(\epsilon f)\rangle =0
\end{equation}
if $\deg f < \dim D_{n+1,n+2}^-+2=\dim R_{g,n+2}$.
Now combine \eqref{f-vanihsing}, \eqref{f-D-minus} and \eqref{epsilon-f-vanishing}, we obtain
\begin{equation}\label{W-map-vanishing}
\langle \II(\mathbb{W})(\Phi^+(1)), \Phi^+(f)\rangle=0
\end{equation}
when $\deg f \le \dim R_{g,n+2} -3$ or $\deg f=\dim R_{g,n+2} -1$ and 
\begin{equation}\label{W-map-D-minus}
\langle \II(\mathbb{W})(\Phi^+(1)), \Phi^+(f)\rangle=\frac12\Psi(f)[D_{n+1,n+2}^-]
\end{equation}
when $\deg f=  \dim R_{g,n+2}-2$.

If $\Psi(\delta_{n+1})$ and $\Psi(\delta_{n+2})$ are linearly dependent\footnote{In fact $\Psi(\delta_{n+1})$ and $\Psi(\delta_{n+2})$ are always linearly independent unless $g=0,n=1$. 
This can be seen from the fact that the Betti number $b_2(R_{g,n+2})=n+3$ unless $g=0,n=1$ 
(See \cite{Nitsure} or \cite[Section 1.3]{Street}).},
then there are non-zero constants $a,b\in\mathbb{C}$ such that
$a\Psi(\delta_{n+1})+b\Psi(\delta_{n+2})=0$ in $H^\ast(R_{g,n})$. By flip symmetry we also have 
$a\Psi(\delta_{n+1})-b\Psi(\delta_{n+2})=0$. This means $\Psi(\delta_{n+1})=\Psi(\delta_{n+2})=0$ in $H^\ast(R_{g,n})$. 
By Proposition \ref{deform-vanishing-polynomial} we have $\Phi^+(\delta_{n+1}+a')=0$ and
$\Phi^+(\delta_{n+2}+b')=0$ for some $a',b'\in \mathbb{C}$. Therefore we have 
\begin{equation*}
(\delta_{n+1}+\delta_{n+2}+a'+b')J_{g,n}^+=\{0\}\subset J_{g,n+2}^+.
\end{equation*}

Now we assume $\Psi(\delta_{n+1})$ and $\Psi(\delta_{n+2})$ are linearly independent.
According to Proposition \ref{delta=Dij}, we have 
$$
P.D.\Psi(\delta_{n+1})=[D_{n+1,n+2}^+]+[D_{n+1,n+2}^-]
$$ 
after choosing the orientations of $[D_{n+1,n+2}^+]$ and $[D_{n+1,n+2}^-]$ properly. Since 
$\Psi(\delta_{n+1})$ and $\Psi(\delta_{n+2})$ are linearly independent, we must have
$$
P.D.\Psi(\delta_{n+2})=\pm([D_{n+1,n+2}^+]-[D_{n+1,n+2}^-]).
$$
Hence there is some $q_{g,n}\in \{1,-1\}$ such that 
\begin{equation}\label{delta+delta=D-minus}
P.D.\Psi(\delta_{n+1}+q_{g,n}\delta_{n+2})=2[D_{n+1,n+2}^-].
\end{equation}

According to \eqref{W-map-vanishing}, \eqref{W-map-D-minus}, \eqref{delta+delta=D-minus}, and Proposition \ref{Agn-pairing-Rgn},
we have 
\begin{equation*}
\langle \II(\mathbb{W})(\Phi^+(1))- \frac12 \Phi^+ (\delta_{n+1}+q_{g,n}\delta_{n+2}), \Phi^+(f) \rangle=0
\end{equation*}
when $\deg f \le \dim R_{g,n+2}-1$. 
Now let $F^+:H^\ast(R_{g,n+2})\to \mathbb{V}_{g,n+2}$ be the map in Proposition \ref{ring-deformation}. 
Assume 
\begin{equation*}
\langle \II(\mathbb{W})(\Phi^+(1))- \frac12 \Phi^+ (\delta_{n+1}+q_{g,n}\delta_{n+2}), F^+(P.D.[\pt]) \rangle=:\frac14 c_{g,n}
\end{equation*}
By the above two equations and Proposition \ref{Agn-pairing-Rgn} we have
\begin{equation*}
\langle \II(\mathbb{W})(\Phi^+(1))- \frac12 \Phi^+ (\delta_{n+1}+q_{g,n}\delta_{n+2}+c_{g,n}), F^+(g) \rangle =0 
\end{equation*}
for all $g\in H^\ast(R_{g,n+2})$. The non-degeneracy of the pairing on $\mathbb{V}_{g,n}^+$ implies 
\begin{equation*}
 \II(\mathbb{W})(\Phi^+(1))= \frac12 \Phi^+ (\delta_{n+1}+q_{g,n}\delta_{n+2}+c_{g,n})
\end{equation*}
Since $\II(\mathbb{W})$ is equivariant with respect to the $Sp(2g,\mathbb{Z})$-action and commutes with the $\epsilon$-action, it maps the $Sp(2g,\mathbb{Z})$-invariant part
of $\mathbb{V}_{g,n}^+$ into the  $Sp(2g,\mathbb{Z})$-invariant part of $\mathbb{V}_{g,n+2}^+$. Since $\II(\mathbb{W})$ is an
$\mathbb{A}_{g,n}$-module homomorphism, we have
\begin{equation}
\label{eqn_U_map_bW}
0=\II(\mathbb{W})(\Phi^+(h))=\Phi^+(h)\cdot \II(\mathbb{W})(\Phi^+(1))=\Phi^+(h) \cdot \frac12 \Phi^+ (\delta_{n+1}+q_{g,n}\delta_{n+2}+c_{g,n})
\end{equation}  
for all $h\in J_{g,n}^+$, hence $(  \delta_{n+1} +   q_{g,n}\,\delta_{n+2} +c_{g,n} )J_{g,n}^+\subset J_{g,n+2}^+.$ A similar argument works for $J_{g,n}^-$.
\end{proof}

Notice that the cobordism map $\bW$ defined in the proof of Lemma \ref{lem_multiplyDelta} has a left inverse given by another U-cobordism where the U-arc connects $p_n$ and $p_{n+1}$, therefore $\bW$ is an injective map. Hence by \eqref{eqn_U_map_bW} we have the following lemma, which is also proved in \cite{Street} using different methods. 
\begin{Lemma}[{cf. {\cite[Corollary 2.6.8]{Street}}}]
\label{lem_inclusionWhenIncreaseN}
Let $\iota^{g,n}_{g,n+2}$ be the map defined by \eqref{eqn_def_iota}. We have 
\begin{equation*}
 (\iota^{g,n}_{g,n+2})^{-1} J_{g,n+2}^\pm\subset  J_{g,n}^\pm.
\end{equation*}
\end{Lemma}

\begin{proof}
Suppose $f\in  (\iota^{g,n}_{g,n+2})^{-1} J_{g,n+2}^+$, then by \eqref{eqn_U_map_bW}, we have 
$$\II(\bW)(\Phi^+(f))  = 0 \in \bV_{g,n+2}.$$
Since $\II(\bW)$ is injective, we have $\Phi^+(f)=0$, namely $f\in J_{g,n}^+$. The same argument works for $J_{g,n}^-$.
\end{proof}

\section{Eigenvalues of the surface class}\label{section-spec} 
In this section, let $\Sigma_{g,n}$ be a closed surface of genus $g$ with $n$ marked points $\{p_1,\cdots,p_n\}$. We study the eigenvalues and eigenspaces of $\muu(\Sigma)$ on $\bV_{g,n}$.

\subsection{Existence of eigenvectors}\label{subsection-spec-existence}

Let $(Y_{g},L_{g,n}):=S^1\times (\Sigma_{g,n},\{p_1,\cdots,p_n\})$. Fix a point $x_0\in S^1$.
If $g\ge 1$, let $c_1\subset \Sigma_{g,n}-\{p_1,\cdots,p_n\}$ be a non-separating simple closed curve. Let $\omega$ be the circle 
$\{x_0\} \times c_1$ in $(Y_{g},L_{g,n})$ and let
$$
\bW_{g,n} := \II(Y_{g},L_{g,n},\omega).
$$
If $n\ge 2$, let $c_2\subset \Sigma_{g,n}$ be an arc that connects $p_1$ and $p_2$.
Let $\omega'$ be the arc $ \{x_0\}\times c_2$ in $(Y_{g},L_{g,n})$ and let 
$$\bU_{g,n} := \II(Y_{g},L_{g,n},\omega').$$ 

Notice that $\bW_{g,n}$ and $\bU_{g,n}$ are still well-defined when $n$ is \emph{even}, which is different from the case of 
$\mathbb{V}_{g,n}$. 
We will study the eigenvalues of $\muu(\Sigma_{g,n})$ and $\mu(\pt)$ on $\bU_{g,n}$ and $\bW_{g,n}$ using 
Theorem \ref{thm_nonsingularExcision}. By Proposition \ref{mu-pt=2}, $\mu(\pt)=2$ on $\bU_{g,n}$.
The following result holds for both even and odd values of $n$.

\begin{Lemma}\label{Lem-UgnWgn-eigenvalues}
The eigenvalues of $\muu(\Sigma_{g,n})$ on $\bU_{g,n}$ ($n\ge 2$) are 
$$
\{-(2g+n-2),-(2g+n-4),\cdots, 2g+n-4,2g+n-2 \}.
$$
Moreover, the generalized eigenspaces of $\muu(\Sigma_{g,n}) $
with the eigenvalue $2g+n-2$ and $-(2g+n-2)$ are 1-dimensional. 
The same result holds for $\bW_{g,n}$ ($g\ge 1$) when restricted to the generalized eigenspace of $\mu(\pt)$ for 
the eigenvalue $2$.
\end{Lemma}
\begin{proof}
We use $\bW_{g,n}^2$ to denote the generalized eigenspaces of $\muu(\pt)$ for the eigenvalue $2$
in $\bW_{g,n}$.
When $g\ge 1$ and $n\ge 2$, consider the disjoint union of $(Y_{g},L_{g,0},\omega)$ and 
$(Y_{0},L_{0,n},\omega')$. Let $u_1$ be a simple closed curve on $\Sigma_{g,0}$ that intersects $c_1$ transversely at one point, and let $u_2$ be a 
simple closed curve on $\Sigma_{0,n}$ that intersects $c_2$ transversely at one point. 
Apply Theorem \ref{thm_nonsingularExcision}
on $(Y_{g,0},L_{g,0},\omega)\sqcup(Y_{0,n},L_{0,n},\omega')$ for the excision along $S^1 \times u_1$ and $S^1\times u_2$ yields
\begin{equation}
\label{eqn_U_{g-1,n}GivenByTensorProduct}
\bW_{g,0}^2\otimes\bU_{0,n}\cong \bU_{g-1,n}.
\end{equation}

By \cite[Proposition 2.5.4]{Street}, the spectrum of $\muu(\Sigma_{g,0})=\mu(\Sigma_{g,0})$ on $\bW_{g,0}^2$ is 
$$
\{-2(g-1),-2(g-2),\cdots, -2,0,2,\cdots,2(g-2),2(g-1)\},
$$
and the generalized eigenspace for the top (bottom) eigenvalue is 1-dimensional. 
The eigenvalues are rescaled here to agree with our conventions.
By \cite[Section 2.8]{Street}, the spectrum of $\muu(\Sigma_{0,n})$ on $\bU_{0,n}$  is 
$$
\{-(n-2),-(n-4),\cdots,-2,0,2,\cdots, n-4,n-2 \},
$$
and the generalized eigenspace for the top (bottom) eigenvalue is 1-dimensional. 
Since $\Sigma_{g,0}\cup\Sigma_{0,n}$ is homologous to $\Sigma_{g,n}$ in the excision cobordism, the isomorphism
\eqref{eqn_U_{g-1,n}GivenByTensorProduct} intertwines $\muu(\Sigma_{g,0})\otimes 1+1\otimes \muu(\Sigma_{0,n})$ with
$\muu(\Sigma_{g-1,n})$. 
Therefore by \eqref{eqn_U_{g-1,n}GivenByTensorProduct}, when $g\ge 0$ and $n\ge 2$, the spectrum of $\muu(\Sigma_{g,n})$ on $\bU_{g,n}$ is 
$$
\{-(2g+n-2),-(2g+n-4),\cdots, 2g+n-4,2g+n-2  \},
$$
and the generalized eigenspace with the top (bottom) eigenvalue is 1-dimensional. 

Let $u_3$ and $u_4$ be two small  
circles around $p_1$ and $p_2$ respectively on $\Sigma_{g-1,n}$ so  
that both $u_3$ and $u_4$ intersect $c_2$ transversely at one point.
Excision on $(Y_{g-1,n},L_{g-1,n},\omega')$ along $S^1\times u_3$ and $S^1\times u_4$ yields that if $g\ge 1$ and $n\ge 2$,
\begin{equation*}
\bU_{g-1,n}\cong \bW_{g,n-2}^2\otimes \bU_{0,2}\cong \bW_{g,n-2}^2.
\end{equation*}
The above isomorphism and the previous result for $\bU_{g,n}$ then determines the spectrum of $\muu(\Sigma_{g,n})$ on $\bW_{g,n}^2$ when $g\ge 1$ by a similar argument.
\end{proof}

Recall that we always assume $(g,n)\neq (0,1)$ when talking about $\bV_{g,n}$.
\begin{Lemma} \label{lem_ExistenceOfEigenvectors}
Suppose $n$ is odd.
For each $\lambda\in \{-(2g+n-2),-(2g+n-4),\cdots, 2g+n-4,2g+n-2 \}$, there is a non-zero simultaneous eigenvector of
$ (\muu(\Sigma),\mu(\pt))$
with eigenvalues
$(\lambda, 2)$
in $\mathbb{V}_{g,n}$.
\end{Lemma}

\begin{proof}
It suffices to show that the spectrum of $(\muu(\Sigma_{g,n}),\mu(\pt))$ on $\bU_{g,n}$ (when $n\ge 2$) and $\bW_{g,n}$ (when $g\ge 1$)
is a subset of the spectrum of $(\muu(\Sigma_{g,n}),\mu(\pt))$ on $\bV_{g,n}$. This follows from the same argument as \cite[Lemma 2.8.2]{Street}. To start, if $n\ge 2$, remove a disk from the interior of $S^1\times[0,1]$ that is disjoint from $\{x_0\}\times[0,1]$, and take the product space of the resulting pair of pants with the triple $(\Sigma_{g,n}, \{p_1,\cdots,p_n\}, c_2)$. This gives a cobordism from $(Y_{g,n},L_{g,n},\omega')\sqcup (Y_{g,n},L_{g,n},\omega')$ to $(Y_{g,n},L_{g,n},\emptyset)$, hence it defines a map from $\bU_{g,n}\otimes\bU_{g,n}$ to $\bV_{g,n}$. There is another map from $\bV_{g,n}$ to $\bC$ defined by the cobordism
$$(D^2\times \Sigma_{g,n} , D^2\times \{p_1,\cdots,p_n\}, \emptyset).$$
The composition of these two maps is a bilinear pairing $\bU_{g,n}\otimes\bU_{g,n}\to \bC$ that is given by the product triple
$$([0,1]\times S^1 \times \Sigma_{g,n}, [0,1]\times S^1 \times\{p_1,\cdots,p_n\}, [0,1]\times\{x_0\}\times c_2)$$
with the orientation of one of the boundary components reversed. 
Therefore, the composition map $\bU_{g,n}\otimes\bU_{g,n}\to \bC$ is a non-degenerate pairing, hence the map $\bU_{g,n}\otimes\bU_{g,n}\to \bV_{g,n}$ is non-degenerate on each simultaneous generalized eigenspace of $(\muu(\Sigma_{g,n}),\mu(\pt))$, which implies that the spectrum of $(\muu(\Sigma_{g,n}),\mu(\pt))$ on $\bU_{g,n}$ is a subset of  its spectrum on $\bV_{g,n}$.
A similar argument works for $\bW_{g,n}$ when $g\ge 1$. 
\end{proof}

\subsection{Multiplicity of the top eigenvalues}
Recall that the ring $\bV_{g,n}$ decomposes as $\bV_{g,n}^+\oplus \bV_{g,n}^-$ by the eigenvalues of $E$, and $\Phi=\Phi^++\Phi^-$ under this decomposition. In the following we will focus our discussion on $\bV_{g,n}^+$, as the case for $\bV_{g,n}^-$ is essentially the same. Since $\bV_{g,n}^+$ is finite dimensional, there exists an integer $k_{g,n}$ such that
$$\ker(\bV_{g,n}^+\xrightarrow{(\Phi^+(\beta)-2)^{k_{g,n}}} \bV_{g,n}^+) = \ker(\bV_{g,n}^+\xrightarrow{(\Phi^+(\beta)-2)^{1+k_{g,n}}} \bV_{g,n}^+).$$
By the definition of $k_{g,n}$, we have
$$
(J^+_{g,n}, (\beta-2)^{k_{g,n}}) = (J^+_{g,n}, (\beta-2)^{k_{g,n}+h})
$$
for every integer $h\ge 0$.

This subsection studies the spectrum of $\muu(\Sigma_{g,n})$ in $\bV_{g,n}^+$ in the generalized eigenspace of $\mu(\pt)$ for the eigenvalue $2$. 

\begin{Lemma}
\label{lem_JVSameEigenvalues}
The set of simultaneous eigenvalues of
$$(\muu(\Sigma_{g,n}),\mu(\pt), \sum_{j}\mu(a_j)\mu(a_{j+g}),\sigma_{p_1},\cdots,\sigma_{p_n})$$ 
in $\bV_{g,n}^+$ is equal to the set of simultaneous eigenvalues of the multiplications by
$(\alpha,\beta,\gamma,\delta_1,\cdots,\delta_n)$
in
$\bC[\alpha,\beta,\gamma,\delta_1,\cdots,\delta_n]/J^+_{g,n}.$
\end{Lemma}

\begin{proof}
By Lemma \ref{lem_inclusionOfJIng+1}, every vanishing polynomial of $\alpha,\beta,\gamma,\delta_1,\cdots,\delta_n$ in 
\begin{equation}\label{eqn_decompositionRingJgn+1}
\bC[\alpha,\beta,\gamma,\delta_1,\cdots,\delta_n]/J^+_{g+1,n}
\end{equation}
 is a vanishing polynomial in
 \begin{equation}\label{eqn_decompositionRingJgn}
\bC[\alpha,\beta,\gamma,\delta_1,\cdots,\delta_n]/J^+_{g,n},
\end{equation}
hence the spectrum of $(\alpha,\beta,\gamma,\delta_1,\cdots,\delta_n)$ in \eqref{eqn_decompositionRingJgn} is a subset of its spectrum in \eqref{eqn_decompositionRingJgn+1}. The result then follows from the decomposition \eqref{Vgn-decom2}.
\end{proof}

\begin{Corollary}
\label{cor_JVSameEigenvalues}
The set of simultaneous eigenvalues of
$$(\muu(\Sigma_{g,n}),\sum_{j}\mu(a_j)\mu(a_{j+g}),\sigma_{p_1},\cdots,\sigma_{p_n})$$ 
on $\bV_{g,n}^+$ in the generalized eigenspace of $\mu(\pt)$ for the eigenvalue $2$ is equal to the set of simultaneous eigenvalues of the multiplications by
$(\alpha,\gamma,\delta_1,\cdots,\delta_n)$
in
$\bC[\alpha,\beta,\gamma,\delta_1,\cdots,\delta_n]/(J^+_{g,n}, (\beta-2)^{k_{g,n}}).$
\end{Corollary}

\begin{Proposition} \label{lem_AllEigenvalues}
The set of simultaneous eigenvalues of
$$(\muu(\Sigma_{g,n}),\sum_{j}\mu(a_j)\mu(a_{j+g}),\sigma_{p_1},\cdots,\sigma_{p_n})$$ 
on $\bV_{g,n} = \bV_{g,n}^+\oplus \bV_{g,n}^-$ in the generalized eigenspace of $\mu(\pt)$ for the eigenvalue $2$ is given by 
$(\lambda, 0, \cdots, 0)$ for $\lambda\in \{-(2g+n-2),-(2g+n-4),\cdots, 2g+n-4,2g+n-2 \}$. Moreover, half of 
the eigenvalues are from $\bV_{g+n}^+$, the other half comes from  $\bV_{g+n}^-$.
\end{Proposition}

\begin{proof}
By Lemma \ref{lem_multiplyGamma} and the observation that $\Phi(\gamma)=0\in \bV_{0,n}$, the element $\Phi^\pm(\gamma)$ is nilpotent in $\bV_{g,n}^\pm$. Therefore the eigenvalue of $\sum_{j}\mu(a_j)\mu(a_{j+g})$, which equals the multiplication by $\Phi(\gamma)$, is always zero. \

By \eqref{delta^2+beta} we have $\mu(\pt) + \sigma_{p_i}^2= 2\id $, hence the eigenvalue of $\sigma_{p_i}$ in the generalized eigenspace of $\mu(\pt)$ for the eigenvalue $2$ is always zero. Therefore, the spectra of $\muu(\Sigma_{g,n})$ and $\mu(\Sigma_{g,n})$ are the same in the generalized eigenspace of $\mu(\pt)$ for the eigenvalue $2$.

By Corollary \ref{deform-vanishing-polynomial} and Proposition \ref{prop_MumfordRelations}, there exists a polynomial
$$f^+(\alpha,\beta,\gamma,\delta_1,\cdots,\delta_n)\in J_{g,n}^+,$$ 
such that $\deg f^+=2(g+m)$ and the coefficient of $\alpha^{g+m}$ in $f^\pm$ is $1$. Since in the generalized eigenspace of $\mu(\pt)$ for the eigenvalue $2$, the only eigenvalue of the multiplications by $\Phi^+(\gamma)$ and $\Phi^+(\delta_i)$ is zero,  the multiplication by $\alpha$ has at most $(g+m)$ different eigenvalues in
\begin{equation*}
\bC[\alpha,\beta,\gamma,\delta_1,\cdots,\delta_n]/(J^+_{g,n}, (\beta-2)^{k_{g,n}}).
\end{equation*}
By Corollary \ref{cor_JVSameEigenvalues}, we conclude that in the generalized eigenspace of $\mu(\pt)$ for the eigenvalue $2$, the operator $\muu(\Sigma_{g,n})$ has at most $(g+m)$ different eigenvalues in $\bV_{g,n}^+$. A similar result holds for $\bV_{g,n}^-$. The result then follows from Lemma \ref{lem_ExistenceOfEigenvectors}.
\end{proof}

\begin{Lemma} \label{lem_eigenvaluesOfAlphaInV^+}
There exists a sequence of integers $\lambda_1,\lambda_2,\cdots,$ such that
\begin{itemize}
\item $|\lambda_i|=2i-1$ for each $i$.
\item For each pair $(g,n)$ with $n=2m+1$, The set of simultaneous eigenvalues of 
$$(\muu(\Sigma_{g,n}),\sum_{j}\mu(a_j)\mu(a_{j+g}),\sigma_{p_1},\cdots,\sigma_{p_n})$$ 
in $\bV_{g,n}^+$ in the generalized eigenspace of $\mu(\pt)$ for the eigenvalue $2$ is given by 
$(\lambda, 0, \cdots, 0)$ for $\lambda=\lambda_1,\cdots,\lambda_{g+m}$
\end{itemize}
\end{Lemma}

\begin{proof}
By Lemma \ref{lem_inclusionWhenIncreaseN}, every vanishing polynomial of $\alpha$ in 
$$
\bC[\alpha,\beta,\gamma,\delta_1,\cdots,\delta_{n+2}]/(J^+_{g,n+2}, (\beta-2)^{k_{g,n+2}})
$$ 
is vanishing in  
$$
\bC[\alpha,\beta,\gamma,\delta_1,\cdots,\delta_{n}]/(J^+_{g,n}, (\beta-2)^{k_{g,n}}).
$$ 
Therefore by Corollary \ref{cor_JVSameEigenvalues},  in the generalized eigenspaces of $\mu(\pt)$ for the eigenvalue $2$, the spectrum of 
$\muu(\Sigma_{g,n})$ in $\bV_{g,n}^+$ is a subset of the spectrum of $\muu(\Sigma_{g,n+2})$ in $\bV_{g,n+2}^+$. It also follows from \eqref{Vgn-decom2} that in the generalized eigenspaces of $\mu(\pt)$ for the eigenvalue $2$, the spectrum of 
$\muu(\Sigma_{g,n})$ in $\bV_{g,n}^+$ is a subset of the spectrum of $\muu(\Sigma_{g+1,n})$ in $\bV_{g+1,n}^+$.
 The result then follows from Proposition \ref{lem_AllEigenvalues} and induction on $g$ and $n$.
\end{proof}

For the rest of this section, let $\lambda_1,\lambda_2,\cdots,$ be the sequence of eigenvalues  given by Lemma \ref{lem_eigenvaluesOfAlphaInV^+}.

\begin{Corollary} \label{cor_P(alpha)}
Suppose $n=2m+1$.
There exists a positive integer $N_{g,n}$ such that 
$$
P_{g,n}(\alpha):=\prod_{i=1}^{g+m} (\alpha - \lambda_i)^{N_{g,n}} \in (J^+_{g,n}, (\beta-2)^{k_{g,n}}).
$$
\end{Corollary}
For notational convenience, we define $P_{-1,n}(\alpha):=1$.
\begin{Lemma}\label{lem_Q(alpha)}
Suppose $n=2m+1$, we have
$$
Q_{g,n}(\alpha):=\prod_{i=1}^{g+m} (\alpha-\lambda_i) \in (J^+_{g,n}, \beta-2,\delta_1,\cdots,\delta_n,\gamma).
$$
\end{Lemma}

\begin{proof}
By Corollary \ref{deform-vanishing-polynomial} and Proposition \ref{prop_MumfordRelations}, there exists a monic polynomial $f_{g,n}\in\bC[\alpha]$ with leading term $\alpha^{g+m}$ in the ideal $(J^+_{g,n}, \beta-2,\delta_1,\cdots,\delta_n,\gamma)$. By Lemma \ref{lem_JVSameEigenvalues}, every eigenvalue of $\muu(\Sigma_{g,n})$ in the simultaneous generalized eigenspace of 
$$(\mu(\pt),\sum_{j}\mu(a_j)\mu(a_{j+g}),\sigma_{p_1},\cdots,\sigma_{p_n})$$ for the eigenvalues $(2,0,\cdots,0)$ in $\bV_{g,n}^+$ is a root of $f_{g,n}(\alpha)$. By Lemma \ref{lem_eigenvaluesOfAlphaInV^+}, 
$$f_{g,n}(\alpha)=\prod_{i=1}^{g+m} (\alpha-\lambda_i) . $$
\end{proof}

\begin{Lemma}
\label{lem_PolynomialHgn}
Let $q_{g,n}\in\{1,-1\}$, $c_{g,n}\in \bC$ be the constants given by Lemma \ref{lem_multiplyDelta}.
There exists a positive integer $M_{g,n}$, such that
$$
H_{g,n}(\alpha):= \prod_{i=1}^{g+m} (\alpha-\lambda_i)^{M_{g,n}}
$$
satisfies the following property. For every $S\subset\{1,\cdots,n\}$, we have
$$
H_{g,n}(\alpha+\sum_{i\in S} \delta_i)(\delta_{n+1}+q_{g,n}\,\delta_{n+2}+c_{g,n}) \in (J_{g,n+2}^+, (\beta-2)^{k_{g,n+2}}).
$$
\end{Lemma}
\begin{proof}
By \eqref{delta^2+beta}, for each $i=1,\cdots,n+2,$ we have $\Phi^+(\delta_i)$ is nilpotent in $\bV_{g,n+2}^+$ in the generalized eigenspace of $\mu(\pt)$ for the eigenvalue $2$, thus there exists a positive integer $m_{g,n+2}$ such that 
\begin{equation}
\label{eqn_deltaNilpotent}
\delta_i^{m_{g,n+2}}\in (J_{g,n+2}^+, (\beta-2)^{k_{g,n+2}}) ~\text{ for }i=1,\cdots,n+2.
\end{equation}
Let $N_{g,n}$ be the positive integer given by Corollary \ref{cor_P(alpha)}. By Lemma \ref{lem_multiplyDelta}, we have
\begin{equation}
\label{eqn_PTimesDelta}
\prod_{i=1}^{g+m} (\alpha-\lambda_i)^{N_{g,n}}(\delta_{n+1}+q_{g,n}\delta_{n+2}+c_{g,n})\in (J_{g,n+2}^+, (\beta-2)^{k_{g,n+2}}).
\end{equation}
Let $M_{g,n}$ be a positive integer such that $M_{g,n}\ge n\,m_{g,n+2} + N_{g,n}$, it follows from \eqref{eqn_deltaNilpotent} and \eqref{eqn_PTimesDelta} 
that the polynomial $H_{g,n}(\alpha)$ satisfies the desired property.
\end{proof}

\begin{Lemma}
\label{lemHDelta}
For $i=1,\cdots,n+2$, we have 
$$
\delta_i\cdot H_{g,n}(\alpha) \in (J_{g,n+2}^+, (\beta-2)^{k_{g,n+2}}).
$$
\end{Lemma}
\begin{proof}
Let $q_{g,n}\in\{1,-1\}$, $c_{g,n}\in \bC$ be the constants given by Lemma \ref{lem_multiplyDelta}. By Lemma \ref{lem_PolynomialHgn}, we have
$$
H_{g,n}(\alpha+\delta_1)(\delta_{n+1}+q_{g,n}\,\delta_{n+2}+c_{g,n}) \in (J_{g,n+2}^+, (\beta-2)^{k_{g,n+2}}),
$$
hence by the flip symmetry along $S^1\times\{p_1,p_{n+1}\}$,
$$
H_{g,n}(\alpha)(-\delta_{n+1}+q_{g,n}\,\delta_{n+2}+c_{g,n}) \in (J_{g,n+2}^+, (\beta-2)^{k_{g,n+2}}).
$$
On the other hand, Lemma \ref{lem_PolynomialHgn} also gives
$$
H_{g,n}(\alpha)(\delta_{n+1}+q_{g,n}\,\delta_{n}+c_{g,n}) \in (J_{g,n+2}^+, (\beta-2)^{k_{g,n+2}}),
$$
therefore
$$
\delta_{n+1}\,H_{g,n}(\alpha)\in (J_{g,n+2}^+, (\beta-2)^{k_{g,n+2}}).
$$
Since the mapping class group of $(\Sigma,\{p_1,\cdots,p_{n+2}\})$ acts transitively on $\{p_1,\cdots,p_{n+2}\}$, the relation above implies
$$
\delta_i\cdot H_{g,n}(\alpha)\in (J_{g,n+2}^+, (\beta-2)^{k_{g,n+2}}),~\text{ for all }~i=1,\cdots,n+2.
$$
\end{proof}

We can now prove the following result as promised in Remark \ref{rmk_c=0}.

\begin{Lemma}\label{lem_multiplyDeltac=0}
The constant $c_{g,n}$ in Lemma \ref{lem_multiplyDelta} equals zero.
\end{Lemma}
\begin{proof}
By Lemma \ref{lem_PolynomialHgn} and Lemma \ref{lemHDelta},
$$
c_{g,n}H_{g,n}(\alpha)\in (J^+_{g,n+2}, (\beta-2)^{k_{g,n+2}}).
$$
On the other hand, by Lemma \ref{lem_eigenvaluesOfAlphaInV^+} and Corollary \ref{cor_JVSameEigenvalues}, the multiplication by $\alpha$ on 
$$\bC[\alpha,\beta,\gamma,\delta_1,\cdots,\delta_{n+2}]/(J^+_{g,n+2}, (\beta-2)^{k_{g,n+2}})$$ has an eigenvalue that is not a root of $H_{g,n}$. Therefore $c_{g,n}=0$.
\end{proof}

\begin{Lemma}
\label{lem_multiplyP_{g-1,n}P_{g,n-2}}
Suppose $n\ge 3$. We have
\begin{align}
\delta_i \cdot H_{g,n-2}(\alpha)
&\in(J^+_{g,n}, (\beta-2)^{k_{g,n}}),\quad \text{ for } i=1,\cdots, n,
\label{eqnMultiplyByHOnDelta}
\\
(\beta-2) \cdot H_{g,n-2}(\alpha)
&\in(J^+_{g,n}, (\beta-2)^{k_{g,n}}),
\label{eqnMultiplyByHOnBeta-2}
\\
\nonumber
\gamma \cdot P_{g-1,n}(\alpha)
&\in(J^+_{g,n}, (\beta-2)^{k_{g,n}}).
\end{align}
\end{Lemma}

\begin{proof}
The inclusion \eqref{eqnMultiplyByHOnDelta} is given by Lemma \ref{lemHDelta}.
Notice that by \eqref{delta^2+beta},
 $$\beta -2 + \delta_i^2\in J^+_{g,n},$$
 therefore 
\eqref{eqnMultiplyByHOnDelta} implies \eqref{eqnMultiplyByHOnBeta-2}.
If $g\ge 1$, then by Lemma \ref{lem_multiplyGamma} and Corollary \ref{cor_P(alpha)} we have
$$
\gamma\cdot P_{g-1,n}(\alpha) \in (J^+_{g,n}, (\beta-2)^{k_{g,n}}).
$$
The inclusion above also holds for $g=0$ because $\Phi^+(\gamma)=0$ in $\bV^+_{0,n}$ and by definition $P_{-1,n}(\alpha) = 1$.
\end{proof}

\begin{Lemma}\label{lem_QPPInclusion}
Suppose $n\ge 3$, we have
$$
Q_{g,n}(\alpha)\cdot P_{g-1,n}(\alpha)\cdot H_{g,n-2}(\alpha)\in (J^+_{g,n}, (\beta-2)^{k_{g,n}}).
$$
\end{Lemma}

\begin{proof}
By Lemma \ref{lem_Q(alpha)}, 
\begin{multline}\label{eqn_QPPInLongerIdeal}
Q_{g,n}\, P_{g-1,n}\, H_{g,n-2}
\in (J^+_{g,n}, \,\,
(\beta-2)\, P_{g-1,n}\, H_{g,n-2},
\\
\delta_1\, P_{g-1,n}\, H_{g,n-2},\cdots,
\delta_n \, P_{g-1,n}\, H_{g,n-2},\,\,
\gamma \, P_{g-1,n}\, H_{g,n-2}).
\end{multline}
By Lemma \ref{lem_multiplyP_{g-1,n}P_{g,n-2}}, every term on the right hand side of \eqref{eqn_QPPInLongerIdeal} is included in the ideal $(J^+_{g,n}, (\beta-2)^{k_{g,n}})$, therefore the result is proved.
\end{proof}

\begin{Corollary} \label{cor_diagonalizable}
Suppose $n\ge 3$, then the generalized eigenspace of the multiplication by $\alpha$ for the eigenvalue $\lambda_{g+m}$ in $$\bC[\alpha,\beta,\gamma,\delta_1,\cdots,\delta_n]/(J^+_{g,n}, (\beta-2)^{k_{g,n}})$$
coincides with its eigenspace. 
\end{Corollary}

\begin{proof}
The result follows from Lemma \ref{lem_QPPInclusion}, because the multiplicity of the factor $(\alpha-\lambda_{g+m})$ in 
$Q_{g,n}(\alpha)\cdot P_{g-1,n}(\alpha)\cdot H_{g,n-2}(\alpha)$ is $1$. 
\end{proof}

We now state the main theorem of this section.
\begin{Theorem} \label{thm_TopEigenspaceDim1}
Suppose $n\ge 3$, then the simlutaneous generalized eigenspaces of $(\muu(\Sigma_{g,n}),\mu(\pt))$ in $\bV_{g,n}$ for the eigenvalues $(2g+n-2,2)$ and $(-2g-n+2,2)$ have dimension $1$.
\end{Theorem}

\begin{proof}
By Lemma \ref{lem_eigenvaluesOfAlphaInV^+} and the analogous result for $\bV^-_{g,n}$, we only need to show that the simlutaneous generalized eigenspace of $(\muu(\Sigma_{g,n}),\mu(\pt))$ in $\bV^+_{g,n}$ for the eigenvalues $(\lambda_{g+m},2)$ has dimension at most $1$. Notice that by \eqref{delta^2+beta}, in the generalized eigenspace of $\mu(\pt)$ for the eigenvalue $2$ the operator $\sigma_{p_i}$ is nilpotent for all $i$, hence the generalized eigenspace of $\muu(\Sigma_{g,n})$ for the eigenvalue $\lambda_{g+m}$ is the same as the corresponding generalized eigenspace of $\mu(\Sigma_{g,n})$. 
Therefore by Lemma \ref{lem_eigenvaluesOfAlphaInV^+} and the decomposition \eqref{Vgn-decom2}, we only need to show that the generalized eigenspace of the multiplication by  $\alpha$ in 
$$
\bC[\alpha,\beta,\gamma,\delta_1,\cdots,\delta_n]/(J^+_{g,n}, (\beta-2)^{k_{g,n}})
$$
for the eigenvalue $\lambda_{g+m}$ has dimension at most $1$.

By Corollary \ref{cor_diagonalizable}, the generalized eigenspace of the multiplication by $\alpha$ in 
$$
\bC[\alpha,\beta,\gamma,\delta_1,\cdots,\delta_n]/(J^+_{g,n}, (\beta-2)^{k_{g,n}})
$$
for the eigenvalue $\lambda_{g+m}$ coincides with its eigenspace, hence we only need to show 
$$
\dim \bC[\alpha,\beta,\gamma,\delta_1,\cdots,\delta_n]/(J^+_{g,n},\, \alpha-\lambda_{g+m},\, (\beta-2)^{k_{g,n}}) \le 1.
$$
Since $P_{g-1,n}(\lambda_{g+m})\neq 0$, $H_{g,n-2}(\lambda_{g+m})\neq 0$, by Lemma \ref{lem_multiplyP_{g-1,n}P_{g,n-2}},
$$
(\beta-2,\delta_1,\cdots,\delta_n,\gamma)\subset (J^\pm_{g,n},\, \alpha-\lambda_{g+m},\, (\beta-2)^{k_{g,n}}),
$$
therefore 
\begin{multline*}
\dim \bC[\alpha,\beta,\gamma,\delta_1,\cdots,\delta_n]/(J^+_{g,n},\, \alpha-\lambda_{g+m},\, (\beta-2)^{k_{g,n}}) 
\\
\le \dim \bC[\alpha,\beta,\gamma,\delta_1,\cdots,\delta_n]/(\alpha-\lambda_{g+m},\beta-2,\delta_1,\cdots,\delta_n,\gamma) =1,
\end{multline*}
and the result is proved.
\end{proof}

Take a point $q\in \Sigma-\{p_1,\cdots,p_n\}$, let $\omega:=S^1\times \{q\}\subset S^1\times \Sigma$. The flip symmetry at $S^1\times \{p_1\}$ then maps 
$$\II( S^1\times \Sigma, S^1\times \{p_1,\cdots,p_n\},\emptyset)$$ isomorphically to 
$$\II( S^1\times \Sigma, S^1\times \{p_1,\cdots,p_n\},\omega),$$ and the map intertwines with $\muu(\Sigma)$ and $\mu(\pt)$. Therefore we have the following corollary.

\begin{Corollary}
\label{Cor_verticalw2}
Suppose $n$ is odd, and let $\omega$ be given as above. Then the spectrum of 
$\muu(\Sigma)$ in $I(S^1\times \Sigma,S^1\times \{p_1,\cdots,p_n\},\omega)$  in the eigenspace of $\mu(\pt)$ for the eigenvalue $2$ is
$$\{-(2g+n-2),-(2g+n-4),\cdots, 2g+n-4,2g+n-2 \}.$$
Moreover, if $n\ge 3$, the simultaneous generalized eigenspaces of $(\muu(\Sigma),\mu(\pt))$ for the eigenvalues $(-(2g+n-2),2)$ and $(2g+n-2,2)$ have dimension $1$.
\qed
\end{Corollary}

\section{Excision for singular Instanton Floer Homology}
\label{sec_excision}
\begin{Proposition} \label{prop_allEigenvaluesForGeneral3mfld}
Let $Y$ be a closed connected 3-manifold and $L$ a link in $Y$.
Let $\Sigma\subset Y$ be a connected embedded closed surface in $Y$ with genus $g$, and suppose $\Sigma$ intersects $L$ transversely at $n$ points with $n$ odd. Then for every choice of $\omega$, the spectrum of $\muu(\Sigma)$ on $\II(Y,K,\omega)$, in the generalized eigenspace of $\mu(pt)$ for the eigenvalue $2$, is a subset of $\{-(2g+n-2),-(2g+n-4),\cdots, 2g+n-4,2g+n-2 \}$.
\end{Proposition}

\begin{proof}
If $(g,n)=(0,1)$, then the moduli spaces of flat $\SU(2)$ connections on $Y-L$ with the holonomy condition is empty, therefore $\II(Y,K,\omega)=0$, and the statement is vacuously true. For the rest of the proof, assume $(g,n)\neq(0,1)$.

Without loss of generality, assume $\omega$ and $\Sigma$ intersect transeversely.
Let $W:=[0,1]\times Y $ and let $N\subset W$ be a tubular neighborhood of $\{1/2\}\times\Sigma$. Let $\{p_1, \cdots , p_n\}:=\Sigma\cap L$.
Then $N$ is a cobordism from the empty set to $S^1\times \Sigma$, and $W-N$ is a cobordism from $Y\sqcup S^1\times \Sigma$ to $Y$. Let 
$$\varphi_1:\bC\to \II( S^1\times\Sigma ,S^1 \times \{p_1, \cdots , p_n\}, S^1 \times(\omega\cap\Sigma))$$  
be defined by $\varphi_1:=\II(N,N\cap L, N\cap ([0,1]\times \omega))$, and let 
$$\varphi_2: \II( S^1\times\Sigma ,S^1 \times \{p_1, \cdots , p_n\}, S^1 \times(\omega\cap\Sigma)) \otimes \II(Y,L,\omega) \to \II(Y,L,\omega)$$
be defined by $\varphi_2:= (W-N,W-N\cap L, (W-N)\cap ([0,1]\times \omega))$. Both $\varphi_1$ and $\varphi_2$ are defined up to a sign. Since the cobordism map induced by $(W,[0,1]\times \omega)$ is the identity on $\II(Y,L,\omega)$, we have 
$$
\varphi_2 \circ (\varphi_1 \otimes \id_{\II(Y,L,\omega)}) = \pm \id_{\II(Y,L,\omega)}.
$$
On the other hand, by the functoriality of $\II$, the maps $\muu(\Sigma)$ and $\muu(\pt)$ satisfies 
$$
\varphi_2 \circ ( ( \muu(\Sigma)\circ \varphi_1)\otimes \id_{\II(Y,L,\omega)})= \pm \muu(\Sigma),
$$
$$
\varphi_2 \circ ( ( \muu(\pt)\circ \varphi_1)\otimes \id_{\II(Y,L,\omega)})= \pm \muu(\pt),
$$
where the $\mu$ maps on the left hand side are defined on 
$$\II(S^1\times \Sigma,S^1\times \{p_1, \cdots,p_n\},S^1\times (\omega\cap\Sigma) ),$$
 and the $\mu$ maps on the right hand side are defined on $\II(Y,L,\omega)$.
Hence the spectrum of $(\muu(\Sigma),\mu(\pt))$ on $\II(Y,L,\omega)$ is a subset of the spectrum of $(\muu(\Sigma),\mu(\pt))$ on $\II(S^1\times \Sigma,S^1\times \{p_1, \cdots,p_n\},S^1\times (\omega\cap\Sigma) )$. Therefore the result follows from Proposition \ref{lem_AllEigenvalues} and Corollary \ref{Cor_verticalw2}.
\end{proof}

\begin{remark}
Since $\muu(\Sigma)$ only depends on the fundamental class $[\Sigma]\in H_2(Y;\bZ)$, 
Proposition \ref{prop_allEigenvaluesForGeneral3mfld} gives an lower bound for $2g+n$ in terms of the homology class of $\Sigma$ in $H_2(Y;\bZ)$.
\end{remark}

\begin{Definition}
Suppose $(Y,L,\omega)$ is an admissible triple with $Y$ connected.
Let $\Sigma\subset Y$ be an embedded surface in $Y$,  suppose the connected components of $\Sigma$ are $\Sigma_1,\cdots,\Sigma_k$, and suppose $\Sigma_i$ has genus $g_i$ and intersects $L$ transversely at $n_i$ points. 
Define 
$$\II(Y,L,\omega|\Sigma)\subset \II(Y,L,\omega)$$ 
to be the simultaneous generalized eigenspace of $(\muu(\Sigma_1),\cdots,\muu(\Sigma_k),\mu(\pt))$ for the eigenvalues $(2g_1+n_1-2,\cdots,2g_k+n_k-2,2)$ in $\II(Y,L,\omega)$.
\end{Definition}

\begin{Theorem} \label{thm_excision}
Let $Y$ be a closed 3-manifold and $L$ a link in $Y$.
Let $\Sigma_1,\Sigma_2\subset Y$ be two disjoint embedded connected surfaces in $Y$ such that they have the same genus  and both intersect $L$ transversely at $n\ge 3$ points with $n$ odd. 
Let $(\omega,\partial \omega)\subset(Y,L)$ be a $1$-manifold that intersects $\Sigma_1$ and $\Sigma_2$ transversely in its interior at an equal number of points. Let $\varphi:\Sigma_1\to\Sigma_2$ be a diffeomorphism that maps $\Sigma_1\cap L$ to $\Sigma_2\cap L$ and maps $\Sigma_1\cap \omega$ to $\Sigma_2\cap \omega$. Let $(\widetilde{Y}, \widetilde{L}, \tilde{\omega})$ be the resulting triple after cutting $Y$ open along $\Sigma=\Sigma_1\cup \Sigma_2$ and gluing the boundary from cutting $\Sigma_1$ to the boundary from cutting $\Sigma_2$ by the map $\varphi$, let $\widetilde{\Sigma}\subset\widetilde{Y}$ be the image of $\Sigma$ after gluing.  Then
$$
\II(Y,L,\omega|\Sigma) \cong \II(\widetilde{Y},\widetilde{L},\tilde{\omega}| \widetilde{\Sigma}).
$$
\end{Theorem}

\begin{proof}
The idea of the proof is originally due to Floer \cite{floer1990instanton, braam1995floer}, and our argument follows the strategy of \cite[Theorem 7.7]{KM:suture}.

The surgery defines a cobordism $(W_1,S_1,u_1)$ from $(Y,L,\omega)$ to $(\widetilde{Y}, \widetilde{L},\tilde{\omega})$, and a cobordism $(W_2,S_2,u_2)$ from $(\widetilde{Y}, \widetilde{L},\tilde{\omega})$ to $(Y,L,\omega)$. For a detailed construction of the cobordism defined by the excision, the reader may refer to, for example, \cite[Section 3.2]{KM:suture}. Since $\Sigma$ and $\widetilde{\Sigma}$ are homologous in $W_1$, the map $\II(W_1,S_1,u_1)$ takes $\II(Y,L,\omega|\Sigma)$ into $\II(\widetilde{Y},\widetilde{L},\tilde{\omega}| \widetilde{\Sigma})$. Similarly, the map $\II(W_2,S_2,u_2)$ takes $\II(\widetilde{Y},\widetilde{L},\tilde{\omega}| \widetilde{\Sigma})$ into $\II(Y,L,\omega|\Sigma)$.

Let $(W,S,u)$ be the composition of $(W_1,S_1,u_1)$ and $(W_2,S_2,u_2)$, then 
$$\II(W,S,u) = \II(W_2,S_2,u_2)\circ \II(W_1,S_1,u_1).$$ 
Let $\Sigma_g$ be a surface diffeomorphic to $\Sigma_1$, and let $L_\Sigma$, $\omega_\Sigma\subset \Sigma_g$ be the images of $L\cap\Sigma_1$ and $\omega\cap \Sigma_1$ respectively under the diffeomorphism.
The cobordism $(W,S,u)$ has the following property: 
there exists an embedded $S^1\times \Sigma_g\subset W$, such that its intersection with $S$ is $S^1\times L_\Sigma$, and its intersection with $u$ is $S^1\times \omega_\Sigma$. If one cuts $(W,S,u)$ open along $(S^1\times\Sigma_g, S^1\times L_\Sigma,S^1\times \omega_\Sigma)$, and fill the boundary with $( (D^2\sqcup D^2)\times\Sigma_g , (D^2\sqcup D^2)\times L_\Sigma,(D^2\sqcup D^2)\times \omega_\Sigma)$, then the resulting cobordism is the product triple $([0,1]\times Y, [0,1]\times L,[0,1]\times \omega)$.

There are two maps 
$$
\varphi_1,\varphi_2:\II(S^1\times \Sigma_g,  S^1\times L_\Sigma,S^1\times \omega_\Sigma|\Sigma_g)\otimes \II(S^1\times \Sigma_g,  S^1\times L_\Sigma,S^1\times \omega_\Sigma|\Sigma_g)\to \bC.
$$
The map $\varphi_1$ is defined by 
$$\varphi_1:=\II([0,1]\times S^1 \times \Sigma_g,[0,1]\times S^1 \times  L_\Sigma, [0,1]\times S^1 \times \omega_\Sigma)$$
where $[0,1]\times S^1 \times \Sigma_g$ is considered as a cobordism from $S^1 \times \Sigma_g\sqcup S^1 \times \Sigma_g$ to the empty set.
The map $\varphi_2$ is defined by 
$$\varphi_2:=\II((D^2\sqcup D^2)\times \Sigma_g , (D^2\sqcup D^2)\times  L_\Sigma, (D^2\sqcup D^2)\times \omega_\Sigma).$$
The signs of $\varphi_1$ and $\varphi_2$ can be fixed by taking canonical almost complex structures on the cobordisms, although we do not need it in the current proof.
We claim that $\varphi_1$ and $\varphi_2$ differ by a non-zero multiplicative factor. 

In fact, by Theorem \ref{thm_TopEigenspaceDim1} and Corollary \ref{Cor_verticalw2},
$$
\dim_\bC \II( S^1\times \Sigma_g,  S^1\times L_\Sigma, S^1\times \omega_\Sigma|\Sigma_g)=1,
$$ 
therefore we only need to show that both $\varphi_1$ and $\varphi_2$ are non-zero. The map $\varphi_1$ is not zero because the cobordism is diffeomorphic to a trivial product, thus the map is dual to the identity map. 
Let $\varphi_3$ be the cobordism map defined by the product of the pair-of-pants cobordism from $S^1\sqcup S^1$ to $S^1$ with $(\Sigma_g, L_\Sigma,\omega_\Sigma)$, let $\varphi_4$ be the cobordism map defined by the product of the disk cobordism from $S^1$ to $\emptyset$ with $(\Sigma_g, L_\Sigma,\omega_\Sigma)$, then $\varphi_1=\varphi_4\circ\varphi_3$, therefore $\varphi_4\neq 0$. On the other hand, $\varphi_2=\varphi_4\otimes\varphi_4$, hence $\varphi_2\neq 0$.

It then follows from the formal properties of singular instanton Floer homology that the restriction of $\II(W,S,u)=\II(W_2,S_2,u_2)\circ \II(W_1,S_1,u_1)$ to $\II(Y,L,u|\Sigma)$ equals a non-zero constant multiplication of $\id_{\II(Y,L,\omega|\Sigma)} $. Similarly, the restriction of the map $\II(W_1,S_1,u_1)\circ \II(W_2,S_2,u_2)$ to $\II(\widetilde{Y},\widetilde{L},\tilde{\omega}|\widetilde{\Sigma})$ equals a non-zero constant multiplication of $\id_{\II(\widetilde{Y},\widetilde{L},\tilde{\omega}|\widetilde{\Sigma})}$. Therefore $\II(Y,L,\omega|\Sigma)\cong \II(\widetilde{Y},\widetilde{L},\tilde{\omega}|\widetilde{\Sigma})$.
\end{proof}

\begin{Proposition}\label{inst-fiber-mfd}
Let $n\ge 3$ be an odd integer, let $\Sigma$ be a surface with genus $g$, let $\{p_1,\cdots, p_n,u_0\}\subset \Sigma$.
Let $f:\Sigma\to\Sigma$ be a diffeomorphism such that $f(u_0)=u_0$ and $f(\{p_1,\cdots, p_n\}) = \{p_1,\cdots, p_n\}$. Let 
$$M_f=[0,1]\times \Sigma/(0,x)\sim (1,f(x))$$
be the mapping torus of $f$, let $L_f$ be the image of $[0,1]\times \{p_1,\cdots,p_n\}$ in $M_f$, and let $u_f$ be the image of $[0,1]\times u_0$ in $M_f$. Then
$$
\II(M_f,L_f,\emptyset|\Sigma)\cong \mathbb{C}, \quad
\II(M_f,L_f,u_f|\Sigma)\cong \mathbb{C}.
$$
\end{Proposition}

\begin{proof}
The proof follows from the same argument as \cite[Lemma 4.7]{KM:suture}.

Let $f, g :(\Sigma_g, \{p_1,\cdots, p_n\},\{u_0\}) \to (\Sigma_g, \{p_1,\cdots, p_n\},\{u_0\})$ be two diffeomorphisms. Apply Theorem \ref{thm_excision} to $M_f\sqcup M_g$ yields 
\begin{align*}
\II(M_f,L_f,\emptyset |\Sigma) \otimes \II(M_g,L_g,\emptyset |\Sigma) &\cong \II(M_{f\circ g},L_{f\circ g},\emptyset |\Sigma),
\\
\II(M_f,L_f,u_f |\Sigma) \otimes \II(M_g,L_g,u_g |\Sigma) &\cong \II(M_{f\circ g},L_{f\circ g},u_{f\circ g} |\Sigma).
\end{align*}
Let $g=f^{-1}$, it follows from Theorem \ref{thm_TopEigenspaceDim1} and Corollary \ref{Cor_verticalw2} that 
\begin{align*}
\II(M_f,L_f,\emptyset |\Sigma) \otimes \II(M_{f^{-1}},L_{f^{-1}},\emptyset |\Sigma) &\cong \bC,
\\
\II(M_f,L_f,u_f |\Sigma) \otimes \II(M_{f^{-1}},L_{f^{-1}},u_{f^{-1}} |\Sigma) &\cong \bC,
\end{align*}
therefore $\II(M_f,L_f,\emptyset |\Sigma)\cong \II(M_f,L_f,u_f |\Sigma) \cong \bC$.
\end{proof}

\begin{Theorem}
\label{thm_excisionAlongOneSurface}
Let $Y$ be a closed 3-manifold and $L$ a link in $Y$.
Let $\Sigma\subset Y$ be an embedded connected surfaces in $Y$ such that it intersects $L$ transversely at $n\ge 3$ points with $n$ odd. Let $(\omega,\partial \omega)\subset(Y,L)$ be a $1$-manifold that intersects $\Sigma$ transversely in its interior. Let $\varphi:\Sigma\to\Sigma$ be a diffeomorphism that maps $\Sigma\cap L$ to $\Sigma\cap L$ and maps $\Sigma\cap \omega$ to $\Sigma\cap \omega$. Let $(\widetilde{Y}, \widetilde{L}, \widetilde{\Sigma},\tilde{\omega})$ be the resulting manifolds after cutting $Y$ open along $\Sigma$ and gluing back using the map $\varphi$. Then
$$
\II(Y,L,\omega|\Sigma) \cong \II(\widetilde{Y},\widetilde{L},\tilde{\omega}| \widetilde{\Sigma}).
$$
\end{Theorem}

\begin{proof}
Let $\Sigma'$ be a parallel copy of $\Sigma$. Let $(M_\varphi,L_\varphi,\omega_\varphi)$ be the mapping torus of $\varphi$ on $(\Sigma,\Sigma\cap L,\Sigma\cap \omega)$. Apply Theorem \ref{thm_excision} to $(Y,\Sigma\cup\Sigma')$, we have
$$
\II(Y,L,\omega|\Sigma) \cong \II(\widetilde{Y},\widetilde{L},\tilde{\omega}| \widetilde{\Sigma}) \otimes \II(M_\varphi,L_\varphi,\omega_\varphi|\Sigma).
$$
The result then follows from Proposition \ref{inst-fiber-mfd}.
\end{proof}

\begin{Proposition}\label{horizontal-w2}
Let $n\ge 3$ be an odd integer, suppose $\{p_1,\cdots,p_n\}\subset\Sigma$, let $Y:=S^1\times \Sigma$, $L:=S^1\times \{p_1,\cdots,p_n\}$. Let $x_0\in S^1$, $u_0\in \Sigma-\{p_1,\cdots,p_n\}$. Let $(c,\partial c)\subset (\Sigma, \{p_1,\cdots,p_n\})$ be a properly embedded $1$-manifold, let $\omega:=\{x_0\}\times c$, $u:= S^1\times u_0$. Then
\begin{align*}
\II(Y,L,\omega|\Sigma)&\cong \bC,
\\
\II(Y,L,\omega+u|\Sigma)&\cong \bC.
\end{align*}
\end{Proposition}

\begin{proof}
The proof follows the argument of \cite[Proposition 7.8]{KM:suture}.
Let $x_1$ be a point on $S^1$ that is distinct from $x_0$.
Notice that cutting $Y\sqcup Y$ open along $\{x_1\}\times\Sigma \sqcup\{x_1\}\times\Sigma $ and glue back using the identity map on $\Sigma$ yields $Y$. Therefore by Theorem \ref{thm_excision},
\begin{align*}
\II(Y,L,\omega|\Sigma)\otimes \II(Y,L,\omega|\Sigma)&\cong \II(Y,L,\emptyset|\Sigma),
\\
\II(Y,L,\omega+u|\Sigma)\otimes \II(Y,L,\omega+u|\Sigma)&\cong \II(Y,L,u|\Sigma).
\end{align*}
By Theorem \ref{thm_TopEigenspaceDim1} and Corollary \ref{Cor_verticalw2}, $\II(Y,L,\emptyset|\Sigma)\cong \II(Y,L,u|\Sigma)\cong \bC$, therefore $\II(Y,L,\omega|\Sigma)\cong \II(Y,L,\omega+u|\Sigma)\cong \bC$.
\end{proof}

\section{Instanton Floer homology for sutured manifolds with tangles}
\label{sec_suturedFloer}
We first recall the definition of a balanced sutured manifold from \cite{Juhasz-holo-disk,KM:suture}, which is an adjustment of Gabai's
notion of sutured manifold \cite{G:Sut-1}.
\begin{Definition}\label{Def-sutured-mfd}
A \emph{balanced sutured manifold} is a compact oriented 3-manifold $M$ together with a collection of oriented circles 
$\gamma\subset\partial M$ and a decomposition 
\begin{equation}\label{partial-M-decom}
\partial M=A(\gamma)\cup R^+(\gamma)\cup R^-(\gamma),
\end{equation}
 where $R^+(\gamma)$ is oriented by the same orientation as $\partial M$ and $R^-(\gamma)$ is oriented by the opposite orientation of $\partial M$, 
which satisfy the following conditions:
\begin{itemize} 
\item Suppose $\{\gamma_i\}$ are the connected components of $\gamma$, then $A(\gamma)$ is the disjoint union of $A(\gamma_i)$, where $A(\gamma_i)$ is a closed tubular neighborhood of $\gamma_i$; \item
$R(\gamma):=R^+(\gamma)\cup R^-(\gamma)$ is the closure of $\partial M- A(\gamma)$; 
\item $\partial R(\gamma)=\partial A(\gamma)$ as oriented manifolds, 
where $\partial R(\gamma)$ carries the boundary orientation and $\partial A(\gamma)$ is oriented in the same way 
as $\gamma$;
\item $M$ and $R(\gamma)$ have no closed components;
\item $\chi(R^+(\gamma))=\chi(R^-(\gamma))$. 
\end{itemize}
\end{Definition}

In the following, we will write $R^\pm(\gamma)$ as $R^\pm$ if the suture $\gamma$ is clear from the context.
The next lemma is a well-known property of sutured manifolds. We state it here for later reference.

\begin{Lemma}\label{lem_boundaryHasOneCircle}
Let $(M,\gamma)$ be a sutured manifold. Then there exists a surface $F$, such that 
\begin{enumerate}
\item $(F,\partial F) \subset (M,A(\gamma))$;
\item $[(F,\partial F)]$ is homologous to $[(R^+, \partial R^+)]$ in $H_2(M,A(\gamma))$, and $F$ minimizes the Thurston norm in the homology class $[(R^+, \partial R^+)]$.
\item For every $\gamma_i$ that is a connected component of $\gamma$, the intersection $\partial F\cap A(\gamma_i)$ consists of one circle that is parallel to $\gamma_i$ and has the same orientation as $\gamma_i$.
\end{enumerate}

\end{Lemma}

\begin{proof}
Let $F$ be a surface such that $\partial F$ has the minimum number of connected components among the surfaces that satisfy the the first two conditions in the lemma.  
Let $\gamma_i$ be a connected component of $\gamma$, and let $A(\gamma_i)$ be the associated annulus. Notice that $A(\gamma_i)\cap \partial F$
is a collection of circles that are either parallel to $\gamma_i$ or null-homologous in $A$. If at least one of the circles is null-homologous, we can reduce the number of connected components of $\partial F$ by attaching a disk along an inner most null-homologous circle to $F$. If all of the circles are parallel to the suture and at least two of them have opposite orientations, we can reduce the number of connected components of $\partial F$ by
joining a pair of adjacent and oppositely oriented circles using an annulus. Since the operations above do not increase the Thurston norm of $F$, we conclude that  $A(\gamma_i)\cap \partial F$ can only consist of circles that are parallel to the suture and have the same orientation. Since $[(F,\partial F)]$ is homologous to $[(R^+, \partial R^+)]$ in $H_2(M, A(\gamma_T))$, it follows that $A(\gamma_i)\cap \partial F$ consists of exactly one circle that is parallel to $\gamma_i$ and has the same orientation as $\gamma_i$.
\end{proof}

Recall the the following definition from Section \ref{sec_intro}.

\begin{Definition}
Let $(M,\gamma)$ be a balanced sutured manifold with the decomposition \eqref{partial-M-decom} as above. 
A \emph{tangle} $T$ in $M$ is a properly embedded 1-manifold with $\partial T\subset R^+\cup R^-$. 
A tangle $T$ in $M$ is called \emph{balanced} if $|T\cap R^+|=|T\cap R^-|$.
A tangle $T$ in $M$ is called \emph{vertical} if $T$ has no closed components and every connected component of $T$
intersects both $R^+$ and $R^-$. 
\end{Definition}

Now we review some notions introduced in \cite{Schar}, which generalize the original definitions of Gabai \cite{G:Sut-1}. 
Let $T$ be a tangle in the sutured manifold $(M,\gamma)$ as above. Suppose $(S,\partial S)\subset (M,A(\gamma))$ is a connected surface
which is transverse to $T$, define
\begin{equation}
x_T(S):=\max \{0,|S\cap T|-\chi(S) \}.
\end{equation}
Notice that if $S$ is a closed connected surface of genus $g$ that intersects $T$ at $n$ points, then $x_T(S)=\max\{2g+n-2,0\}$.
If $S$ is not connected, then we define
\begin{equation}
x_T(S):=\sum_i x_T(S_i),
\end{equation}
where the $S_i$'s are the connected components of $S$. Given a homology class $a\in H_2(M,A(\gamma))$, its \emph{Thurston norm} is defined by
\begin{equation*}
x_T(a):=\min \{x_T(S)|(S,\partial S)\subset (M,A(\gamma)), [(S,\partial S)]=a\}.
\end{equation*}
We say that a surface $S$ is $T$-\emph{norm-minimizing} if 
\begin{equation*}
x_T(S)=x_T[(S,\partial S)].
\end{equation*}
When $T=\emptyset$, we simply denote $x_T$ by $x$.

\begin{Definition}[\cite{Schar}]\label{Def-taut}
Let $T$ be a tangle in a balanced sutured manifold $(M,\gamma)$.
\begin{itemize}
\item $M$ is called \emph{T-irreducible} if $M- T$ is irreducible, i.e. any 2-sphere in $M- T$ bounds a 3-ball in $M- T$.

\item A properly embedded surface $(S,\partial S)\subset (M,A(\gamma))$ is called \emph{incompressible} if 
$S- T$ is incompressible in $M- T$. i.e. any circle in $S- T$ that bounds a disk in 
$M- T$ must also bound a disk in $S- T$.

\item The triple $(M,\gamma, T)$ is called \emph{taut} if (1) each connected component of $T$ intersects 
$R=R^+\cup R^-$ with the same sign,
(2) $M$ is $T$-irreducible,
(3) $R$ is $T$-incompressible, 
and (4) $R^+$ and $R^-$ are both $T$-norm-minimizing,  

\end{itemize}
\end{Definition}

\begin{remark}
Condition (1) in the third part of Definition \ref{Def-taut} is the original formulation in \cite[Definition 1.2]{Schar}. Since the condition does not depend on the choice of orientation of $T$, the definition is valid for an unoriented tangle, and it is equivalent to the condition that every non-closed component of $T$ intersects both $R^+$ and $R^-$.
\end{remark}

Suppose $T$ is a balanced tangle in a balanced sutured manifold $(M,\gamma)$. Take an oriented connected
surface $F$ with $n$ ($n\ge 2$) marked points $\{p_1,\cdots,p_n\}$, and suppose there is 
an orientation-reversing diffeomorphism from $\gamma$ to $\partial F$.
Let $[-1,1]\times F$ be the product sutured manifold with sutures $\{0\}\times \partial F$, let
$$T_n:=[-1,1]\times \{p_1,\cdots,p_n\}\subset [-1,1]\times F$$
be the product tangle, and let $u$ be an arc joining $p_1$ and $p_2$ on $\{0\}\times F$. The diffeomorphism from $\gamma$ to $\partial F$ extends to a diffeomorphism from $A(\gamma)$ to $[-1,1]\times \partial F$ that maps $\{1\}\times\partial F$ to $A(\gamma)\cap R^+(\gamma)$, and maps $\{-1\}\times\partial F$ to $A(\gamma)\cap R^-(\gamma)$.
Define
\begin{equation*}
M':=M\cup_{A(\gamma)} [-1,1]\times F, ~ T':=T\cup T_n,
\end{equation*}
and
\begin{equation*}
\bar{R}^\pm:=R^\pm(\gamma)\cup \{\pm 1\}\times F.
\end{equation*}
Now $M'$ is a manifold with two boundary components $\bar{R}^\pm$. Take a diffeomorphism
\begin{equation*}
f:(\bar{R}^+, T'\cap \bar{R}^+)\to (\bar{R}^-, T'\cap \bar{R}^-)
\end{equation*}
which fixes $p_1,p_2$. Close up the boundary of $M'$ using the map $f$, we obtain an admissible triple 
$(Y(M,\gamma),L(T),u):=(M'/f(x)\sim x, T'/f(x)\sim x, u).$
We use $\bar{R}$ to denote the surface in $Y({M},\gamma)$ given by the image of $\bar{R}^\pm$.
\begin{Definition}
\label{def_SHI}
For a balanced sutured manifold $(M,\gamma)$ with a balanced tangle $T$, let $(Y(M,\gamma),L(T),u)$ and $\bar{R}$ be given as above. The instanton Floer homology of $(M,\gamma,T)$ is defined by
\begin{equation*}
\SHI(M,\gamma, T):=\II(Y(M,\gamma),L(T),u|\bar{R}).
\end{equation*} 
\end{Definition}

\begin{Proposition}
The isomorphism class of $\SHI(M,\gamma, T)$ does not depend on the choice of the surface $F$, the tangle $T_n$ or the diffeomorphism $f$. 
\end{Proposition}
\begin{proof}
First we prove that $\SHI(M,\gamma, T)$ does not depend on the genus of $F$ and the number of marked points on $F$.
Let $\mathbb{U}_{g,n}$ be the Floer homology of the triple 
$(S^1\times \Sigma_{g,n},S^1\times \{q_1,\cdots,q_n\},\omega')$ defined in Section \ref{subsection-spec-existence},
where $\omega'$ is an arc joining $S^1\times \{q_1\}$ and $S^1\times \{q_2\}$. 
Take a small circle $c$ around $p_1$ on $F$ and assume $f$ fixes $c$. Then $[-1,1]\times c$ becomes a torus $S^1\times c$ in $Y(M,\gamma)$
which intersects $u$ at one point.
Let $c'$ be a small circle around $q_1$ on $\Sigma_{g,n}$. Then $S^1\times c'$ is a torus which intersects $\omega'$ at one point.
The excision on 
$$(Y(M,\gamma),L(T),u)\sqcup  (S^1\times \Sigma_{g,n},S^1\times \{q_1,\cdots,q_n\},\omega')$$ along $S^1\times c$ and $S^1\times c'$ yields the disjoint union 
$$(S^1\times S^2, S^1\times \{r_1,r_2\},\omega'')\sqcup (Y',L',u'),$$
where $r_1,r_2\in S^2$, $\omega''$ is an arc connecting $S^1\times \{r_1\}$ and $S^1\times \{r_2\}$, and $(Y',L',u')$ is diffeomorphic to the closure obtained by increasing the genus of $F$ by $g$ and the number of marked points on $F$
by $n-2$. 
By Theorem \ref{thm_nonsingularExcision} and Lemma \ref{Lem-UgnWgn-eigenvalues}, we have an isomorphism
\begin{equation}\label{iso-F+Ugn}
\II(Y(M,\gamma),L(T),u)\otimes \mathbb{U}_{g,n}\cong \II(Y',L',u)\otimes \mathbb{U}_{0,2}
\cong \II(Y',L',u),
\end{equation}
Let $\bar R'$ be the image of $\bar R\cup \Sigma_{g,n}$ in $Y'$.
Since the isomorphism \eqref{iso-F+Ugn} intertwines 
$\muu(\bar{R})\otimes 1+1\otimes \muu(\Sigma_{g,n})$ with $\muu(\bar{R}')$, and by Lemma \ref{Lem-UgnWgn-eigenvalues} the top eigenspace 
of $\muu(\Sigma_{g,n})$ is 1-dimensional in $\bU_{g,n}$, we have
$$
\II(Y(M,\gamma),L({T}),u|\bar{R})\cong \II(Y',L',u|\bar{R}').
$$ 
In summary, we can change the genus of $F$ and the number of marked points on $F$ without changing the isomorphism class of $\SHI(M,\gamma,T)$.

If $\bar{R}$ intersects $\overline{T}$ at $m$ points where $m\ge 3$ and $m$ is odd, then the independence on $f$ follows from
Theorem \ref{thm_excisionAlongOneSurface}. In general, we can always increase the number of marked points on $F$ to make $m$ odd and at least $3$. 
\end{proof}

\begin{remark}
When $T=\emptyset$, our definition of $\SHI(M,\gamma,\emptyset)$ coincides with Kronheimer and Mrowka's instanton Floer homology $\SHI(M,\gamma)$
in \cite{KM:suture}. This can be seen by using a surface $F$  with two marked points $\{p_1, p_2\}$ in Definition \ref{def_SHI}, and
conducting an excision along the boundary tori of the neighborhoods of $S^1\times \{p_1\}$ and $S^1\times \{p_2\}$ in $Y(M,\gamma)$.
The result of the excision is the disjoint union of a mainfold that defines $\SHI(M,\gamma)$ in \cite{KM:suture} and a triple that defines $\bU_{0,2}$.
\end{remark}

Suppose $T$ is a balanced tangle on a balanced sutured manifold $(M,\gamma)$ such that each component of $T$ intersects $R^\pm$ with
the same sign. This implies $T$ is the union of a vertical tangle 
and a link in $(M,\gamma)$. From now on, we will use the notation $N(\cdot)$ to denote tubular neighborhoods. Define $M_T:=M- N(T)$ and define a suture $\gamma_T$ on $\partial M_T$ to be the union of
\begin{itemize}
\item the original sutures $\gamma$;
\item a meridian suture $s(t_i)$ on $\partial N(t_i)$ for each non-closed component $t_i$ of $T$;
\item two oppositely-oriented meridian sutures $s^\pm(l_i)$ 
on $\partial N(l_i)$ for each closed component $l_i$ of $T$. 
\end{itemize}
Moreover, if $s(t_i)$ is the suture given by a non-closed component $t_i$ of $T$, we require that 
$A(t_i)$ is the closure of $\partial N(t_i)-\partial M$.
We have a decomposition
$$
\partial M_T=A(\gamma_T)\cup R^+(\gamma_T) \cup R^-(\gamma_T),
$$
where $R^+(\gamma_T)$ is the union of 
\begin{itemize}
\item An annulus $R^+_T(l_i)$ for each closed component $l_i$ of $T$; 
\item $S^+:=R^+(\gamma)- \bigcup_{t_i}D(t_i)$ where $D(t_i)$  is the interior of $\partial N(t_i)\cap \partial M$. 
\end{itemize}
A similar description works for $R^-(\gamma_T)$. 

\begin{Lemma}\label{M_T-taut}
If $(M,\gamma,T)$ is taut, then
the sutured manifold $(M_T,\gamma_T)$ defined above is taut.
\end{Lemma}
\begin{proof}
To simplify notations, we will use $R^\pm$ to denote $R^\pm(\gamma)$ and use $R^\pm_T$ to denote $R^\pm(\gamma_T)$.
The irreducibility of $M_T$ follows directly from Definition \ref{Def-taut}. Recall that $x_T$ is defined to be the Thurston norm with respect to $T$, and when $T$ is empty we simply write $x_T$ as $x$. We have
$$
x(R^+_T(l_i))=0,~x(S^+)=x_T(R^+).
$$
We show that $R^+_T$ is norm-minimizing. Suppose there is a properly embedded 
surface $(F,\partial F)\subset (M_T, A(\gamma_T))$ with 
$$
[(F,\partial F)]=[(R^+_T, \partial R^+_T)]\in H_2(M_T, A(\gamma_T))
$$
such that 
$$
x(F)<x(R^+_T)=x(S^+).
$$
By Lemma \ref{lem_boundaryHasOneCircle}, we can choose the surface $F$ such that for each annulus 
 $A(\gamma_i)$ where $\gamma_i$ is a component of the suture $\gamma_T$, the intersection $A(\gamma_i)\cap \partial F$ consists of a circle that is parallel to $\gamma_i$ and has the same orientation as $\gamma_i$.

For each closed component $l_i$ of $T$, the previous argument shows that $\partial F\cap \partial N(l_i)$ consists of two circles. The two circles decompose $\partial N(l_i)$ into two annuli, which we label as $J^+(l_i)$ and $J^-(l_i)$ such that $J^\pm(l_i)$ is homologous to $R_T^\pm(l_1)$ in $H_2(M_T, A(\gamma_T))$ respectively.
 Attach $J^-(l_i)$ to $F$ for all closed components $l_i$ of $T$, and attach a disk in $N(t_i)$ along $A(s(t_i))\cap \partial F$ for all vertical components $t_i$ of $T$, we obtain a surface $F'$ that satisfies $x_T(F')\le x(F)$, and
$$
[(F',\partial F')]=[(R^+,\partial R^+)]\in H_2(M,A(\gamma)).
$$
Since
$$
x_T(F')\le x(F)< x(S^+)=x_T(R^+)
$$
this contradicts the tautness assumption of $(M,\gamma,T)$. 

Now we prove $R^+_T$ is incompressible. Suppose there is a compressing disk
$C\subset M_T$ such that $\partial C\subset R^+_T$. The boundary of $C$ cannot be on $S^+$ by the tautness assumption of $(M,\gamma,T)$.
Therefore $\partial C$ is the meridian of $R^+_T(l_i)$ for a closed component $l_i$ of $T$. Attaching a disk in $N(l_i)$ to $C$ 
along $\partial C$ gives a 2-sphere $H\subset M$ which intersects $T$ at a single point. Let $N(H\cup l_i)$ be a regular neighborhood of $H\cup l_i$, then $S^2\cong \partial N(H\cup l_i)\subset M- T$. The $T$-irreducibility assumption 
of $(M,\gamma,T)$ implies that $\partial N(H\cup l_i)$ bounds a 3-ball $B^3\subset M- T$. Therefore 
$N(H\cup l_i)\cup B^3$ forms a closed component of $M$, which contradicts Definition \ref{Def-sutured-mfd}.

Apply the same argument to $R_T^-$ shows that $(M_T,\gamma_T)$ is taut.
\end{proof}

\begin{Lemma}\label{SHI-vertical-tangle}
Suppose $T$ is a vertical tangle in a balanced sutured manifold $(M,\gamma)$, let $(M_T,\gamma_T)$ be defined as above. Then we have
$$
\SHI(M,\gamma,T)\cong \SHI(M_T,\gamma_T).
$$
\end{Lemma}
\begin{proof}
As in Lemma \ref{M_T-taut}, we use $R^\pm$ to denote $R^\pm(\gamma)$ and use $R^\pm_T$ to denote $R^\pm(\gamma_T)$.

Suppose $T$ has $m$ components $t_1,t_2,\cdots,t_{m}$. 
Let $F$ be the auxiliary surface used in the definition of $\SHI(M,\gamma,T)$. Let $p_1,p_2,\cdots,p_{m+2}, q$ be $(m+3)$ marked points on $F$ and let $u$ be 
an arc joining $p_{m+1}$ and $p_{m+2}$.
Let $T_{m+3}:=[-1,1]\times\{p_1,\cdots,p_{m+2},q\}$, and let $T_{m+2}:=[-1,1]\times\{p_1,\cdots,p_{m+2}\}$. Let $(Y,L)$ be the pair obtained by closing up
$(M\cup_{A(\gamma)} [-1,1]\times F, T\cup T_{m+3})$. We also assume that the closing map preserves the components of $T\cup T_{m+3}$, so that $L$ can be written as
$$
L=\overline{T}\cup \overline{T}_{m+3}=\bigcup_{i=1}^m \bar{t}_i \cup S^1\times \{p_1,\cdots,p_{m+2},q\},
$$
where $\overline{T}$ is the image of $T$, and $\bar{t}_i$ are the images of the components of $T$, and $\overline{T}_{m+3}$ is the image of $T_{m+3}$. Let $\overline{T}_{m+2}$ be the image of $T_{m+2}$ in $Y$. Let $\bar R\subset Y$ be the image of the closed-up boundary.
By definition, we have
$$
\SHI(M,\gamma,T)=\II(Y, L,u|\bar{R}).
$$

Consider the admissible triple $(S^1\times \Sigma_g, S^1\times \{q_1,\cdots,q_{2m+2},q\}, \omega)$ where $g$ equals the genus of $\bar{R}$, and $\Sigma_g$ is a closed surface of genus $g$, and
$\omega$ is a collection of arcs on $\{\pt\}\times \Sigma_g$ joining the pairs 
$$(q_1,q_{m+1}),(q_2,q_{m+2}),\cdots, (q_{m},q_{2m}).$$
According to
Proposition \ref{horizontal-w2}, we have
\begin{equation}\label{w2-pairing-markings}
\II(S^1\times \Sigma_g, S^1\times \{q_1,\cdots,q_{2m+3}\}, \omega|\Sigma_g)\cong\mathbb{C}.
\end{equation}
Apply excision to 
$$
(Y, L,u)\sqcup (S^1\times \Sigma_g, S^1\times \{q_1,\cdots,q_{2m+3}\}, \omega) 
$$
along $\bar{R}$ and $\Sigma_g$, by Theorem \ref{thm_excision} and \eqref{w2-pairing-markings} we obtain
\begin{equation}\label{YLu=YLw'+u}
\II(Y, L,u|\bar{R})\cong \II(Y, L,\omega'+u|\bar{R}),
\end{equation}
where $\omega'$ 
is a collection of arcs in $\bar{R}$ 
joining the pairs $(\bar{t}_1,S^1\times \{p_1\}),\cdots, (\bar{t}_{m},S^1\times \{p_{m}\})$. 

Let $c$ be the boundary of a disk in $F$ which contains $p_{m+2}$ and $q$ and is disjoint from the other marked points.
For $n\ge 2$, let $\mathbb{Y}_{g,n}$ be the triple given by $(S^1\times \Sigma_g,S^1\times \{r_1,\cdots,r_n\},s)$, where $r_1,\cdots,r_n\in \Sigma_g$, and $s$ is an arc connecting $S^1\times \{r_1\}$ to $S^1\times \{r_2\}$.
Recall from Section \ref{subsection-spec-existence} that $\bU_{g,n}=\II(\mathbb{Y}_{g,n})$. 
Apply torus excision to $(Y, L,\omega'+u) \sqcup \mathbb{Y}_{0,2}$ along
$S^1\times c$ and the boundary torus of a neighborhood of $S^1\times\{r_1\}$ in $\mathbb{Y}_{0,2}$, we obtain
\begin{equation*}
\II(Y, L,\omega'+u)\otimes \mathbb{U}_{0,2}\cong 
\II(Y, \overline{T}\cup \overline{T}_{m+2},\omega'+u)\otimes \mathbb{U}_{0,3}.
\end{equation*}
The above isomorphism intertwines $\muu(\bar{R})\otimes 1 +1\otimes \muu(\Sigma_{0,2})$ on the left with
$\muu(\bar{R})\otimes 1+1\otimes \muu(\Sigma_{0,3})$ on the right. The result on the eigenspaces of $\mathbb{U}_{g,n}$
in Lemma \ref{Lem-UgnWgn-eigenvalues} implies
\begin{equation}\label{YL=YL'}
\II(Y, L,\omega'+u|\bar{R})\cong 
\II(Y,L' ,\omega'+u|\bar{R})
\end{equation}
where $L':=\overline{T}\cup \overline{T}_{m+2}$.

Now apply excision to $(Y, L',\omega'+u)$ along the pairs $(\partial N(S^1\times \{p_{m+1}\}), \partial N(S^1\times \{p_{m+2}\}))$ and 
$(\partial N(\bar{t}_i),\partial N(S^1\times \{p_{i}\}))$ ($1\le i \le m$),
the resulting manifold  is the union of copies of $\mathbb{Y}_{0,2}$ and a triple which we denote by $(Y_1,\emptyset,\bar{\omega}'+\bar{u})$.
We have
\begin{equation}\label{Y=Y_1}
\II(Y, L',\omega'+u|\bar{R})\cong \II(Y_1,\emptyset,\bar{\omega}'+\bar{u}|S)
\end{equation}
where $S$ is the surface with genus $g+m+1$ obtained from $\bar{R}$ by excision.

It is straightforward to see that $Y_1$ is the closure of $(M_T,\gamma_T)$ by 
an auxiliary surface $G$ which is obtained from $F$ by
\begin{itemize}
\item removing the disks $N(p_{m+1}),N(p_{m+2})$ and attaching a handle along the boundary circles;
\item removing the disks $N(p_i)$ ($1\le i \le m$).
\end{itemize}
The boundary $\partial N(p_i)$ ($1\le i \le m$) of $G$ is paired with the meridian suture of $t_{i}$.
 
It was proved in \cite[Section 7.4]{KM:suture} that  
\begin{equation*}
\II(Y_1,\emptyset,\bar{\omega}'+\bar{u}|S)\cong \SHI(M_T,\gamma_T),
\end{equation*}
hence by \eqref{YLu=YLw'+u}, \eqref{YL=YL'} and \eqref{Y=Y_1}, we have 
$\SHI(M_T,\gamma_T)\cong \SHI(M,T,\gamma).$
\end{proof}

The following result is a generalization of \cite[Theorem 3.4.4]{Street}. The original statement 
needs to assume that the surface $\bar{R}$ in Definition \ref{def_SHI} has genus $0$. 
\begin{Theorem}\label{braid-detection}
Suppose $T$ is a vertical tangle
in a  sutured manifold  $(M,\gamma)$ which is a homology product, i.e. 
the inclusions $R^+(\gamma)\to M$ and 
$R^-(\gamma)\to M$ induce isomorphisms on the integer homology groups.
If 
\begin{equation*}
\SHI(M,\gamma,T)\cong \mathbb{C},
\end{equation*}
then $(M,\gamma,T)$ is diffeomorphic to a product sutured manifold with a product tangle.
\end{Theorem}
\begin{proof}
Let $(M_T,\gamma_T)$ be the sutured manifold defined as before.
Since $(M,\gamma)$ is a homology product and $T$ is vertical, by a straightforward argument using the Mayer-Vietoris sequence, $(M_T,\gamma_T)$ is also a homology product.
According to Lemma \ref{SHI-vertical-tangle}, we have
$$
\SHI(M_T,\gamma_T)\cong \SHI(M,\gamma,T)\cong \mathbb{C}.
$$
It then follows from \cite[Theorem 7.18]{KM:suture} that 
the sutured manifold $(M_T,\gamma_T)$ is a product sutured manifold, therefore $(M,\gamma,T)$ is diffeomorphic to a product sutured manifold with a product tangle.
\end{proof}

\begin{Theorem}\label{taut-non-vanishing}
Let $(M,\gamma,T)$ be a balanced sutured manifold with a balanced tangle.
If $(M,\gamma, T)$ is taut, then $\SHI(M,\gamma, T)\neq 0$.
\end{Theorem}
\begin{proof}
Let $(M_T,\gamma_T)$ be the sutured manifold as before.
By Lemma \ref{M_T-taut}, $(M_T,\gamma_T)$ is taut. Hence by \cite[Theorem 7.12]{KM:suture},
$$
\SHI(M_T,\gamma_T)\neq 0.
$$
Therefore in the case when $T$ is vertical, the theorem follows by Lemma \ref{SHI-vertical-tangle}.

In general, let 
$$
T=T_v\cup T_c=\bigcup_{1\le i \le m} t_i\cup \bigcup_{1\le i \le s} l_i
$$ 
be a decomposition of $T$, where $T_v=\bigcup_{1\le i \le m} t_i$ consists of the non-closed 
(hence vertical) components of $T$, and $T_c= \bigcup_{1\le i \le s} l_i$
consists of the closed components of $T$.  
We will continue using the notation from the proof of Lemma \ref{SHI-vertical-tangle}.
(the current $T_v$ corresponds to the tangle $T$ in Lemma \ref{SHI-vertical-tangle}). 

Let $(Y,L')$ be the resulting manifolds obtained by excision on the closure of $(M,\gamma,T_v)$ as given in \eqref{YL=YL'}.
The link $L'$ consists of $\bar{t}_1,\cdots,\bar{t}_m$
and $S^1\times \{p_1,\cdots,p_{m+2}\}$, where $p_1,\cdots, p_{m+2}$ are marked points on the auxiliary surface $F$.
Recall that $\bar R\subset Y$ is the image of the boundary surface, $u$ is an arc on $\bar{R}$ joining $S^1\times \{p_{m+1}\}$ and $S^1\times \{p_{m+2}\}$ and $\omega'$
is a collection of arcs on $\bar{R}$ joining the pairs 
$(\bar{t}_1,S^1\times \{p_1\}),\cdots, (\bar{t}_{m},S^1\times \{p_{m}\})$. 
It is proved in Lemma \ref{SHI-vertical-tangle} that
$$
\SHI(M,\gamma,T_v)\cong \II(Y,L',\omega'+u|\bar{R})\cong \SHI(M_{T_v},\gamma_{T_v}).
$$
Now we take the closed part $T_c$ into consideration. From the construction, it is clear
that $\bar{R}\subset Y$ is disjoint from $T_c$. Repeating the argument in the proof of Lemma \ref{SHI-vertical-tangle}, we have
\begin{equation}\label{MT=YL'+Tc}
\SHI(M,\gamma,T)\cong \II(Y,L'\cup T_c,\omega'+u|\bar{R}).
\end{equation}

By definition, an \emph{earring} on a closed component $l_i$ of $T$ is a pair $(e_i,u_i)$ where $e_i$ is a small meridian around $l_i$
and $u_i$ is a small arc joining $e_i$ and $l_i$. Let $(e_i,u_i)$ be an earing on $l_i$. 
Apply the unoriented skein exact triangle \cite[Proposition 6.11]{KM:Kh-unknot} to a crossing between $l_i$ and $e_i$, we have
an exact triangle
\begin{equation*}
\xymatrix{
  \II(Y,L'\cup T_c\cup e_i, \omega'+u+u_i|\bar{R}) \ar[r]^{} 
                &     \II(Y,L'\cup T_c,\omega'+u|\bar{R}) \ar[d]^{}    \\
                & \II(Y,L'\cup T_c,\omega'+u|\bar{R})   \ar[ul]_{}             }
\end{equation*}
Therefore $\II(Y,L'\cup T_c\cup e_i, \omega'+u+u_i|\bar{R})\neq 0$ implies 
$\II(Y,L'\cup T_c,\omega'+u|\bar{R})\neq 0$. Repeating this argument, we have:
\begin{equation}\label{YL'e-neq0}
\II(Y,L'\cup T_c\cup \bigcup_{1\le i \le s} e_i, \omega'+u+\sum_{1\le i \le s}u_i|\bar{R})\neq 0
\end{equation}
implies $\II(Y,L'\cup T_c,\omega'+u|\bar{R})\neq 0$.

As in the proof of Lemma \ref{SHI-vertical-tangle}, we apply excision to 
$$
(Y,L'\cup T_c\cup \bigcup_{1\le i \le s} e_i, \omega'+u+\sum_{1\le i \le s}u_i)
$$
along the pairs 
$(\partial N(S^1\times \{p_{m+1}\}), \partial N(S^1\times \{p_{m+2}\}))$ and 
$(\partial N(\bar{t}_i),\partial N(S^1\times \{p_{i}\}))$ ($1\le i \le m$).  
We further apply excisions along $\partial N(l_i)$ and $\partial N(e_i)$, for all $1\le i \le s$, 
using a diffeomorphism which maps the longitude of $l_i$ to the meridian of $e_i$.

The resulting triple is the disjoint union of an ``interesting'' component which is given by the image of 
$Y-\cup N(\bar{t}_i)-\cup N(S^1\times\{p_i\}) - \cup N(l_i)-\cup N(e_i)$,
and other admissible triples with 1-dimensional instanton Floer homology on the relevant eigenvalues. Let 
$$
(Y_2, \emptyset, \bar{\omega}'+\bar{u}+\sum_{1\le i \le s}\bar{u}_i)
$$ 
be the ``interesting'' part of the resulting triple after the excision.

In fact, if we take $[-1,1]\times A$ with $A$ an annulus, and attach it to $M_T$ by gluing $[-1,1]\times\partial A$ to the neighborhoods of the two meridian sutures $s^\pm(l_i)$
on $\partial M_T$, then the resulting manifold has two torus boundaries given by $\{\pm 1\}\times A\cup R^\pm_T(l_i)$.
According to \cite[Section 5.1]{KM:suture}, identifying the two tori $\bar{R}^\pm_T(l_i)$ by a suitable map gives the same $3$-manifold as doing the above excision along $\partial N(l_i)$ and $\partial N(e_i)$.
Therefore, the topology of $Y_2$ can be described by the following procedure:
\begin{itemize}
\item Let $G$ be the auxiliary surface given by $T_v$ described in the proof of Theorem \ref{SHI-vertical-tangle}.
 Attach $G$ to $(M_T,\gamma_T)$ along the sutures $\gamma\subset\partial M$ and $\gamma_{T_v}\subset \gamma_T$ and
      denote the top and bottom closure of $G$ by $\bar{G}^+$ and $\bar{G}^-$ respectively;
\item For each $i$, attach an auxiliary annulus $M_i$ ($1\le i \le s$) to $(M_T,\gamma_T)$  along the two meridian sutures $s^\pm(l_i)$ as above and
     denote the  top and bottom closure of $M_i$ by $\bar{R}_i^\pm$;
\item close up $M_T\cup [-1,1]\times G \cup [-1,1]\times M_i$ by diffeomorphisms from  
       $\bar{G}^+,\bar{R}_1^+,\cdots,\bar{R}_s^+$ to $\bar{G}^-,\bar{R}_1^-,\cdots,\bar{R}_s^-$ respectively.
\end{itemize}
Although neither the auxiliary surface $G\cup\bigcup_i M_i$ nor the closed-up surface
$\bar{R}_T:=\bar{G}^+\cup \bigcup_i \bar{R}_i^\pm$ is connected, the closure $Y_2$ still
defines the sutured Floer homology $\SHI(M_T,\gamma_T)$, as is explained in \cite[Section 2.3]{KM:Alexander}. Therefore we have
$$
\II(Y,L'\cup T_c\cup \bigcup_{1\le i \le s} e_i, \omega'+u+\sum_{1\le i \le s}u_i|\bar{R} )\cong 
\II(Y_2,\emptyset,\bar{\omega}'+\bar{u}+\sum_{1\le i \le s}\bar{u}_i|\bar{R}_T)\cong \SHI(M_T,\gamma_T).
$$
Since $\SHI(M_T,\gamma_T)\neq 0$, we have verified \eqref{YL'e-neq0}. Hence 
$\II(Y,L'\cup T_c,\omega'+u|\bar{R})\neq 0$ from the previous discussion, and the theorem follows from \eqref{MT=YL'+Tc}.
\end{proof}

\section{Applications to annular links}
\label{sec_applications}

Let $A:=  S^1\times [0,1]$ be an annulus, 
let $L$ be a link in the thickened annulus $ A\times [0,1]\cong S^1\times D^2$. In this case, $L$ is called an annular link.
The annular instanton Floer homology $\AHI(L)$ is defined in \cite{AHI} 
in the following way:
\begin{itemize}
\item Let $\mathcal{K}_2$ be the product link $S^1\times \{p_1,p_2\}$ in $S^1\times D^2$, let $u$ be an arc in $S^1\times D^2$ connecting $S^1\times\{p_1\}$ and $S^1\times\{p_2\}$;

\item View $A\times [0,1]$ as $S^1\times D^2$ and form the new link $L\cup \mathcal{K}_2$ in 
      \begin{equation*}
      S^1\times S^2=S^1\times D^2\cup_{S^1\times S^1} S^1\times D^2
      \end{equation*}
     where $L$ lies in the first copy of $S^1\times D^2$ and $\mathcal{K}_2$ lies in the second copy. 
\item Define
     \begin{equation*}
      \AHI(L):=\II(S^1\times S^2, L\cup \mathcal{K}_2,u).
     \end{equation*}
\end{itemize}
Recall that in this paper, all the Floer homology groups have coefficients in $\bC$. Also recall that by Proposition \ref{mu-pt=2} we have $\mu(\pt)=2\id$ on $\AHI(L)$.
The group $\AHI(L)$ is equipped with a $\mathbb{Z}$-grading called the f-grading. The degree-$i$ component $\AHI(L,i)$ of $\AHI(L)$ is defined 
to be the generalized eigenspace of $\muu(S^2)$ for the eigenvalue $i$. The f-grading is \emph{not} a lifting of the $\mathbb{Z}/4$-grading of the singular instanton Floer homology group $\II(S^1\times S^2, L\cup \mathcal{K}_2,u)$.

We will use $\mathcal{K}_n$ to denote the product link $S^1\times \{p_1,\cdots,p_n\}$ in $S^1\times D^2$, and use $\mathcal{U}_n$  to denote the unlink (i.e. a link that bounds a disjoint union of disks) in $S^1\times D^2$
with $n$ components.

\begin{Lemma}
\label{lem_nonvanishingOfAHI}
Suppose $L$ is an annular link that is included in a solid ball in $A\times[0,1]$, then 
$$
\dim \AHI(L)\ge 2^n,
$$
where $n$ is the number of components of $L$.
\end{Lemma}

\begin{proof}
By
\cite[Proposition 4.7]{AHI}, 
$$
\AHI(L)\cong \II^\sharp(L)
$$
where on the right hand side $L$ is viewed as a link in the $3$-ball, 
and $\II^\sharp$ is the knot invariant defined in \cite[Section 4.3]{KM:Kh-unknot}.
A local system $\Gamma$ on the configuration space $\mathcal{B}$ is introduced in \cite[Section 3]{KM-Ras}  
(see also \cite[Section 5]{KM:YAFT}). The instanton Floer homology with local coefficients $\II^\sharp(L;\Gamma)$ is 
an $\mathcal{R}:=\mathbb{C}[t,t^{-1}]$-module. It follows from 
\cite[Section 5.2]{KM:YAFT} that 
$\rank_{\mathcal{R}} \II^\sharp(L;\Gamma)$ is the same as 
$\rank_{\mathcal{R}} \II^\sharp(\mathcal{U}_n;\Gamma)=2^n$.
The homology $\II^\sharp(L)=\II^\sharp(L;\mathbb{C})$ can be computed from
$\II^\sharp(L;\Gamma)$ by the universal coefficient theorem where $\mathbb{C}$ is viewed as the $\mathcal{R}$-module
$\mathcal{R}/(t-1)$. Therefore we have
$$
\rank_\mathbb{C}\II^\sharp(L;\mathbb{C})\ge \rank_{\mathcal{R}} \II^\sharp(L;\Gamma)=2^n.
$$
\end{proof}

We now prove Theorem \ref{2g+n_intro} stated in the introduction. 
 Recall that a properly embedded, connected, oriented 
surface $S\subset S^1\times D^2$ is called a \emph{meridional surface} if $\partial S$ is a meridian of $S^1\times D^2$. The theorem shows that the annular instanton Floer homology detects the minimal Thurston norm among meridional surfaces. Let us repeat the statement of the theorem here.
\begin{Theorem}\label{2g+n}
Given an annular link $L$, suppose $S$ is a meridional surface that intersects $L$ transversely. Let $g$ be the genus of $S$, and let $n:=|S\cap L|$. Suppose $S$ minimizes the value of $2g+n$ among meridional surfaces,
then we have
\begin{equation*}
\AHI(L,i)= 0
\end{equation*}
when $|i|> 2g+n$ and
\begin{equation*}
\AHI(L,\pm(2g+n))\neq 0.
\end{equation*}
\end{Theorem}
\begin{proof}
Construct a closed surface $\overline{S}:=S\cup_{\partial S} D^2$ in $S^1\times S^2=S^1\times D^2\cup S^1\times D^2$, by attaching to $S$ an embedded disk bounded by $\partial S$
in the second copy of $S^1\times D^2$. Since $\overline{S}$ is homologous to $\{\pt\}\times S^2$ in $S^1\times S^2$, the operators 
$\muu(\{\pt\}\times S^2)$ and $\muu(\overline{S})$ are the same on 
 $\II(S^1\times S^2,L,\omega)$, for all choices of $\omega$.

We first prove that $\AHI(L,i)= 0$
when $|i|> 2g+n$. In fact,
if $|S\cap L|$ is odd, then the statement follows from  Proposition \ref{prop_allEigenvaluesForGeneral3mfld}. If $|S\cap L|$ is even,
torus excision gives an isomorphism
\begin{equation*}
\AHI(L)\otimes \mathbb{U}_{0,3}\cong \II(S^1\times S^2,L\cup \mathcal{K}_3,u)\otimes \mathbb{U}_{0,2},
\end{equation*}
and the statement follows from Proposition \ref{Lem-UgnWgn-eigenvalues} and Proposition \ref{prop_allEigenvaluesForGeneral3mfld}.

Now we prove $\AHI(L,\pm(2g+n))\neq 0$. We will only show $\AHI(L,2g+n)\neq 0$, since the other case is essentially the same.
If a component of $L$ intersects $S$ at points with different signs, 
then we can cancel a pair of consecutive intersection points with opposite signs of intersection by attaching a tube. This operation
increases $g$ by $1$ and decreases $n$ by 2, so the value of $2g+n$ does not change.
Therefore, without loss of generality, we may assume that each component of $L$ intersects $S$ with the same sign. 
 
Remove a tubular neighborhood of $S$ from $S^1\times D^2$. Let $\gamma$ be a meridian of $S^1\times D^2$ that is disjoint from $\partial S$. We obtain a balanced sutured manifold with a balanced tangle $T$ given by
$$
(M,\gamma,T):=(S^1\times D^2 - N(S),\gamma, L- N(S))
$$ 
with 
$$
\partial M=A(\gamma)\cup R^+\cup R^-:=A\cup S^+\cup S^-,
$$
where $S^+,S^-$ are the closures of the two parallel copies of $S$ in $\partial N(S)-\partial (S^1\times D^2)$, and $A$ is the closure of $\partial (S^1\times D^2)- N(S)$.
From the definition,
$$
\AHI(L,2g+n)\cong \SHI(M,\gamma,T).
$$

We discuss two cases.

{\bf Case 1. $M$ is $T$-irreducible.}
In this case, we show that $(M,\gamma,T)$ is taut, hence the result follows form Theorem \ref{taut-non-vanishing}.
To show that $S^+$ is $T$-norm minimizing, suppose there is a properly embedded surface 
$(F,\partial F)\subset (M,A)$ such that $[(F,\partial F)]=[(S^+,\partial S^+)]$ and its $T$-norm $x_T$ satisfies
\begin{equation}\label{xF<xS}
x_T(F)< x_T(S^+).
\end{equation}
By the same argument as Lemma \ref{lem_boundaryHasOneCircle}, $F$ can be chosen such that  $A\cap \partial F$ consists of one circle that is parallel to $\gamma$ and has the same orientation as $\gamma$. The component of $F$ that contains $\partial F$ then gives a meridional surface with a smaller value of $2g+n$, which contradicts the assumption on $S$.

To show that $S^+$  is T-incompressible.
Suppose $C\subset M- T$ is a $T$-compressing disk for $S^+$, such that $\partial C\subset S^+- T$
does not bound a disk in $S^+- T$. 
Compressing $S^+$ along $C$ then gives a meridional surface with a smaller value of $2g+n$, which is a contradiction. 

The same argument applies to $S^-$, hence $(M,\gamma,T)$ is taut.

{\bf Case 2. $M$ is $T$-reducible.}
In this case, the annular link $L$ can be written as the disjoint union of $L_1$ and $L_2$ with the following properties: 
\begin{itemize}
\item There exists a ball $B\in A\times [0,1]$, such that $B$ is disjoint from $L_1\cup S$, and $\emptyset\neq L_2\subset B$.
\item Let $T_1:=L_1-N(S)$, then $(M,\gamma)$ is $T_1$-irreducible.
\end{itemize}
By \cite[Proposition 4.3]{AHI},
$$
\AHI(L,2g+n) \cong \AHI(L_1,2g+n)\otimes \AHI(L_2).
$$
By the previous case and Lemma \ref{lem_nonvanishingOfAHI}, we have $\AHI(L,2g+n)\neq 0$.
\end{proof}

\begin{Corollary}\label{link-in-ball}
Suppose $L$ is an annular link with $\AHI(L)$ supported at the f-grading $0$. Then $L$ is included in a 3-ball in $S^1\times D^2$. 
\end{Corollary}

\begin{proof}
By Theorem \ref{2g+n}, there exists a meridional surface with genus zero that is disjoint from $L$. Therefore $L$ is included in a 3-ball.
\end{proof}

\begin{Corollary}\label{AHI-braid-detection}
Let $L$ be an annular link.
Then $L$ is isotopic to the closure of 
a braid with $n$ strands
if and only if the top f-grading of $\AHI(L)$ is $n$, and $\AHI(L,n)\cong \bC.$
\end{Corollary}
\begin{proof}
Suppose $L$ has components $L_1,\cdots,L_n$, we define 
the \emph{absolute winding number} of $L$ to be the sum of the absolute values 
of $[L_i]/[\mathcal{K}_1]$ for $1\le i\le n$, where $[L_i]$, $[\mathcal{K}_1]\in H_1(A\times [0,1];\bZ)$ are the fundamental classes.

 Suppose the top f-grading of $\AHI(L)$ is $n$ and $\AHI(L,n)\cong \bC.$ 
 We first want to show that the absolute winding number of ${L}$ is at least $n$. 
 By \cite[Section 4.4]{AHI}, the parity of $\dim_\bC\AHI(L,n)$ is invariant under crossing-change.
 If the absolute winding number of ${L}$ is smaller than $n$, then 
 we can apply a finite number of crossing changes on $L$ and obtain a link $L'$ which intersects a meridional disk $S$ at less than
 $n$ points. By Theorem \ref{2g+n}, this implies $\AHI(L',n)=0$. Hence $\dim_\bC \AHI(L,n)$ is even, which contradicts the assumption.

By Theorem \ref{2g+n}, there exists a meridional surface $S$ with genus $g$ that intersects ${L}$ at $m$ points such that $2g+m=n$. 
Since the absolute winding number of ${L}$ is at least $n$, we must have $g=0$ and $m=n$. 
Let $\gamma$ be a meridian of $S^1\times D^2$ that is disjoint from $\partial S$, we have
\begin{equation*}
\mathbb{C}\cong\AHI({L},n)\cong \SHI(S^1\times D^2-N(S), \gamma, {L}-N(S)).
\end{equation*}
The assumption on the absolute winding number also implies that ${L}-N(S)$ is vertical in $(S^1\times D^2-N(S), \gamma)$. 
Therefore by Theorem \ref{braid-detection}, $(S^1\times D^2-N(S), \gamma, {L}-N(S))$ is a product sutured manifold with a product tangle. Since the mapping class group of a punctured disk is the same as the braid group, it follows that $\overline{L}$ is a braid closure, hence $L$ is a braid closure. (See \cite[Corollary 3.9, (4.3)]{AHI} for more details on the last step.)

The other direction is included in \cite[Corollary 3.9, (4.3)]{AHI}.
\end{proof}

It is clear from the definition that 
$\AHI(\mathcal{K}_1)\cong \mathbb{U}_{0,3}$, which is 1-dimensional at
$f$-degrees $\pm 1$ and vanishes at all other degrees. By \cite[Proposition 4.3]{AHI}, we have
$\AHI(\mathcal{K}_n)\cong \AHI(\mathcal{K}_1)^{\otimes n}$.

\begin{Corollary}\label{knot-K1}
An annular knot $K$ is isotopic to $\mathcal{K}_1$ if and only if
\begin{equation*}
\AHI(K)\cong \AHI(\mathcal{K}_1)
\end{equation*} 
as graded vector spaces with respect to the f-grading.
\end{Corollary}
\begin{proof}
Suppose $\AHI(K)\cong \AHI(\mathcal{K}_1)$ as graded vector spaces. By Corollary \ref{AHI-braid-detection},
$K$ is isotopic to the closure of a braid with $1$ strand. Therefore $K$ is isotopic to $\mathcal{K}_1$ 
\end{proof}

The annular Khovanov homology for an oriented annular link is a triply-graded abelian group whose three gradings are called the 
h-grading, q-grading, and f-grading. We use $\AKh(L)$ to denote the annular Khovanov homology of
an oriented annular link $L$, and use $\AKh(L,i)$ to denote the component of $\AKh(L)$ with f-grading $i$.
In \cite{AHI}, a spectral sequence relating the annular Khovanov homology and annular instanton Floer homology is constructed.
\begin{Theorem}[{{\cite[Theorem 5.15]{AHI}}}] \label{s-sequence}
Let $L$ be an oriented annular link and $\overline{L}$ be its mirror image. For each $i\in \mathbb{Z}$,
there is a spectral sequence that has an $E_2$-page isomorphic to the annular Khovanov homology $\AKh(\overline{L},i;\bC)$ and converges to $\AHI(L,i)$. 
\end{Theorem}

Combine the above theorem and Theorem \ref{2g+n}, we obtain the following immediate corollary.
\begin{Corollary}\label{2g+n<AKh}
Let $L$ be an oriented annular link. If the top $f$-grading of $\AKh(L;\mathbb{Q})$ is $k$, then there is a meridional surface $S$ 
that intersects $L$ transversely and
$$
2g+n \le k
$$
where $g$ is the genus of $S$ and $n=|S\cap L|$. 
\end{Corollary}

We now prove Theorem \ref{thm_annular_intro} and Corollary \ref{cor_annular_intro} using Theorem \ref{2g+n} and Theorem \ref{s-sequence}. 
Let us repeat the statements from Section \ref{sec_intro} and combine them into the following statement.
\begin{Theorem}
Let $L$ be an oriented annular link. We have the following results.
\begin{itemize}
\item[(a)] $L$ is included in a 3-ball in $S^1\times D^2$ if and only if $\AKh(L;\bQ)$ is supported at the f-grading 0.

\item[(b)] Suppose $L$ has $n$ components, then $L$ is isotopic to the unlink $\mathcal{U}_n$ if and only if 
     $\AKh(L;\bZ/2)\cong\AKh(\mathcal{U}_n;\bZ/2)$ as triply-graded abelian groups.

\item[(c)] 
$L$ is isotopic to the closure of 
a braid with $n$ strands
if and only if the top f-grading of $\AKh(L;\bQ)$ is $n$, and $\AKh(L,n;\bQ)\cong \bQ.$

\item[(d)] $L$ is isotopic to $\mathcal{K}_n$ if and only if $\AKh(L;\bZ/2)\cong \AKh(\mathcal{K}_n;\bZ/2)$ as triply-graded abelian groups.
\end{itemize}
\end{Theorem}

\vspace{0.07in}

\begin{proof}
~
(a)
Suppose $\AKh(L;\bQ)$ is supported at the f-grading $0$, by Corollary \ref{2g+n<AKh}, there exists a meridional surface with genus zero that is disjoint from $L$. Therefore $L$ is included in a 3-ball. The other direction is clear from the definition of $\AKh$ in \cite[Section 2]{Rob}. 
\\

(b) Suppose $\AKh(L;\bZ/2)\cong\AKh(\mathcal{U}_n;\bZ/2)$ 
as triply-graded abelian groups. Then $\AKh(L;\bZ/2)$ is supported at the f-grading $0$. The universal coefficient theorem then implies that $\AKh(L;\bQ)$ is supported at the f-grading $0$, hence
by (a) we know $L$ is included in a 3-ball. In this case $\AKh(L,0)\cong \text{Kh}(L)$ as bi-graded  abelian groups by the h-grading and q-grading. 
Since it is known that the bi-graded Khovanov homology with $\bZ/2$-coefficients
detects the unlink \cite[Theorem 1.3]{Kh-unlink}, $L$ must be the unlink with $k$ components.
The other direction is obvious.
\\

(c) Suppose the top f-grading of $\AKh(L;\bQ)$ is $n$ and $\AKh(L,n;\bQ)\cong \bQ.$


By the universal coefficient theorem and Theorem \ref{s-sequence}, we have $\AHI(\overline{L},i)=0$ for all $i>n$ 
and $\dim_\bC \AHI(\overline{L},n)\le 1$. Since $\AKh(L,n;\bC)$ is 1-dimensional, 
it cannot collapse anymore in the spectral sequence. Therefore $\dim_\bC \AHI(\overline{L},n)=1$. By Corollary \ref{AHI-braid-detection},
$\overline{L}$ is the closure of a braid with $n$ strands. So does $L$.


The other direction follows from \cite[Proposition 2.5]{GN-braid}. Although \cite[Proposition 2.5]{GN-braid} 
is stated for $\bZ/2$-coefficients, the proof of this proposition works for arbitrary fields. 
\\

(d) Suppose $\AKh(L;\bZ/2)\cong \AKh(\mathcal{K}_n;\bZ/2)$. 
We have $\AKh(L,i;\bZ/2)=0$ for $i>n$ and $\rank\AKh(L,n;\bZ/2)=1$. By the universal coefficient theorem, we have  
 $\rank\AKh(L,i;\bQ)=0$ for $i>n$ and $\rank\AKh(L,n;\bQ)=1$. 
Hence $L$ is the closure of a braid with $n$ strands by Part (c).
By \cite[Theorem 3.1]{BG-braid}, 
the triply-graded annular Khovanov homology distinguishes the trivial braid closure among braid closures. Therefore the result is proved.
The other direction is obvious.
\end{proof}


\bibliographystyle{amsalpha}
\bibliography{references}

\end{document}